\pdfoutput=1
\documentclass[10pt, a4paper, notitlepage]{article}

\usepackage{biblatex}
\usepackage{bbm}

\usepackage{etoolbox}
\usepackage{tikz}
\usetikzlibrary{calc}
\usetikzlibrary{cd}
\usetikzlibrary{decorations.markings}
\usetikzlibrary{decorations.pathreplacing}
\usetikzlibrary{decorations.pathmorphing}
\usetikzlibrary{decorations.text}
\usetikzlibrary{arrows.meta}
\usetikzlibrary{arrows}
\usetikzlibrary{positioning}

\usepackage{geometry} 
\usepackage{tocloft} 
\usepackage{fancyhdr} 
\usepackage{longtable}
\usepackage{subcaption}
\usepackage{enumitem}
\usepackage{amssymb}
\usepackage{amsmath}
\usepackage{stackrel} 
\usepackage{uniinput}
\usepackage{amsthm}
\usepackage{stmaryrd}
\usepackage{mathtools}
\usepackage{fouridx}
\usepackage{xparse}
\usepackage{bm}
\usepackage{aliascnt}
\usepackage{import}

\usepackage{hyperref}
\theoremstyle{definition}
\newtheorem{theorem}{Theorem}
\let\emph\textbf

\newaliascnt{remark}{theorem}
\newaliascnt{definition}{theorem}
\newaliascnt{proposition}{theorem}
\newaliascnt{lemma}{theorem}
\newaliascnt{corollary}{theorem}
\newaliascnt{example}{theorem}
\newaliascnt{convention}{theorem}
\newaliascnt{todo}{theorem}

\newtheorem{remark}[remark]{Remark}
\newtheorem{definition}[definition]{Definition}
\newtheorem{proposition}[proposition]{Proposition}
\newtheorem{lemma}[lemma]{Lemma}

\newtheorem{example}[example]{Example}

\aliascntresetthe{remark}
\aliascntresetthe{definition}
\aliascntresetthe{proposition}
\aliascntresetthe{lemma}
\aliascntresetthe{corollary}
\aliascntresetthe{example}
\aliascntresetthe{convention}
\aliascntresetthe{todo}

\numberwithin{equation}{section}
\numberwithin{figure}{section}
\numberwithin{table}{section}
\makeatletter\let\c@table\c@figure\makeatother

\numberwithin{theorem}{section}
\numberwithin{remark}{section}
\numberwithin{definition}{section}
\numberwithin{proposition}{section}
\numberwithin{lemma}{section}
\numberwithin{corollary}{section}
\numberwithin{example}{section}
\numberwithin{convention}{section}
\numberwithin{todo}{section}

\DeclareFieldFormat*{url}{}
\DeclareFieldFormat*{doi}{}
\DeclareFieldFormat*{issn}{}
\DeclareFieldFormat[misc]{url}{\mkbibacro{URL}\addcolon\space\url{#1}}
\DeclareFieldFormat[online]{url}{\mkbibacro{URL}\addcolon\space\url{#1}}

\setlist{topsep=0.5em, itemsep=0em}
\renewcommand{\arraystretch}{1.3}

\SetLabelAlign{parleftalign}{\parbox[t]\textwidth{#1\par\mbox{}}}

\newcommand{\id}{\mathrm{id}}

\newcommand{\Hom}{\operatorname{Hom}}
\newcommand{\Ker}{\operatorname{Ker}}

\newcommand{\cat}[1]{\mathcal{#1}}

\newcommand{\Gtl}{\operatorname{Gtl}}

\newcommand{\HH}{\operatorname{HH}}
\newcommand{\HC}{\operatorname{HC}}

\newcommand{\odd}{\operatorname{odd}}
\newcommand{\even}{\operatorname{even}}

\def\Im{\operatorname{Im}}

\newcommand{\running}{~|~}

\newcommand{\sporadic}{\mathbb{P}_0}

\NewDocumentCommand{\tensor}{t_}
 {%
  \IfBooleanTF{#1}
   {\tensorop}
   {\otimes}%
 }
\NewDocumentCommand{\tensorop}{m}
 {%
  \mathbin{\mathop{\otimes}\displaylimits_{#1}}%
 }

\newcommand{\angletop}{\mathrm{top}}
\newcommand{\anglemid}{\mathrm{mid}}
\newcommand{\anglebot}{\mathrm{bot}}
\newcommand{\anglefirst}{\mathrm{first}}
\newcommand{\anglenext}{\mathrm{next}}
\newcommand{\Sporadic}{\mathsf{S}}

\newcommand{\paperone}{\cite{Paper-I}}
\newcommand{\papertwoA}{\cite{Paper-IIA}}
\newcommand{\cA}{\mathcal{A}}

\title{Explicit Hochschild Cocycles for Gentle Algebras}
\author{Jasper van de Kreeke}
\date{15 March 2023}

\begin{document}

\maketitle

\newcommand{\foobarfoo}{Test}

\begin{abstract}
Hochschild cohomology is crucial for understanding deformation theory. In \paperone, we have computed the Hochschild cohomology for gentle algebras of punctured surfaces. The construction of that paper is rather implicit and fails if the punctured surface has only a single puncture. In the present note, we supplement the earlier method by providing an explicit construction of Hochschild cocycles which also succeeds in the case of a single puncture.
\end{abstract}

\tableofcontents

\section{Introduction}
Hochschild cohomology is a crucial invariant for understanding an object's deformation theory. Gentle algebras are discrete models for Fukaya categories of punctured surfaces. We have computed the Hochschild cohomology of gentle algebras in \paperone\ under a technical restriction regarding the surface. In this note, we go beyond the restriction. More precisely, we extend the computation of \paperone\ to include all gentle algebras of punctured surfaces. The idea is to write down explicit Hochschild cocycles, instead of constructing them implicitly as in \paperone.

The findings from \paperone\ can be summarized as follows:
\begin{description}
\item[Generation criterion:] It determines for a given collection of cocycles with certain prescribed shape in low adicity whether it concerns a basis of Hochschild cohomology or not. We recall the generation criterion in \autoref{sec:existing-crit}.
\item[Odd cocycles:] We provided an explicit collection of odd cocycles. This family satisfies the requirements of the generation criterion and therefore forms a basis for odd Hochschild cohomology. We recall the odd cocycles in \autoref{sec:existing-odd}.
\item[Sporadic even cocycles:] We provided a semi-explicit collection of even cocycles. In the present note, we refer to them as the \emph{sporadic} cocycles. We recall the sporadic cocycles in \autoref{sec:existing-sporadic}.
\item[Ordinary even cocycles:] We provided an implicit collection of even cocycles under the assumption that the arc system satisfies the [NL2] condition. In the present paper, we refer to these cocycles as \emph{ordinary} even cocycles. The sporadic and ordinary even cocycles together satisfy the generation criterion. Under the condition [NL2], they form a basis for even Hochschild cohomology.
\end{description}

The problem with \paperone\ is the requirement of the [NL2] condition. In fact, the ordinary even cocycles were constructed in \paperone\ as cup products of carefully selected sporadic and odd classes. The sporadic classes required for this construction only exist if the [NL2] condition holds. Without the [NL2] condition, the construction via the cup product fails, leaving us without proof of existence of the ordinary even cocycles. The aim of the present note is to circumvent the [NL2] condition. The idea is to provide an explicit construction for the ordinary even cocycles. The amount of explicitness makes the construction independent of the [NL2] condition. A drawback is that long calculations are required to check that the ordinary even classes actually satisfy the cocycle condition. This is the reason we cut the present note from \paperone.

This note is organized as follows: In \autoref{sec:4prelim}, we recall $ A_∞ $-categories, Hochschild cohomology and gentle algebras. In \autoref{sec:existing}, we recall the generation criterion, the odd cocycles and the sporadic even cocycles from \paperone. In \autoref{sec:even}, we construct the ordinary even cocycles and perform detailed checks that they indeed satisfy the Hochschild cocycle condition. We summarize the findings in \autoref{sec:even-summary}.

\section{Preliminaries}
\label{sec:4prelim}
In this section, we recall $ A_∞ $-categories, Hochschild cohomology and gentle algebras. A more extensive introduction can be found in \paperone\ or \papertwoA.

\subsection{$ A_∞ $-categories}
In this section we briefly recall $ A_∞ $-categories.

\begin{definition}
A ($ ℤ $- or $ ℤ/2ℤ $-graded, strictly unital) \emph{$ A_∞ $-category} $ \cat C $ (over e.g.~$ ℂ $) consists of a collection of objects together with $ ℤ $- or $ ℤ/2ℤ $-graded hom spaces $ \Hom(X, Y) $, distinguished identity morphisms $ \id_X ∈ \Hom^0 (X, X) $ for all $ X ∈ \cat C $, together with multilinear higher products
\begin{equation*}
μ^k: \Hom(X_k, X_{k+1}) ¤ … ¤ \Hom(X_1, X_2) → \Hom(X_1, X_{k+1}), \quad k ≥ 1
\end{equation*}
of degree $ 2-k $ such that the $ A_∞ $-relations and strict unitality axioms hold: For every compatible morphisms $ a_1, …, a_k $ we have
\begin{align*}
& \sum_{0 ≤ n < m ≤ k} (-1)^{‖a_n‖ + … + ‖a_1‖} μ(a_k, …, μ(a_m, …, a_{n+1}), a_n, …, a_1) = 0, \\
& μ^2 (a, \id_X) = a, ~ μ^2 (\id_Y, a) = (-1)^{|a|} a, ~ μ^{≥3} (…, \id_X, …) = 0.
\end{align*}
The symbol $ ‖a‖ = |a| - 1 $ denotes the reduced degree of $ a $.
\end{definition}

\subsection{The Hochschild DGLA}
\label{sec:4prelim-hochschild}
In this section, we recall Hochschild cohomology for $ A_∞ $-categories. First, we recall DG Lie algebras (DGLAs). Second, we recall the Hochschild complex for $ A_∞ $-categories together with its DGLA structure. Third, we comment on the cup product.

\begin{definition}
A \emph{DG Lie algebra} (DGLA) is a $ ℤ $- or $ ℤ/2ℤ $-graded vector space $ L $ together with a differential $ d: L^i → L^{i+1} $ and a bracket $ [-, -]: L × L → L $ satisfying the Leibniz rule and the Jacobi identity:
\begin{align*}
[a, b] &= (-1)^{|a||b| + 1} [b, a], \\
d(d(a)) &= 0, \\
d([a, b]) &= [da, b] + (-1)^{|a|} [a, db], \\
0 &= (-1)^{|a||c|} [a, [b, c]] + (-1)^{|b||a|} [b, [c, a]] + (-1)^{|c||b|} [c, [a, b]].
\end{align*}
\end{definition}

Hochschild cohomology has historically been defined for ordinary associative algebras. The Hochschild complex carries a natural DGLA structure. In more modern times, Hochschild cohomology together with the DGLA structure has been extended to the case of $ A_∞ $-categories. We recall this construction as follows:

\begin{definition}
Let $ \cat C $ be a $ ℤ $- or $ ℤ/2ℤ $-graded $ A_∞ $-category. Then its Hochschild complex $ \HC(\cat C) $ is given by the graded vector space
\begin{equation*}
\HC(\cat C) = \prod_{\substack{X_1, …, X_{k+1} ∈ \cat C \\ k ≥ 0}} \Hom\big(\Hom_{\cat C} (X_k, X_{k+1})[-1] \tensor … \tensor \Hom_{\cat C} (X_1, X_2)[-1], \Hom_{\cat C} (X_1, X_{k+1})[-1]\big).
\end{equation*}
For $ η, ω ∈ \HC(\cat C) $, temporarily denote by $ μ · ω ∈ \HC(\cat C) $ the Gerstenhaber product given by
\begin{equation*}
(η · ω) (a_k, …, a_1) = \sum (-1)^{(∥a_l∥ + … + \Vert a_1 \Vert)∥ω∥} η(a_k, …, ω(…), a_l, …, a_1).
\end{equation*}
Define a $ ℤ $- or $ ℤ/2ℤ $-graded DGLA structure on $ \HC(\cat C) $ as follows: Its grading $ ‖·‖ $ is the one induced from the shifted degrees of the hom spaces of $ \cat C $. In other words, we have the equation
\begin{equation*}
‖η(a_k, …, a_1)‖ = ‖η‖ + ‖a_k‖ + … + ‖a_1‖, \quad η ∈ \HC(\cat C).
\end{equation*}
The bracket on $ \HC(\cat C) $ is the Gerstenhaber bracket
\begin{equation*}
[η, ω] = η · ω - (-1)^{‖ω‖ ‖η‖} ω · η.
\end{equation*}
Its differential is given by commuting with the product $ μ_{\cat C} ∈ \HC^1 (\cat C) $:
\begin{equation*}
dν = [μ_{\cat C}, ν].
\end{equation*}
\end{definition}

\begin{remark}
Let $ ν ∈ \HC^1 (\cat C) $. Then $ μ_{\cat C} + εν $ is an infinitesimal (curved $ A_∞ $-)deformation of $ μ_{\cat C} $ over the local ring $ B = ℂ[ε] / (ε^2) $ if and only if $ dν = 0 $. More precisely, Hochschild cohomology $ \HH^1 (\cat C) $ classifies infinitesimal deformations of $ \cat C $ up to gauge equivalence.
\end{remark}

\begin{remark}
In case $ \cat C $ is only $ ℤ/2ℤ $-graded, the Hochschild DGLA is only a $ ℤ/2ℤ $-graded DGLA and Hochschild cohomology is only a $ ℤ/2ℤ $-graded vector space.
\end{remark}

\begin{remark}
The DGLA structure on $ \HC(\cat C) $ induces (noncanonically) the structure of an $ L_∞ $-algebra on Hochschild cohomology $ \HH(\cat C) $.
\end{remark}

\begin{remark}
For ordinary algebras, which are concentrated in degree zero and have vanishing higher products, the Hochschild cohomology is typically defined without the shifts. This results in a grading difference of $ 1 $ from what we present here. For example, the center of the algebra is the classical zeroth Hochschild cohomology. In our $ A_∞ $-setting, this cohomology will rather be found in degree $ -1 $.
\end{remark}

There is also a second product on $ \HC(\cat C) $: the cup product.

\begin{definition}
The \emph{cup product} on $ \HC(\cat C) $ is given by
\begin{equation*}
\begin{split}
(\kappa ∪ \nu) (a_r,\dots a_1) := \sum_{0\le i \le j \le u \le v \le r}
  (-1)^{\maltese}
 \mu(a_r,\dots,\kappa(a_v,\dots,a_{u+1}),\dots,\nu(a_j,\dots,a_{i+1}),\dots,a_1)
\end{split}
\end{equation*}
with $\maltese = {(‖a_1‖ + … + ‖a_u‖)‖\kappa‖ + (‖a_1‖ + … + ‖a_i‖)‖\nu‖ + ‖ν‖ + 1}$.
\end{definition}

\subsection{The gentle algebra and its deformation}
\label{sec:4prelim-gtl}
In this section, we recall $ A_∞ $-gentle algebras from \cite{Bocklandt}.

\begin{definition}
A \emph{punctured surface} is a closed oriented surface $ S $ with a finite set of punctures $ M ⊂ S $. We assume that $ |M| ≥ 1 $, or $ |M| ≥ 3 $ if $ S $ is a sphere.
\end{definition}

\begin{definition}
Let $ (S, M) $ be a punctured surface. An \emph{arc} in $ S $ is a not necessarily closed curve $ γ: [0, 1] → S $. A \emph{loop} is an arc $ γ $ with $ γ(0) = γ(1) $. An \emph{arc system} $ \cA $ on $ (S, M) $ is a finite collection of arcs such that the arcs in $ \cA $ meet only at the set $ M $ of punctures. Intersections and self-intersections are not allowed.

An arc system $ \cA $ is \emph{full} if the arcs cut the surface into contractible pieces, which we call \emph{polygons}. The arc system $ \cA $ satisfies the condition \emph{[NMD]} if:
\begin{itemize}
\item Any two arcs in $ \cA $ are non-homotopic in $ S \setminus M $.
\item All loops in $ \cA $ are non-contractible in $ S \setminus M $.
\end{itemize}
\end{definition}

Let us now recall the construction of the gentle algebra $ \Gtl \cA $. It is an $ A_∞ $-category and we shall start by describing its objects, differential and product. After that, we will recall the higher products on $ \Gtl \cA $.

\begin{definition}
Let $ \cA $ be a full arc system with [NMD]. Then the \emph{gentle algebra} $ \Gtl \cA $ is the $ A_∞ $-category defines as follows:
\begin{itemize}
\item Its objects are the arcs $ a ∈ \cA $.
\item A basis for the hom space $ \Hom_{\Gtl \cA} (a, b) $ is given by the set of all angles around punctures from $ a $ to $ b $.
\item The $ ℤ/2ℤ $-grading on $ \Gtl \cA $ is given by declaring all indecomposable angles to have odd degree.
\item The differential $ μ^1 $ is zero.
\item The product $ μ^2 $ is defined as a signed version of the ordinary product of $ \Gtl \cA $:
\begin{equation*}
μ^1 ≔ 0, \quad μ^2 (α, β) ≔ (-1)^{|β|} αβ.
\end{equation*}
\end{itemize}
The angles include \emph{empty angles}, which are the \emph{identities} on the arcs. A non-empty angle is \emph{indecomposable} if it cannot be written as $ αβ $ where $ α, β $ are non-empty angles. The higher products are defined in \autoref{def:4prelim-gtl-higher}.
\end{definition}

\begin{remark}
The hom spaces of $ \Gtl \cA $ are not finite-dimensional, in contrast to what is classically called a gentle algebra.
\end{remark}

The ordinary product $ αβ $ still means concatenation of angles, and we will keep this notation. We now recall the higher products $ μ^{≥3} $ of $ \Gtl \cA $. They capture the topology of the arcs and angles. Roughly speaking, a higher product of a sequence of angles is nonzero if the sequence bounds a disk. Such a disk may either be a polygon, or a sequence of polygons stitched together, known as an immersed disk. Let us make this precise:

\begin{definition}
An \emph{immersed disk} consists of an oriented immersion of a standard polygon $ P_k $ into the surface, such that
\begin{itemize}
\item Every edge of $ P_k $ is mapped to an arc.
\item The immersion does not cover any punctures.
\end{itemize}
A sequence of angles $ α_1, …, α_k $ is a \emph{disk sequence} if there exists an immersed disk such that $ α_1, …, α_k $ are the interior angles of the immersed disk, counting in clockwise order.
\end{definition}

\begin{remark}
A disk sequence $ α_1, …, α_k $ always has length $ k ≥ 3 $ because all polygons in the arc system $ \cA $ have at least three corners. Simply speaking, there are no digons.
\end{remark}

We can now describe the higher products $ μ^{≥3} $ on $ \Gtl \cA $ as follows:

\begin{definition}
\label{def:4prelim-gtl-higher}
Let $ α_1, …, α_k $ be a disk sequence. Let $ β $ be an angle composable with $ α_1 $ in the sense that $ β α_1 ≠ 0 $, and let $ γ $ be an angle post-composable with $ α_k $ in the sense that $ α_k γ ≠ 0 $. Then we define higher products
\begin{equation*}
μ^k (β α_k, …, α_1) ≔ β, \quad μ^k (α_k, …, α_1 γ) ≔ (-1)^{|γ|} γ.
\end{equation*}
The higher products vanish on all angle sequences other than these. If $ α_1, …, α_k $ is a disk sequence, we call the sequence $ α_1, …, β α_k $ \emph{final-out} if $ β ≠ \id $ and the sequence $ α_1 γ, …, α_k $ \emph{first-out} if $ γ ≠ \id $. We call either of them \emph{all-in} if $ β $ and $ γ $ are merely identities.
\end{definition}

\begin{remark}
In \paperone, we have imposed the additional condition [NL2] on arc systems. The condition entails that $ \cA $ contains no loops and no two arcs share more than one endpoint. In particular, the [NL2] condition requires that the number of punctures $ |M| $ in the surface is at least two. We do not require the [NL2] condition in the present note.
\end{remark}

\section{Previous calculations}
\label{sec:existing}
In this section, we summarize findings on Hochschild cohomology from \paperone. We divide the section into three parts: In \autoref{sec:existing-crit}, we recall the generation criterion from \paperone. The criterion determines for a given collection of cocycles with certain prescribed shape in low adicity whether it concerns a basis of Hochschild cohomology or not. In \autoref{sec:existing-odd}, we recall a certain collection of odd cocycles. This family satisfies the requirements of the generation criterion and therefore forms a basis for odd Hochschild cohomology. In \autoref{sec:existing-sporadic}, we recall a certain collection of even cocycles, the \emph{sporadic} even cocycles.

\subsection{The generation criterion}
\label{sec:existing-crit}
In this section, we recall the two generation criteria for odd and even Hochschild cohomology of $ \Gtl \cA $ from \paperone. The two generation criteria hold for all full arc systems of punctured surfaces and are not restricted to the assumption [NL2]. In \autoref{sec:existing-odd}, we use the generation criterion for odd Hochschild cohomology to explain the basis for odd Hochschild cohomology we obtained in \paperone. In \autoref{sec:even}, we use the generation criterion for even Hochschild cohomology to construct and verify an explicit basis for even Hochschild cohomology.

We use the notation $ ℓ_m $ to denote a full turn around the puncture $ m $. The meaning is depicted in \autoref{fig:existing-odd-lm}:

\begin{definition}
Whenever $ m ∈ M $ is a puncture, the letter $ ℓ_m $ denotes the sum of full turns around the puncture $ m $, starting from every incident arc. Every loop incident at $ m $ gives rise to two contributions to $ ℓ_m $. The element $ ℓ_m $ is a formal sum of endomorphisms of the arcs incident at $ m $. In other words, $ ℓ_m $ can be interpreted as a cochain $ ℓ_m ∈ \HC(\Gtl \cA) $ of arity 0. When $ r ≥ 1 $ is a natural number, the expression $ ℓ_m^r $ denotes the $ r $-th power of $ ℓ_m $, equally consisting of endomorphisms of the arcs incident at $ m $.
\end{definition}

\begin{proposition}[\paperone]
\label{th:existing-generation-odd}
Let $ \cA $ be a full arc system with [NMD]. Let $ ν_{\id} $ and $ \{ν_{m, r}\}_{m ∈ M, r ≥ 1} $ be odd Hochschild cocycles. Assume the following conditions:
\begin{itemize}
\item $ ν_{\id}^0 = \sum_{a ∈ \cA} \id_a $.
\item $ ν_{m, r}^0 = ℓ_m^r $.
\end{itemize}
Then the collection of $ \{ν_{\id}\} ∪ \{ν_{m, r}\}_{m ∈ M, r ≥ 1} $ is a basis for $ \HH^{\odd} (\Gtl \cA) $.
\end{proposition}

In \autoref{def:existing-generation-sporadic}, we fix notation for a certain class $ \Sporadic $ of even Hochschild cocycles which merely “scale angles”. The idea is that the quotient $ \Sporadic / [\id, -] $ is the 1-adic part of Hochschild cohomology which merely “scales angles”. The precise definition is as follows:

\begin{definition}
\label{def:existing-generation-sporadic}
Let $ \cA $ be a full arc system with [NMD]. Denote by $ \Sporadic ⊂ \HC^{\even} (\Gtl \cA) $ the space of all even cochains $ ν $ such that $ ν^{≠1} = 0 $, $ dν = 0 $, and $ ν(α) = λ_α α $ for some scalar $ λ_α $ for every indecomposable angle $ α $. Denote by $ [\id, -] ⊂ \Sporadic $ the subspace spanned by all Gerstenhaber commutators $ [\id_a, -] ∈ \Sporadic $ ranging over $ a ∈ \cA $.
\end{definition}

\begin{proposition}[\paperone]
\label{th:existing-generation-even}
Let $ \cA $ be a full arc system with [NMD]. Let $ \{ν_P\}_{P ∈ \sporadic} $ be a collection of even Hochschild cocycles, indexed by some set $ \sporadic $. Assume the following conditions:
\begin{itemize}
\item $ ν_P ∈ \Sporadic $.
\item The collection $ \{ν_P\}_{P ∈ \sporadic} $ is a basis for $ \Sporadic / [\id, -] $.
\end{itemize}
Let $ \{ν_{m, r}\}_{m ∈ M, r ≥ 1} $ be another collection of even Hochschild cocycles. Assume the following conditions:
\begin{itemize}
\item $ ν_{m, r}^0 = 0 $,
\item $ ν_{m, r}^1 (α) = ℓ_m^r α $ for indecomposable angles $ α $ winding around $ m $.
\item $ ν_{m, r}^1 (α) = 0 $ for indecomposable angles $ α $ not winding around $ m $.
\end{itemize}
Then the two collections $ \{ν_P\}_{P ∈ \sporadic} $ and $ \{ν_{m, r}\}_{m ∈ M, r ≥ 1} $ together form a basis for $ \HH^{\even} (\Gtl \cA) $.
\end{proposition}

The generation criteria build on the following technical lemma:

\begin{lemma}[\paperone]
\label{th:existing-crit-gauging}
Let $ \cA $ be a full arc system with [NMD]. Then:
\begin{itemize}
\item A cochain $ ν ∈ \HC^{\odd} (\Gtl \cA) $ with $ dν = 0 $ and $ ν^0 = 0 $ satisfies $ ν ∈ \Im(d) $.
\item A cochain $ ν ∈ \HC^{\even} (\Gtl \cA) $ with $ dν = 0 $ and $ ν^1 = 0 $ satisfies $ ν ∈ \Im(d) $.
\end{itemize}
\end{lemma}

\subsection{Odd Hochschild cocycles}
\label{sec:existing-odd}
In this section, we recall odd Hochschild cohomology of gentle algebras from \paperone. The idea to find Hochschild cocycles of $ \Gtl \cA $ is to trace Seidel's explanation \cite{Seidel-relative} on deformations of Fukaya categories. Seidel's proposal is to define the higher products relative to a divisor. In \paperone, we translated this idea to gentle algebras of punctured surfaces. In particular, we describe in \paperone\ a basis of the odd Hochschild cohomology. In the present section, we recall this basis.

\begin{example}
We will define the odd Hochschild cocycles $ ν_{m, r} $ by their behavior on orbigons of type $ (m, r) $. If $ r = 1 $, then orbigons of type $ (m, r) $ can be interpreted in a more standard way. In fact, they are the same as immersed disks covering the puncture $ m $ precisely once, and no other punctures apart from $ m $. More precisely, we may say the sequence $ α_1, …, α_k $ of angles is an immersed disk covering the puncture $ m ∈ M $ if there is an immersion of the standard polygon $ P_k $ into $ S $ such that every edge is mapped to an arc and the immersion covers a single puncture, and only once, namely $ m $. An example of an immersed disk $ α_1, …, α_k $ covering a puncture is depicted in \autoref{fig:existing-odd-covering}.
\end{example}

This basis is best recalled as follows:

\begin{definition}
\label{def:existing-odd-def}
Let $ m ∈ M $ be a puncture and $ r ≥ 1 $ a natural number. Then the Hochschild cocycle $ ν_{m, r} $ is defined by
\begin{itemize}
\item The 0-adic component $ ν^0 $ is $ ℓ_m^r $.
\item The 1-adic component $ ν^1 $ vanishes.
\item The 2-adic component $ ν^2 $ vanishes.
\item Assume that $ α_1, …, α_i^{(1)}, ℓ_m^r, α_i^{(2)}, …, α_k $ is a disk sequence. Put $ α_i = α^{(2)} α^{(1)} $. Let $ β $ be an angle composable with $ α_1 $ in the sense that $ β α_1 ≠ 0 $, and let $ γ $ be an angle post-composable with $ α_k $ in the sense that $ α_k γ ≠ 0 $. Then put
\begin{equation*}
ν^k (β α_k, …, α_1) = β, \quad ν^k (α_k, …, α_1 γ) = (-1)^{|γ|} γ.
\end{equation*}
The higher products $ ν^{≥3} $ vanish on all angle sequences other than these.
\end{itemize}
The single Hochschild cocycle $ ν_{\id} $ is given by $ ν_{\id}^0 = \sum_{a ∈ \cA} \id_a $ and $ ν_{\id}^{>0} = 0 $.
\end{definition}

\begin{figure}
\centering
\begin{subfigure}{0.3\linewidth}
\centering
\begin{tikzpicture}
\path[draw] (0, 0) -- ++(right:1) coordinate[pos=0.3] (1-start) coordinate[pos=0.7] (6-end) -- ++(60:1) coordinate[pos=0.3] (6-start) coordinate[pos=0.7] (5-end) -- ++(120:1) coordinate[pos=0.3] (5-start) coordinate[pos=0.7] (4-end) -- ++(left:1) coordinate[pos=0.3] (4-start) coordinate[pos=0.7] (3-end) -- ++(240:1) coordinate[pos=0.3] (3-start) coordinate[pos=0.7] (2-end) -- ++(300:1) coordinate[pos=0.3] (2-start) coordinate[pos=0.7] (1-end);
\foreach \i in {1, 2, 3, 4, 5, 6} {
\path[draw, ->, bend right=60] (\i-start) to (\i-end);
\path (\i-start) -- (\i-end) node[midway] {$ \i $};
};
\path[fill] (0.5, 0.866) circle[radius=0.05];
\end{tikzpicture}
\caption{Disk covering a puncture}
\label{fig:existing-odd-covering}
\end{subfigure}
\begin{subfigure}{0.3\linewidth}
\centering
\begin{tikzpicture}
\path[draw] (0, 0) -- ++(up:1);
\path[draw] (0, 0) -- ++(left:1);
\path[draw] (0, 0) -- ++(right:1);
\path[draw] (0, 0) -- ++(down:1);
\path[draw, <-] (90:0.2) arc(90:-270:0.2);
\path[draw, <-] (180:0.4) arc(180:-180:0.4);
\path[draw, <-] (270:0.6) arc(270:-90:0.6);
\path[draw, <-] (0:0.8) arc(0:-360:0.8);
\path (0, 0) node {\small $ m $};
\end{tikzpicture}
\caption{The element $ ℓ_m $}
\label{fig:existing-odd-lm}
\end{subfigure}
\caption{}
\end{figure}

In terms of orbigons, the assumption in \autoref{def:existing-odd-def} reads that $ α_1, …, α_k $ is the reduced sequence of an orbigon of type $ (m, r) $.

\begin{lemma}[\paperone]
The cochains $ ν_{\id} $ and $ ν_{m, r} $ are Hochschild cocycles.
\end{lemma}

\begin{theorem}[\paperone]
The collection of $ ν_{\id} $ and $ (ν_{m, r})_{m ∈ M, r ≥ 1} $ provides a basis for $ \HH^{\odd} (\Gtl \cA) $.
\end{theorem}

\begin{proof}
According to the generation criterion \autoref{th:existing-generation-odd}, a collection of cocycles is a basis if it has the right 0-ary components. This is clearly the case.
\end{proof}

\subsection{Sporadic even cocycles}
\label{sec:existing-sporadic}
In this section, we recall a first type of even Hochschild cocycles, the \emph{sporadic classes}. The idea is to select just 1-adic cochains $ ν = ν^1 $ with $ ν^1 (α) $ being a multiple of $ α $ whenever $ α $ is an indecomposable angle. More precisely, we pick a collection $ (ν_P)_{P ∈ \mathbb{P}_0} ⊂ \Sporadic $ in such a way that the requirements of \autoref{th:existing-generation-even} are satisfied.

Our starting point is the set $ \Sporadic $. Recall from \autoref{sec:existing-crit} that this set consists of all 1-adic cocycles which are of the form $ ν^1 (α) = λ_α α $ for every angle $ α $. Simply speaking, $ \Sporadic $ is the set of cocycles among the 1-adic cochains which merely scale angles. Our first step in this section is to make the cocycle condition explicit:

\begin{lemma}[\paperone]
Let $ \cA $ be a full arc system with [NMD]. Let $ ν ∈ \HC^{\even} (\Gtl \cA) $ be an even cochain such that $ ν^{≠1} = 0 $ and $ ν(α) = λ_α α $ for some scalar $ λ_α $ for every indecomposable angle $ α $. Then
\begin{equation*}
dν = 0 \quad \Longleftrightarrow \quad \text{for every polygon } α_1, …, α_k: \sum_{i = 1}^k λ_{α_i} = 0.
\end{equation*}
\end{lemma}

Whenever $ ν ∈ \Sporadic $, we also write $ \# ν (α) = λ_α $ for the scalar coefficient of $ ν (α) $ whenever $ α $ is an angle in $ \cA $. For instance, we have $ \# ν (α_k … α_1) = \# ν (α_k) + … + \# ν (α_1) $.

\begin{lemma}
The quotient $ \Sporadic / [\id, -] $ of sporadic cocycles by sporadic inner derivations is isomorphic to $ H_1 (S, M; ℂ) $. This space has dimension $ 2g - 1 + |M| $.
\end{lemma}

\begin{proof}
The proof consists of two steps. To compare $ \Sporadic/[\id, -] $ and $ H_1 (S, M; ℂ) $, we will pick a cell decomposition of the surface $ S $ and show that its degree one cocycles are $ \Sporadic $, while its degree one coboundaries are $ [\id, -] $. In the second step, we read off the dimension of this relative homology space by an alternative cell decomposition.

For the first step, let us describe the cell decomposition we put on $ S $. It is a dual decomposition to the punctures, arcs and polygons of $ \cA $. The zero-dimensional cells are the midpoints of the polygons plus the punctures $ M $. The one-dimensional cells are arrows from the midpoints of the polygons to all corners around the polygon. The surface is split into topological disks by the one-dimensional cells, one for each arc of $ \cA $. The two-dimensional cells are defined to be those disks. This cell complex is depicted in \autoref{fig:hochschild-sporadic-dual-cells}.

Regard the relative cellular chain complex formed by this cell decomposition, relative to the zero-cells $ M $. Our aim is to identify its degree-one cocycles with $ \Sporadic $ and its degree-one coboundaries with $ [\id, -] $.

Let us regard a degree-one cocycle $ η $. Such a cocycle can be written as a linear combination of arrows from the centers of the polygons to the corners. Note that the arrows are precisely in correspondence with the indecomposable angles of $ \cA $. Therefore let us write $ η = \sum_α λ_α α $ with coefficients $ λ_α ∈ ℂ $. The cocycle condition, relative to $ M $, is equivalent to requiring that the sum of the coefficients $ λ_α $ vanishes along each polygon.

What are the coboundaries? They are spanned by the coboundaries of all two-dimensional cells. Regard one two-dimensional cell given by an arc $ a ∈ \cA $. Its boundary consists of the signed sum of the four one-cells bounding it. In terms of the angle interpretation of the one-cells as angles of $ \cA $, this signed sum is precisely $ α_1 - α_2 + α_4 - α_3 $, where the angles are numbered as in \autoref{fig:hochschild-sporadic-angles}. This coboundary corresponds exactly to Hochschild coboundary $ d(\id_a) ∈ [\id, -] $. In other words, the quotient of degree-one cocycles by coboundaries precisely computes $ \Sporadic/[\id, -] $.

For the second step, we are supposed to compute the dimension of $ H_1 (S, M; ℂ) $ by choosing an easier cell decomposition. Note that the relative homology does not depend on the location of the points $ M $, as long as they are distinct. Next, recall that every closed surface of genus $ g $ can be split into a single disk by $ 2g $ non-crossing loops $ a_1, …, a_g $ and $ b_1, …, b_g $, all starting and ending at a single point $ p_1 $. The boundary of the disk is given by the sequence $ a_1, b_1, a_1^{-1}, b_1^{-1}, … $.

Now form the desired cell decomposition as follows. The zero-cells are $ p_1 $, plus $ |M| - 1 $ additional points $ p_2, …, p_{|M|} $ lying on $ a_1 $. The one-cells are the $ 2g - 1 $ arcs plus the intervals between the points on $ a_1 $. Their complement in $ S $ is a single disk. Use this disk as the single two-cell. This cell decomposition is depicted in \autoref{fig:hochschild-sporadic-easy-cells}.

We are now ready to compute the degree-one homology of the cell complex of this cell decomposition, relative to $ p $ and the $ |M| - 1 $ many points lying on $ a_1 $. In fact, all $ 2g - 1 + |M| $ arcs of the cell decomposition are cocycles, since all endpoints were chosen relative. The space of coboundaries is spanned by the boundary of the single disk. Since all arcs appear precisely twice around this disk with opposite orientation, the space of degree-one coboundaries vanishes. We conclude the relative homology $ H_1 (S, M; ℂ) $ is of dimension $ 2g - 1 + |M| $.
\end{proof}

\begin{figure}
\centering
\begin{subfigure}{0.25\linewidth}
\centering
\begin{tikzpicture}
\path[draw, gray, dashed] (0, 0) -- ++(right:1.5) coordinate (B) -- ++(up:1.5) coordinate (C) -- ++(left:1.5) coordinate (D) -- ++(down:1.5);
\path (0.75, 0.75) coordinate (M);
\path[fill] (M) circle[radius=0.05];
\path[draw, gray, dashed] (0, 0) -- ++(225:0.75);
\path[draw, gray, dashed] (B) -- ++(315:0.75);
\path[draw, gray, dashed] (C) -- ++(45:0.75);
\path[draw, gray, dashed] (D) -- ++(135:0.75);
\path[draw, ->] (M) ++(45:0.2) -- ++(45:0.7);
\path[draw, ->] (M) ++(135:0.2) -- ++(135:0.7);
\path[draw, ->] (M) ++(225:0.2) -- ++(225:0.7);
\path[draw, ->] (M) ++(315:0.2) -- ++(315:0.7);
\path[draw, ->] (0.75, 2.25) ++(315:0.2) -- ++(315:0.7);
\path[draw, ->] (0.75, 2.25) ++(225:0.2) -- ++(225:0.7);
\path[fill] (0.75, 2.25) circle[radius=0.05];
\path[draw, ->] (0.75, -0.75) ++(45:0.2) -- ++(45:0.7);
\path[draw, ->] (0.75, -0.75) ++(135:0.2) -- ++(135:0.7);
\path[fill] (0.75, -0.75) circle[radius=0.05];
\path[draw, ->] (-0.75, 0.75) ++(45:0.2) -- ++(45:0.7);
\path[draw, ->] (-0.75, 0.75) ++(315:0.2) -- ++(315:0.7);
\path[fill] (-0.75, 0.75) circle[radius=0.05];
\path[draw, ->] (2.25, 0.75) ++(135:0.2) -- ++(135:0.7);
\path[draw, ->] (2.25, 0.75) ++(225:0.2) -- ++(225:0.7);
\path[fill] (2.25, 0.75) circle[radius=0.05];
\end{tikzpicture}
\caption{Cell decomposition identifying $ S/I $ with relative homology. Arcs of $ \cA $ are drawn dashed.}
\label{fig:hochschild-sporadic-dual-cells}
\end{subfigure}
\begin{subfigure}{0.25\linewidth}
\centering
\begin{tikzpicture}
\path[draw, semithick, ->] (0, 0) -- ++(0, 1.5) node[midway, left] {$ a $} coordinate[pos=0.3] (alpha3-start) coordinate[pos=0.7] (alpha1-end);
\path[draw] (0, 1.5) -- ++(150:1) coordinate[pos=0.5] (alpha1-start);
\path[draw] (0, 1.5) -- ++(30:1) coordinate[pos=0.5] (alpha2-end);
\path[draw] (0, 0) -- ++(210:1) coordinate[pos=0.5] (alpha3-end);
\path[draw] (0, 0) -- ++(330:1) coordinate[pos=0.5] (alpha4-start);
\path[draw, bend right=65, ->] (alpha1-start) to node[near end, above] {$ α_1 $} (alpha1-end);
\path[draw, bend right=65, ->] (alpha1-end) to node[near start, above] {$ α_2 $} (alpha2-end);
\path[draw, bend right=65, ->] (alpha3-start) to node[near start, below] {$ α_3 $} (alpha3-end);
\path[draw, bend right=65, ->] (alpha4-start) to node[near end, below] {$ α_4 $} (alpha3-start);
\end{tikzpicture}
\caption{Surrounding angles}
\label{fig:hochschild-sporadic-angles}
\end{subfigure}
\begin{subfigure}{0.25\linewidth}
\centering
\begin{tikzpicture}
\path[draw] (0, 0) -- ++(right:1) -- ++(45:1) -- ++(up:1) -- ++(135:1) node[midway, above right] {$ b_1^{-1} $} -- ++(left:1) node[midway, above] {$ a_1^{-1} $} coordinate[pos=0] (p6) coordinate[pos=0.25] (p7) coordinate[pos=0.5] (p8) coordinate[pos=0.75] (p9) coordinate[pos=1] (p10) -- ++(225:1) node[midway, above left] {$ b_1 $} -- ++(down:1) node[midway, left] {$ a_1 $} coordinate[pos=0] (p1) coordinate[pos=0.25] (p2) coordinate[pos=0.5] (p3) coordinate[pos=0.75] (p4) coordinate[pos=1] (p5) -- ++(315:1);
\foreach \i in {1,...,10} \path[fill] (p\i) circle[radius=0.03];
\path (p1) node[right] {\small $ p_1 $};
\path (p2) node[right] {$ \vdots $};
\path (p4) node[right] {\small $ p_{|M|} $};
\path (p5) node[right] {\small $ p_1 $};
\end{tikzpicture}
\caption{Easy cell decomposition}
\label{fig:hochschild-sporadic-easy-cells}
\end{subfigure}
\caption{The picture illustrates two different decompositions of the surface into topological cells. While the decomposition (\subref{fig:hochschild-sporadic-dual-cells}) abstractly enumerates the sporadic classes, the alternative cell decomposition (\subref{fig:hochschild-sporadic-easy-cells}) practically computes this number.}
\end{figure}
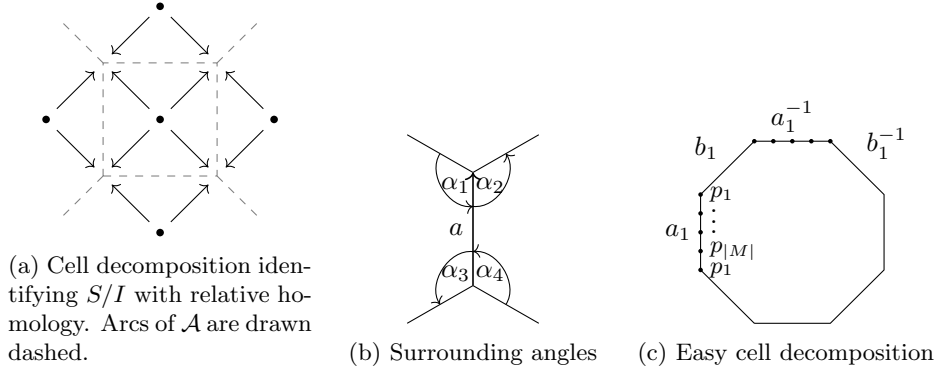

\begin{definition}
The \emph{sporadic classes} are any choice of basis representatives $ (ν_P)_{P ∈ \mathbb{P}_0} ⊂ \Sporadic $ for $ \Sporadic / [\id, -] $. The index set $ \mathbb{P}_0 $ has cardinality $ 2g - 1 + |M| $.
\end{definition}

\section{Even Hochschild cocycles}
\label{sec:even}
In this section, we construct explicit even Hochschild cocycles. Recall that we have already constructed sporadic even Hochschild cocycles $ (ν_P)_{P ∈ \mathbb{P}_0} ⊂ \Sporadic $ in \autoref{sec:existing-sporadic}. In the present section, we define a second class of Hochschild cocycles which we call the \emph{ordinary} even Hochschild cocycles. The sporadic and ordinary even Hochschild cocycles together will form a basis for the even Hochschild cohomology.

We proceed as follows: In \autoref{sec:even-construction}, we given an explicit description of Hochschild cocycles $ ν $. In \autoref{sec:even-parking}, we check that the Hochschild cochain $ d ν $ vanishes on a certain type of sequences which we call parking garage sequences. In \autoref{sec:even-other}, we check that $ d ν $ also vanishes on all other types of sequences. In total, we obtain that $ dν = 0 $. In \autoref{sec:even-summary} we construct the Hochschild cocycles $ ν_{m, r} $. We show that together with the sporadic cocycles they provide a basis for $ \HH^{\even} (\Gtl \cA) $. Finally, we comment on the Gerstenhaber bracket and cup product on Hochschild cohomology.

\subsection{Construction}
\label{sec:even-construction}
In this section we construct even Hochschild cocycles explicitly from certain input data. The input data consists of a choice of puncture $ m ∈ M $, a natural number $ r ≥ 1 $ and a choice of “weight” for every indecomposable angle around $ m $. The construction of the Hochschild cocycle associated with this data is similar to the odd case, albeit more tricky.

The basic idea of the construction is as follows: Let $ α_1, …, α_k, ℓ_m^r $ be a disk sequence. Then we define $ ν(α_k, …, α_1) $ as the identity on the first, equivalently last arc of the sequence. This idea is depicted in \autoref{fig:even-construction-basic}. This does not suffice however to make $ ν $ a cocycle. Instead, we need to give $ ν $ nonzero values on certain other special sequences and choose the scalars of these values in a clever way. It turns out there is no canonical choice for the scalars. We therefore start from the datum of a scalar value $ \# ν (α) $ for every indecomposable angle $ α $ winding around $ m $. Defining the special sequences is rather intricate and makes use of what we call turning angles and magic angles. To define magic angles, we have to define yet another auxiliary notion, the splitting angles. The structure of the section is summarized in the logical diagram \autoref{fig:even-construction-logical}.

\begin{figure}
\centering
\begin{tikzpicture}
\newcommand{\inputconnect}[2]{\path (#1.east) -- (#2.west) node[midway, sloped] {$ \implies $};}
\path (0, 0) node[left] (A) {Input scalars $ \# ν^1 (α) $};
\path (-3, -0.6) node[left] (B) {Splitting angles};
\path (0, -0.6) node[left] (C) {Magic angles};
\path (0, -1.2) node[left] (D) {Turning angles};
\path (3, -0.6) node (F) {Construction of $ ν $};
\inputconnect BC
\inputconnect AF
\inputconnect CF
\inputconnect DF
\end{tikzpicture}
\caption{Logical structure of notions}
\label{fig:even-construction-logical}
\end{figure}
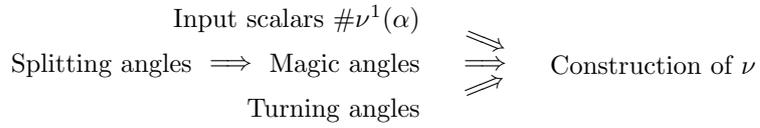

\paragraph*{Arc system}
We fix a full arc system $ \cA $ which satisfies the [NMD] condition.

\paragraph*{Input scalars}
We assume the choice of a puncture $ m ∈ M $, a natural number $ r ≥ 1 $ and the choice of a scalar $ \# ν^1 (α) $ for every indecomposable angle $ α $ winding around the puncture $ m $. An example of input scalars is depicted in \autoref{fig:even-construction-inputscalars}.

\paragraph*{Splitting angles}
We introduce here precise measurement for certain angles. In terms of orbigons, it concerns angles between different ways of viewing an orbigon as a fold. We try to break down the terminology as far as possible to the more elementary notion of disk sequences.

Let $ s = α_1, …, α_k $ be an angle sequence. We regard indices $ 1 ≤ i ≤ k $ such that $ α_i $ has a decomposition $ α_i = α_i^{(2)} α_i^{(1)} $ such that $ α_1, …, α_i^{(1)}, ℓ_m^r, α_{i+1}^{(2)}, …, α_k $ is a disk sequence. We define the \emph{splitting set} $ I_s $ of $ s = α_1, …, α_k $ to be the set of such indices and decompositions:
\begin{align*}
I_s = \{(i, α_i^{(1)}, α_i^{(2)}) \running & 1 ≤ i ≤ k, \quad α_i = α_i^{(2)} α_i^{(1)}, \\
& α_1, …, α_i^{(1)}, ℓ_m^r, α_{i+1}^{(2)}, …, α_k \text{ is a disk sequence}\}.
\end{align*}
The set $ I_s $ may be empty. The more often the sequence $ α_1, …, α_k $ winds around $ m $, the larger the set $ I_s $. An example sequence $ s = α_1, …, α_6 $ in case $ r = 1 $ together with its splitting set $ I_s $ is depicted in \autoref{fig:even-construction-splittingsequence}. The elements of $ I_s $ are totally ordered by the number $ i $, or the length of $ α_i^{(1)} $ among equal indices. We capture the angle between two elements of $ I_s $ in the following terminology:
\begin{definition}
Let $ (i, α_i^{(1)}, α_i^{(2)}) ≤ (j, α_j^{(1)}, α_j^{(2)}) $ be two elements of $ I_s $. Then the \emph{splitting angle between} $ (i, α_i^{(1)}, α_i^{(2)}) $ \emph{and} $ (j, α_j^{(1)}, α_j^{(2)}) $ is
\begin{itemize}
\item if $ i < j $, then $ α $ is the angle such that $ α_i^{(2)}, α_{i+1}, …, α_{j-1}, α_j^{(1)}, α $ is a disk sequence.
\item if $ i = j $, then we set $ α = (α_i^{(1)})^{-1} α_j^{(2)} $
\end{itemize}
\end{definition}
In case $ i = j $, the splitting angle is simply speaking the difference between $ α_i^{(1)} $ and $ α_j^{(1)} $. In \autoref{fig:even-construction-splittingsequence}, we have illustrated the splitting angle in case $ i < j $. In the figure, the splitting angle is drawn dashed.

\begin{definition}
If $ I_s $ is nonempty, the \emph{splitting angle of} an element $ (i, α_i^{(1)}, α_i^{(2)}) ∈ I_s $ is the splitting angle between $ \min I_s $ and $ (i, α_i^{(1)}, α_i^{(2)}) $.
\end{definition}

In terms of orbigons, all terminology is easily described as follows: The set $ I_s $ is nonempty if $ α_1, …, α_k $ is the reduced sequence of an orbigon $ X $ of type $ (m, r) $. The set $ I_s $ is then simply the set of all possible ways the orbigon $ X $ can be obtained via folding. The minimum $ \min I_s $ is the earliest possible way to obtain $ X $ via folding.

\paragraph*{Construction of the cochain}
We are now ready to construct a cochain $ ν $ from given collection of input scalars $ \# ν^1 $. The idea is to define $ ν^1 $ as the derivation which sends an indecomposable angles $ α $ to $ \#ν^1 (α) α ℓ_m^r $. Whenever $ α_1, …, α_k $ are indecomposable angles around $ m $ such that $ α_k … α_1 ≠ 0 $, let us already now write
\begin{equation*}
\# ν^1 (α_k … α_1) = \# ν^1 (α_k) + … + \# ν^1 (α_1).
\end{equation*}
The higher component $ ν^{≥2} $ will be defined on four types of distinguished sequences. To every such distinguished sequence, we define the associated turning angle and the associated magic angle. The contribution of the sequence to $ ν^{≥2} $ is defined in terms of these two angles. The full definition reads as follows:

\begin{figure}
\centering
\begin{subfigure}{0.25\linewidth}
\centering
\begin{tikzpicture}[scale=1.5]
\path[draw] (0, 0) -- ++(45:1) coordinate[pos=0.3] (4-start) coordinate[pos=0.7] (3-end) -- ++(135:1) coordinate[pos=0.3] (3-start) coordinate[pos=0.7] (2-end) -- ++(225:1) coordinate[pos=0.3] (2-start) coordinate[pos=0.7] (1-end) -- ++(315:1) coordinate[pos=0.3] (1-start) coordinate[pos=0.7] (4-end);
\path[draw] (0, 0) -- (0, 0.707) coordinate[pos=0.4] (0) node[pos=0.7, left, shift={(0.1, 0)}] {$ a $};
\path[fill] (0, 0.707) circle[radius=0.05] node[above] {$ m $};
\path[draw, ->, bend right=22] (0) to node[at end, below] {$ α_1 $} (4-end);
\path[draw, ->, bend right=45] (1-start) to node[at end, left] {$ α_2 $} (1-end);
\path[draw, ->, bend right=45] (2-start) to node[at end, above] {} (2-end);
\path[draw, ->, bend right=45] (3-start) to node[at end, right] {} (3-end);
\path[draw, ->, bend right=22] (4-start) to node[at start, below] {$ α_k $} (0);
\path (0, -0.5) node {$ ν(α_k, …, α_1) = \id_a $};
\end{tikzpicture}
\caption{Basic idea}
\label{fig:even-construction-basic}
\end{subfigure}
\begin{subfigure}{0.25\linewidth}
\centering
\begin{tikzpicture}[scale=1.5]
\path[draw] (-1, 0) -- (1, 0) coordinate[pos=0.25] (1) coordinate[pos=0.75] (4);
\path[draw] (240:1) -- (60:1) coordinate[pos=0.25] (2) coordinate[pos=0.75] (5);
\path[draw] (300:1) -- (120:1) coordinate[pos=0.25] (3) coordinate[pos=0.75] (6);
\path[draw, bend right=30] (1) to (2);
\path[draw, bend right=30] (2) to (3);
\path[draw, bend right=30] (3) to (4);
\path[draw, bend right=30] (4) to (5);
\path[draw, bend right=30] (5) to (6);
\path[draw, bend right=30] (6) to (1);
\path[fill] (0, 0) circle[radius=0.05] node[above] {$ m $};
\path (30:0.4) node {3};
\path (90:0.4) node {2};
\path (150:0.4) node {$ -6 $};
\path (210:0.4) node {5};
\path (270:0.4) node {$ 1.5 $};
\path (330:0.4) node {1};
\end{tikzpicture}
\caption{Input scalars}
\label{fig:even-construction-inputscalars}
\end{subfigure}
\begin{subfigure}{0.4\linewidth}
\centering
\begin{tikzpicture}
\path[draw] (0, 0) -- ++(right:1) coordinate (A) coordinate[pos=0.3] (1-start) coordinate[pos=0.7] (6-end) -- ++(60:1) coordinate[pos=0.3] (6-start) coordinate[pos=0.7] (5-end) -- ++(120:1) coordinate (B) coordinate[pos=0.3] (5-start) coordinate[pos=0.7] (4-end) -- ++(left:1) coordinate (C) coordinate[pos=0.3] (4-start) coordinate[pos=0.7] (3-end) -- ++(240:1) coordinate[pos=0.3] (3-start) coordinate[pos=0.7] (2-end) -- ++(300:1) coordinate[pos=0.3] (2-start) coordinate[pos=0.7] (1-end);
\path[draw, ->, bend right=60] (1-start) to (1-end);
\path[draw, ->, bend right=60] (2-start) to (2-end);
\path[draw, ->, bend right=60] (3-start) to (3-end);
\path[draw, ->, bend right=60] (4-start) to (4-end);
\path[draw, ->, bend right=60] (5-start) to (5-end);
\path[draw, ->, bend right=60] (6-start) to (6-end);
\path (A) node[below] {\tiny $ α_2 = α_2^{(2)} α_2^{(1)} $};
\path (B) node[above] {\tiny $ α_6 = α_6^{(2)} α_6^{(1)} $};
\path (1, 0) ++(60:1) node[right] {$ α_1 $};
\path[fill] ($ (A)!0.5!(C) $) circle[radius=0.05] coordinate (m);
\path[draw] (A) -- (m) coordinate[midway] (m-end) -- (B) coordinate[midway] (m-start);
\path[draw, ->, bend right=105, looseness=2, dashed] (m-start) to (m-end);
\path (0.5, -1) node {$ I_s = \{(2, α_2^{(1)}, α_2^{(2)}), (6, α_6^{(1)}, α_6^{(2)})\} $};
\end{tikzpicture}
\caption{Splitting angle (dashed)}
\label{fig:even-construction-splittingsequence}
\end{subfigure}
\caption{Illustration of ideas and auxiliary notions}
\end{figure}
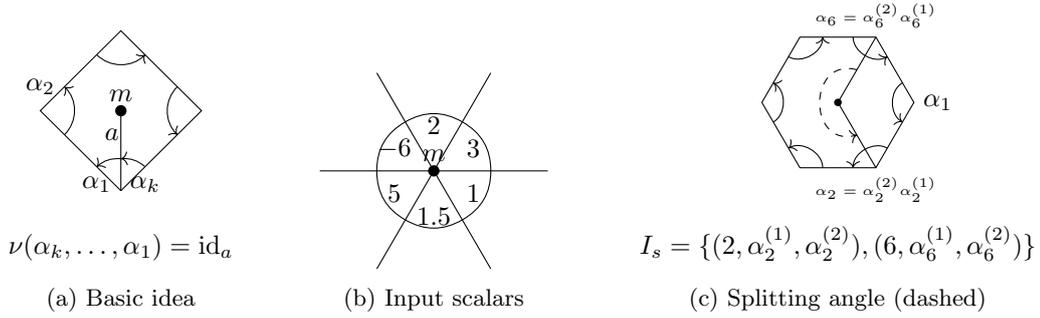

\begin{definition} 
\label{def:even-construction-def}
Let $ \cA $ be a full arc system with [NMD]. Let $ m ∈ M $, $ r ≥ 1 $ and let $ \# ν $ be a collection of input scalars around $ m $. Then the associated even Hochschild cochain $ ν $ is defined by the following rules. For every rule, we define indicate its \emph{turning angle} and its \emph{magic angle}.
\begin{itemize}
\item The 0-adic component $ ν^0 $ vanishes.
\item The 1-adic component $ ν^1 $ is defined by $ ν^1 (α) = \# ν^1 (α) α ℓ_m^r $ for $ α $ winding around $ m $.
\item An angle sequence $ α_1, …, α_k $ is \emph{end-split with turning angle $ < ℓ^r $} if there exists an angle $ α $ winding around $ m $ with $ α < ℓ^r $ such that $ α_1, …, α_k, α $ is a disk sequence. The turning angle of the sequence is the angle $ α $. The magic angle of the sequence is $ ℓ_m^r $. The contribution to $ ν $ is
\begin{equation*}
ν^k (α_k, …, α_1) = (-1)^{‖α‖} \#ν^1 (α) α^{-1} ℓ_m^r.
\end{equation*}
\item An angle sequence $ α_1 β, …, α_k $ with $ α_1 β ≠ 0 $ is \emph{old-era end-split with turning angle $ ℓ^r $} if $ α_1, …, α_k, ℓ^r $ is a disk sequence. The turning angle of the sequence is $ ℓ^r $. The magic angle of the sequence is $ ℓ_m^r $. The contribution to $ ν $ is
\begin{equation*}
ν^k (α_k, …, α_1 β) = - \# ν^1 (ℓ_m^r) β.
\end{equation*}
\item An angle sequence $ α_1 β, …, γ α_k $ with $ γ ≠ \id $ is \emph{new-era end-split with turning angle $ ℓ^r $} if $ α_1, …, α_k, ℓ^r $ is a disk sequence. The turning angle of the sequence is $ ℓ^r $. Regard the arc (more precisely, arc incidence) which is the head of $ α_k $, equivalently the tail of $ α_1 $. The magic angle of the sequence is the indecomposable angle $ n $ around $ m $ which follows this arc clockwise around $ m $. The contribution to $ ν $ is
\begin{equation*}
ν^k (γ α_k, …, α_1 β) = - \# ν^1 (n) γβ.
\end{equation*}
\item An angle sequence $ α_1, …, γ α_k $ or $ α_1 β, …, α_k $ is \emph{middle-split} if there is a $ 0 < i < k $ such that $ α_1, …, α_i, ℓ^r, α_{i+1}, …, α_k $ is a disk sequence. The turning angle of the sequence is $ ℓ^r $. The triple $ (i, α_i, α_{i+1}) $ defines an element of the splitting set $ I_s $ of the sequence $ s = α_1, …, α_{i+1} α_i, …, α_k $. The magic angle of the sequence is the splitting angle $ α $ of $ (i, α_i, α_{i+1}) $ with respect to the angle sequence $ s $. The contribution to $ ν $ is
\begin{align*}
ν^k (γ α_k, …, α_1) &= (-1)^{‖α_1‖ + … + ‖α_i‖} \# ν^1 (α) γ, \\
ν^k (α_k, …, α_1 β) &= (-1)^{‖α_1‖ + … + ‖α_i‖} \# ν^1 (α) β.
\end{align*}
\end{itemize}
\end{definition}

\begin{remark}
The higher components $ ν^{≥2} $ are well-defined: Any angle sequence falls within at most one of the four types presented in \autoref{def:even-construction-def}. Whenever it falls within one of the types, its presentation in terms of $ α_i $, $ β $ or $ γ $ is unique. In case the angle sequence is middle-split, the index $ i $ is unique. Alternatively, it is possible to circumvent this uniqueness statement. Indeed, add up contributions to $ ν $ instead, as in the odd case of \paperone.
\end{remark}

\begin{remark}
The sign rules follow a united pattern: If $ α_1, …, α_k $ is end-split with turning angle $ α < ℓ^r $, the sign $ (-1)^{‖α‖} $ is equal to $ (-1)^{‖α_1‖ + … + ‖α_k‖} $ since $ α_1, …, α_k, α $ is a disk sequence and reduced degrees in a disk sequence add up to even parity. If $ α_1, …, α_k $ is (old-era or new-era) end-split with turning angle $ ℓ^r $, the sign $ -1 $ is equal to $ (-1)^{‖α_1‖ + … + ‖α_k‖} $. More generally, the sign consumes precisely the angles between the minimum element of $ I_s $ and the split actually taken by the sequence.
\end{remark}

\begin{remark}
The rules for middle-split and end-split $ ν $ are analogous: They yield exactly the same result, except that the end-split rule for magic angle $ < ℓ^r $ does not allow for additional $ β $ and $ γ $ at the front and at the back. For example, even the signs agree, since $ ‖ℓ_m^r‖ $ is odd. Let us explain why we distinguish the two rules. The first rule yields $ α_k^{-1} ν^1 (α_k) α_k^{-1} $. This angle always winds around $ m $. Evaluate
\begin{equation*}
0 = (dν) (γ, α_k, …, α_1) = μ^2 (γ, ν(α_k, …, α_1)) - ν(μ^2 (γ, α_k), …, α_1)
\end{equation*}
This yields $ ν(γ α_k, …, α_1) = γ ν(α_k, …, α_1) $. If $ α_k $ is less than $ ℓ_m^r $, then $ ν(α_k, …, α_1) $ is non-empty and winds around $ m $, while $ γ $ leaves the arc at the opposite side. Hence the product vanishes, except if $ ν(α_k, …, α_1) $ is the identity. This happens precisely in the borderline case that $ α_k $ consists of $ r $ full turns: $ α_k = ℓ_m^r $. Only in this case additional $ β $ and $ γ $ on the left and right make sense. This explains the distinction between the first and second rule.

The rule for middle-split sequences has appearance similar to the odd case. One may in principle add $ γ $ and $ β $ simultaneously on both sides. This addition is however vacuous: The angles $ γ $ and $ β $ are not composable and $ γβ = 0 $.
\end{remark}

\subsection{Cancellation on parking garage sequences}
\label{sec:even-parking}
In this section, we perform first checks for the cocycle condition. Our starting point is a cochain $ ν $ defined in \autoref{def:even-construction-def} from input scalars $ \#ν^1 $. In the present section, we check that $ dν = [μ, ν] $ vanishes on certain sequences, which we call \emph{parking garage sequences}.

Many of the terms in $ [μ, ν] $ are easy to cancel away in pairs. Some are harder and cancel away only as a whole. All sequences $ α_1, …, α_k $ producing hard terms in $ [μ, ν] $ are of the same type: they have an angle $ α_i $ around $ m $ such that $ μ(…, ν^1 (α_i), …) $ is a nonzero contribution. That is, once we prolong $ α_i $ by $ r $ turns around $ m $, the sequence $ α_1, …, ℓ^r α_i, …, α_k $ becomes a disk sequence. After $ α_k $ plus $ ℓ^r $ turns around $ m $, the prolonged sequence compensates its turns by winding back around $ m $ in clockwise direction. Let us regard the path traced by the sequence as it winds back. Regard the polygons lying around $ m $ in clockwise order. Since $ α_1, …, ℓ^r α_i, …, α_k $ is supposed to be a disk sequence, the winding back path runs around all these polygons in clockwise order. It may have additional disks stitched to it at the polygons' outside, but the basic structure is a helix consisting of the polygons around $ m $. Such a sequence resembles a parking garage spiral, with optional parking space attached on the exterior of each polygon. The schematic is depicted in \autoref{fig:even-parking-figure}.

Let us describe in formulas how the helix is formed. Denote for a moment by $ P_1, …, P_l $ the polygons traced by the sequence. Let $ α_1^{(i)}, …, α_{s_i}^{(i)} $ be their internal angles, with $ α_1^{(i)} $ being the angle at $ m $. The parking garage sequence then consists of the angles
\begin{equation*}
α_2^{(1)}, …, α_{s_1-1}^{(1)}, α_2^{(2)} α_{s_1}^{(1)}, …, α_{s_2-1}^{(2)}, α_2^{(3)} α_{s_2}^{(2)}, …, α_{s_l-1}^{(l)}
\end{equation*}
plus the long turn angle around $ m $, consisting of all the polygon angles at $ m $ minus $ r $ full turns:
\begin{equation*}
α_1^{(1)} … α_l^{(l)} ℓ^{-r}.
\end{equation*}
Let us explain how the additional parking space is attached. Regard a polygon $ α_1^{(i)}, …, α_{s_i}^{(i)} $ in the spiral. The angle $ α_1^{(i)} $ is the angle at $ q $. The angles $ α_2^{(i)} $ and $ α_{s_i}^{(i)} $ are the angles next to $ q $ and are used to attach the polygons to each other, forming the parking spiral. The parking spiral has $ l $ polygons and $ l-1 $ interior arcs. Its outer boundary consists of $ \sum (s_i - 1) - 1 $ many exterior arcs, a start arc and an end arc. The additional parking space in the form of disk sequences $ β_1, …, β_m $ may now be attached to the exterior arcs. Wherever we add parking space around the spiral, we augment the garage sequence by that additional disk sequence. When attaching to arcs not involved in the spiral gluing, the augmentation looks like
\begin{equation*}
…, α_j^{(i)}, α_{j+1}^{(i)}, … ~\leadsto~ …, β_1 α_j^{(i)}, β_2, …, β_{m-1}, α_{j+1}^{(i)} β_m, ….
\end{equation*}
When attaching to one of the two exterior arcs of the polygon next to the gluing, the augmentation rather looks like
\begin{equation*}
…, α_2^{(j+1)} α_{s_j}^{(j)}, α_3^{(j+1)}, … ~\leadsto~ …, β_1 α_2^{(j+1)} α_{s_j}^{(j)}, β_2, …, β_{m-1}, α_3^{(j+1)} β_m, …
\end{equation*}
in case of the first exterior arc of a polygon, and looks like
\begin{equation*}
…, α_{s_j-1}^{(j)}, α_2^{(j+1)} α_{s_j}^{(j)}, … ~\leadsto~ …, β_1 α_{s_j-1}^{(j)}, β_2, …, β_{m-1}, α_2^{(j+1)} α_{s_j}^{(j)} β_m, ….
\end{equation*}
in case of the last exterior arc of a polygon. Let us put this definition on paper.

\begin{definition}
A \emph{parking garage sequence} consists of tracing consecutive polygons around $ m $ in clockwise order, and compensating all these turns minus $ ℓ^r $ by a single angle around $ m $. Additional disk sequences may be stitched to the exterior arcs of the sequence.
\end{definition}

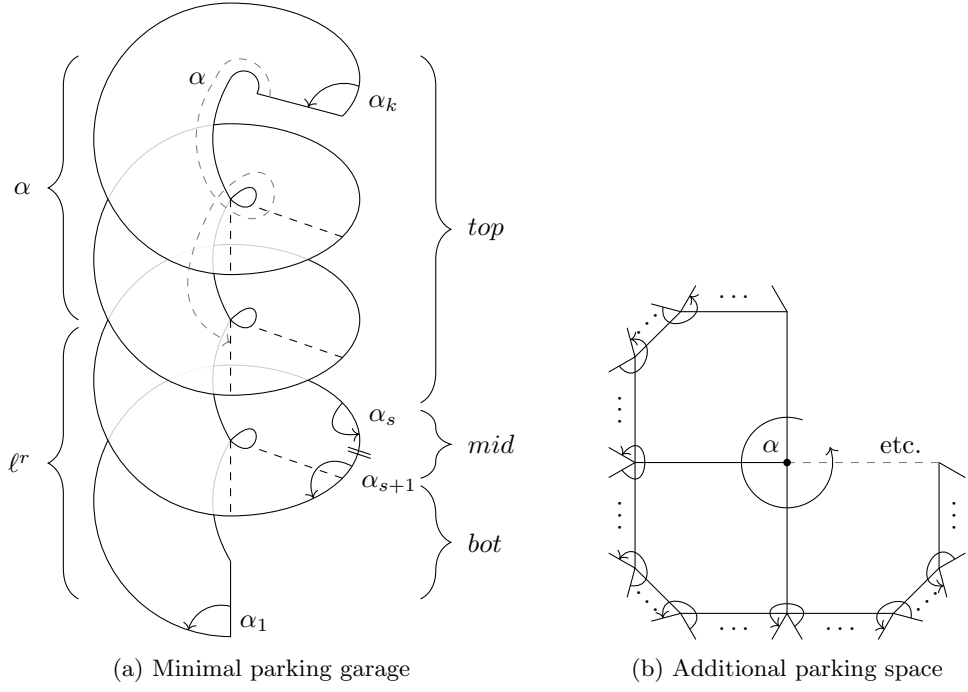
\begin{figure}
\centering
\begin{subfigure}{0.45\linewidth}
\centering
\begin{tikzpicture}
\path[draw] (0, 0) -- (0, -1) coordinate[pos=0.6] (alpha-1-start) arc(270:90:1.8) coordinate[pos=0.1] (alpha-1-end) coordinate (back-1);
\path (back-1) arc(90:-90:1.7 and 1) coordinate (front-1);
\path[draw] (0, 0) to[bend left] ($ (back-1)!0.5!(front-1) $) coordinate (int-left-1);
\foreach \i in {1, 2, 3} {
\path (back-1) arc(90:-90:1.7 and 1) coordinate (front-1) arc(270:90:1.8) coordinate (back-2);
\path (back-2) arc(90:-90:1.7 and 1) coordinate (front-2);
\path (int-left-1) to[bend left] ($ (back-2)!0.5!(front-2) $) coordinate (int-left-2);
\path[fill=white, inner sep=1mm, opacity=0.8] (back-1) arc(90:-90:1.7 and 1) arc(270:90:1.8) -- (int-left-2) to[bend right] (int-left-1) .. controls ++(0.4, -0.4) and ++(0.5, 0.5) .. coordinate[pos=0.3] (int-front-linear-\i) (int-left-1) -- cycle;
\path[draw] (back-1) arc(90:-90:1.7 and 1) arc(270:90:1.8);
\path[draw] (int-left-1) .. controls ++(0.5, 0.5) and ++(0.4, -0.4) .. ++(0, 0) to[bend left] (int-left-2);
\path (int-left-1) coordinate (int-left-linear-\i);
\path (front-1) coordinate (front-linear-\i);
\path (int-left-2) coordinate (int-left-next-linear-\i);
\path (front-2) coordinate (front-next-linear-\i);
\path (back-1) arc(90:-30:1.7 and 1) coordinate (front-right-linear-\i);
\path (back-1) arc(90:-20:1.7 and 1) coordinate (front-right-angle-start-\i);
\path (back-1) arc(90:-50:1.7 and 1) coordinate (front-right-angle-end-\i);
\path (back-1) arc(90:30:1.7 and 1) coordinate (front-right-prev-start-\i);
\path (back-1) arc(90:5:1.7 and 1) coordinate (front-right-prev-end-\i);
\path (back-1) arc(90:-5:1.7 and 1) coordinate (front-right-split-\i);
\path (back-2) coordinate (back-1);
\path (front-2) coordinate (front-1);
\path (int-left-2) coordinate (int-left-1);
};
\path[draw] (back-2) arc(90:-30:1.7 and 1) coordinate[pos=0.8] (alpha-k-start) coordinate (top-end);
\path[draw] (int-left-2) arc(150:-30:0.2) coordinate (int-top-end);
\path[draw] (top-end) -- (int-top-end) coordinate[pos=0.4] (alpha-k-end) coordinate[pos=0.85] (alpha-start);
\path[draw, ->, bend right=45] (alpha-k-start) to node[at start, below right] {$ α_k $} (alpha-k-end);
\path[draw, ->, bend right=45] (alpha-1-start) to node[at start, below right] {$ α_1 $} (alpha-1-end);
\path[draw, ->, gray, dashed, rounded corners] (alpha-start) to[bend right=80, looseness=1.6] node[at end, black, left] {$ α $} ($ (int-left-next-linear-3) + (-0.2, 0) $) to[bend right] ($ (int-left-next-linear-2) + (-0.2, 0) $) .. controls ++(1, -0.8) and ++(1, 1) .. ++(0, 0) .. controls ++(-0.5, -1) and ++(-0.7, 0.3) .. ($ (int-left-next-linear-1) + (0, -0.3) $) coordinate (alpha-end);
\foreach \i in {1, 2, 3} \path[draw, dashed] (int-left-linear-\i) -- (front-linear-\i);
\foreach \i in {1, 2, 3} \path[draw, dashed] (front-right-linear-\i) -- (int-front-linear-\i);
\path[draw, ->, bend right=70, looseness=1.5] (front-right-angle-start-1) to node[at start, below right] {$ α_{s+1} $} (front-right-angle-end-1);
\path[draw, ->, bend right=70, looseness=2] (front-right-prev-start-1) to node[at end, above right] {$ α_s $} (front-right-prev-end-1);
\path[draw] ($ (front-right-split-1) + (-0.15, 0) $) -- ($ (front-right-split-1) + (0.15, -0.1) $);
\path[draw] ($ (front-right-split-1) + (-0.15, -0.05) $) -- ($ (front-right-split-1) + (0.15, -0.15) $);
\path[draw, decorate, decoration={brace, mirror, amplitude=0.4cm}] ($ (int-left-next-linear-3) + (-2, 0.3) $) -- ($ (int-left-linear-2) + (-2, 0) $) node[midway, left, shift={(-0.5, 0)}] {$ α $};
\path[draw, decorate, decoration={brace, mirror, amplitude=0.4cm}] ($ (int-left-linear-2) + (-2, -0.1) $) -- (-2, -0.5) node[midway, left, shift={(-0.5, 0)}] {$ ℓ^r $};
\path[draw, decorate, decoration={brace, amplitude=0.4cm}] ($ (int-left-next-linear-3) + (2.5, 0.3) $) -- ($ (int-left-linear-1) + (2.5, 0.5) $) node[midway, right, shift={(0.5, 0)}] {$ top $};
\path[draw, decorate, decoration={brace, amplitude=0.4cm}] ($ (int-left-linear-1) + (2.5, 0.4) $) -- ($ (int-left-linear-1) + (2.5, -0.5) $) node[midway, right, shift={(0.5, 0)}] {$ mid $};
\path[draw, decorate, decoration={brace, amplitude=0.4cm}] ($ (int-left-linear-1) + (2.5, -0.6) $) -- (2.5, -0.5) node[midway, right, shift={(0.5, 0)}] {$ bot $};
\end{tikzpicture}
\caption{Minimal parking garage}
\label{fig:hochschild-construction-garage}
\end{subfigure}
\begin{subfigure}{0.4\linewidth}
\centering
\begin{tikzpicture}
\begin{scope}[scale=2]
\path[draw] (0, 0) -- (0, -1) coordinate (1) -- (-0.7, -1) coordinate (2) node[midway, shift={(0, -0.2)}, sloped] {…} -- (-1, -0.7) coordinate (3) node[midway, shift={(0, -0.2)}, sloped] {…} -- (-1, 0) coordinate (4) node[midway, shift={(0, 0.2)}, sloped] {…} -- (-1, 0.7) coordinate (5) node[midway, shift={(0, 0.2)}, sloped] {…} -- (-0.7, 1) coordinate (6) node[midway, shift={(0, 0.2)}, sloped] {…} -- (0, 1) coordinate (7) node[midway, shift={(0, 0.2)}, sloped] {…} -- (0, 0) -- (-1, 0);
\path[draw] (1) -- (0.7, -1) coordinate (8) node[midway, shift={(0, -0.2)}, sloped] {…} -- (1, -0.7) coordinate (9) node[midway, shift={(0, -0.2)}, sloped] {…} -- (1, 0) coordinate (10) node[midway, shift={(0, -0.2)}, sloped] {…};
\path[draw, gray, dashed] (1, 0) -- (0, 0) node[near start, above, black] {etc.};
\path[draw, ->] (70:0.3) arc(70:380:0.3);
\path (135:0.15) node {$ α $};
\end{scope}
\path[fill] (0, 0) circle[radius=0.05];
\path[draw] (1) -- ++(300:0.4) coordinate[pos=0.5] (1-start) (1) -- ++(240:0.4) coordinate[pos=0.5] (1-end);
\path[draw] (2) -- ++(300:0.4) coordinate[pos=0.6] (2-start) (2) -- ++(195:0.4) coordinate[pos=0.6] (2-end);
\path[draw] (3) -- ++(255:0.4) coordinate[pos=0.6] (3-start) (3) -- ++(150:0.4) coordinate[pos=0.6] (3-end);
\path[draw] (4) -- ++(210:0.4) coordinate[pos=0.5] (4-start) (4) -- ++(150:0.4) coordinate[pos=0.5] (4-end);
\path[draw] (5) -- ++(210:0.4) coordinate[pos=0.6] (5-start) (5) -- ++(105:0.4) coordinate[pos=0.6] (5-end);
\path[draw] (6) -- ++(165:0.4) coordinate[pos=0.6] (6-start) (6) -- ++(60:0.4) coordinate[pos=0.6] (6-end);
\path[draw] (7) -- ++(120:0.4);
\path[draw] (8) -- ++(345:0.4) coordinate[pos=0.6] (8-start) (8) -- ++(240:0.4) coordinate[pos=0.6] (8-end);
\path[draw] (9) -- ++(30:0.4) coordinate[pos=0.6] (9-start) (9) -- ++(285:0.4) coordinate[pos=0.6] (9-end);
\path[draw] (10) -- ++(330:0.4);
\foreach \i in {2, 3, 5, 6, 8, 9} \path[draw, ->, bend right=120, looseness=3] (\i-start) to (\i-end);
\foreach \i in {1, 4} \path[draw, ->, bend right=140, looseness=8] (\i-start) to (\i-end);
\end{tikzpicture}
\caption{Additional parking space}
\end{subfigure}
\caption{Illustration of parking garage sequences}
\label{fig:even-parking-figure}
\end{figure}

We aim at showing that $ d ν_P $ vanishes on its parking garage sequences. Regard such a garage sequence $ α_s, …, α_1, α, α_k, …, α_{s+1} $. What terms appear in $ d ν_P $, applied to this sequence? First, it is possible to apply $ ν^1 $ to $ α $. Indeed, the garage sequence becomes a disk sequence once we prolong $ α $ by $ r $ turns and hence $ μ(…, ν^1 (α), …) $ is one of the terms in $ d ν_P $. Second, it is possible to apply an inner end-split $ ν $ to any part of the outer sequence that is precisely $ r $ turns long. Such terms give roughly as many contributions as there are polygon sectors around the spiral, and they add up nicely. Finally, it is also possible to apply an inner $ μ $ to the top-most part of the garage or the bottom-most part of the garage.

Let us explain why no other terms appear.

To ease the calculation, we reduce a given garage sequence to its minimal version that has all additional parking space removed. This minimal version has the same individual terms in $ d ν_P $ as the original garage sequence: Additional parking space merely consists of (incomplete) disk sequences and creates no further options to evaluate $ μ $ or $ ν $. Moreover, the value of the individual terms stays exactly the same: For example, the signs for inner end-split $ ν $ evaluations are independent of the length of the $ β $ and $ γ $ angles entering and leaving the additional parking space. While the signs for the top-most inner $ μ $ do depend on the degree of the angle leaving the parking space at angle $ α $ below the top, this is compensated again by the sign $ (-1)^{‖μ‖ · …} $ associated with this term $ ν(…, μ(α, α_k, …), …) $ in the Hochschild differential, since $ μ $ is odd.

Let us now make this rigorous, check the signs and add up all terms.

\begin{lemma}
We have $ d ν_P = 0 $ on parking garage sequences.
\end{lemma}

\begin{proof}
Let us start by listing up all terms with signs and result. Since we choose a parking garage sequence without extra angles $ β $ or $ γ $ at start or end, all results of contributions $ μ(ν) $ or $ ν(μ) $ in $ d ν_P $ are scalar multiples of the identity of the first/last arc $ h(α_s) = t(α_{s+1}) $ of the garage sequence. In the list below, we indicate this scalar for all terms, as well as the sign due to the Hochschild differential.

We start with the generic case, where the beginning and end of the garage sequence lie somewhere one the outer spiral. That is, the final arc is not $ α $ or $ α_k $. In other words, the first arc is not $ α $ or $ α_1 $.

To ease the calculation, let us define three angles $ \angletop $, $ \anglemid $ and $ \anglebot $, all winding around $ m $. These angles can best be read off from \autoref{fig:hochschild-construction-garage}. First of all, $ \anglemid $ is the sector around $ m $ that the start/end of the sequence lies in. For example, the arc $ h(α_s) = t(α_{s+1}) $ lies in this sector. The angle $ \anglemid $ now splits the spiral angle $ αℓ^r $ into two more parts: $ \angletop $ lying above $ \anglemid $, and $ \anglebot $ lying below $ \anglemid $.

In other words, $ \angletop $ forms a disk sequence together with the angles $ α_{s+1}, …, α_k $ (minus the part of $ α_{s+1} $ and those successor angles that reach into the special sector $ \anglemid $). Similarly, $ \anglebot $ forms a disk sequence with $ α_1, …, α_s $ (minus the part lying in $ \anglemid $).

With this notation we can write $ \anglebot · \anglemid · \angletop = α ℓ^r $. Recall that $ \#ν^1 (β) $ is the scalar coefficient of $ ν^1 (β) $, for any angle $ β $. With this in mind, we have
\begin{equation*}
\#ν^1 (\anglebot) + \#ν^1 (\anglemid) + \#ν^1 (\angletop) = \#ν^1 (α) + \#ν^1 (ℓ^r).
\end{equation*}
For convenience, denote by $ \anglefirst $ that sector around $ m $ that is bottom-most in the garage sequence. In other words, $ α_1 $ lies in this sector.

\begin{itemize}
\item[1.] $ μ(…, ν^1 (α), …) $ \par\noindent This term is characteristic for the garage sequence and appears always. Its result has scalar coefficient $ \#ν^1 (α) $. The Hochschild sign is $ +1 $.
\item[2.] $ μ(…, α, ν(α_k, …), …) $ \par\noindent This top-most inner end-split $ ν $ appears if the outer sequence from start to top is at least $ r $ full turns long. Its result has scalar coefficient $ - \#ν^1 (ℓ^r) $. The Hochschild sign is $ +1 $.
\item[3.] $ μ(…, ν(…, α_1), α, …) $ \par\noindent This bottom-most inner end-split $ ν $ appears if the outer sequence from bottom to stop is at least $ r $ full turns long. Its result has scalar coefficient $ - \#ν^1 (\anglefirst) $, where $ \anglefirst $ is the first sector around $ m $ at the bottom of the sequence. The Hochschild sign is $ +1 $.
\item[4.] $ μ(…, α, \underbrace{…}_{≥1}, ν(…), …) $ \par\noindent Such top-part inner end-split $ ν $ terms appear if the outer sequence from start to top is more than $ r $ full turns long. Their individual result coefficients are $ -\#ν^1 $ of the next sector after their end. In total, all these terms add up to $ - \#ν^1 (\angletop · ℓ^{-r}) = - \#ν^1 (\angletop) + \#ν^1 (ℓ^r) $. The Hochschild sign is $ +1 $.
\item[5.] $ μ(…, ν(…), \underbrace{…}_{≥1}, α, …) $ \par\noindent Such bottom-part inner end-split $ ν $ terms appear if the outer sequence from bottom to stop is more than $ r $ full turns long. Their individual result coefficients are $ -\#ν^1 $ of the next sector after their end. In total, all these terms add up to $ - \#ν^1 (\anglemid) - \#ν^1 (\anglebot · ℓ^{-r}) + \#ν^1 (\anglefirst) $. The Hochschild sign is $ +1 $.
\item[6.] $ ν(…, μ(α, α_k, …, α_t), …) $ \par\noindent This top-most inner $ μ $ appears if the outer sequence from start to top includes $ α $. The first angle of the inner $ μ $ is a certain $ α_t $, and a part $ α_t^{(1)} $ of it reaches outside the disk, so write $ α_t = α_t^{(2)} α_t^{(1)} $. The outer $ ν $ application is middle-split and gives an extra sign, so that the result coefficient is
\begin{equation*}
(-1)^{|α_t^{(1)}} \#ν^1 (\angletop · α^{-1}) · (-1)^{‖α_{s+1}‖ + … + ‖α_t^{(1)}‖}.
\end{equation*}
The Hochschild sign is $ (-1)^{1 + ‖α_{s+1}‖ + … + ‖α_{t-1}‖} $, rendering a total contribution to $ d ν_P $ of merely $ \#ν^1 (\angletop · α^{-1}) $.
\item[7.] $ ν(…, μ(α_t, …, α_1, α), …) $ \par\noindent This bottom-most inner $ μ $ appears if the outer sequence from bottom to stop includes $ α $. The final angle of the inner $ μ $ is a certain $ α_t $, and a part $ α_t^{(2)} $ of it reaches outside the disk, so write $ α_t = α_t^{(2)} α_t^{(1)} $. The outer $ ν $ is again middle split, yielding a result coefficient of $ (-1)^{‖α_{s+1}‖ + … + ‖α_k‖} \#ν^1 (\angletop) $.
The Hochschild sign is $ (-1)^{1 + ‖α_{s+1}‖ + … + ‖α_k‖} $, rendering a total contribution to $ d ν_P $ of $ - \#ν^1 (\angletop) $.
\end{itemize}
It remains to check that all these terms cancel out, regardless of the length of $ α $ and the location of the start and end index $ s $. For this, we need a case distinction after the length of the angles involved. For example, by $ \angletop ≥ α $ we mean that the top angle of the garage sequence includes $ α $. We are now ready to summarize the contributions of terms 1–7 as follows:
\begin{itemize}
\item Term 1 always yields a contribution of $ \#ν^1 (\angletop) + \#ν^1 (\anglemid) + \#ν^1 (\anglebot) - \#ν^1 (ℓ^r) $.
\item If $ \angletop ≥ ℓ^r $, then 2+4 yield a total contribution of $ -\#ν^1 (\angletop) $.
\item If $ \anglebot ≥ α $, then 7 yields a total contribution of $ -\#ν^1 (\angletop) $.
\item If $ \angletop ≥ α $, then 6 yields a total contribution of $ \#ν^1 (ℓ^r) - \#ν^1 (\anglemid) - \#ν^1 (\anglebot) $.
\item If $ \anglebot ≥ ℓ^r $, then 3+5 yield a total contribution of $ \#ν^1 (ℓ^r) - \#ν^1 (\anglemid) - \#ν^1 (\anglebot) $.
\end{itemize}
Since $ α ℓ^r = \anglebot · \anglemid · \angletop $ and $ \anglemid $ consists of precisely one sector, we have that either $ \angletop ≥ ℓ^r $ or $ \anglebot ≥ α $ (but not both). Similarly, either $ \angletop ≥ α $ or $ \anglebot ≥ ℓ^r $. We conclude that either way, all contributions to $ d ν_P $ add up as
\begin{equation*}
\#ν^1 (\angletop) + \#ν^1 (\anglemid) + \#ν^1 (\anglebot) - \#ν^1 (ℓ^r)  - \#ν^1 (\angletop) + \#ν^1 (ℓ^r) - \#ν^1 (\anglemid) - \#ν^1 (\anglebot) = 0.
\end{equation*}

Let us now regard the two exceptions where the start is right before $ α $ or right after $ α $. The difference with the generic case is that there is no proper $ \anglemid $ sector. Let us regard the first exceptional case, where the sequence consists of the angles $ α, α_1, …, α_k $. Then there are no 2, 4 or 6 terms, since the top is basically empty. The bottom-most terms 3+5 contribute however with $ \#ν^1 (α) $ as in the generic case, and the bottom-most term 7 is special and contributes by $ \#ν^1 (ℓ^r) $. A special contribution of $ - \#ν^1 (ℓ^r) $ comes from $ μ(ν(α_k, …), …, α) $, the equivalent of the $ \anglemid $ term in the generic case of 5. Finally, the 1 term contributes $ \#ν^1 (α) $ as in the generic case. In total, these four terms add up to zero.

Let us regard the second exceptional case, where the sequence consists of the angles $ α_1, …, α_k, α $. Then there are no 3, 5 or 7 terms, since the bottom is empty. The top-most term 2 contributes $ - \#ν^1 (α) $, the other top 4 terms contribute $ - \#ν^1 (ℓ^r) $, and the top 6 term contributes $ \#ν^1 (ℓ^r) $. Together with the type 1 term $ \#ν^1 (α) $, this adds up to zero again.
\end{proof}

\begin{lemma}
We still have $ d ν_P $ on parking garage sequences when a $ γ $ is attached at the front and/or a $ β $ attached at the back.
\end{lemma}

\begin{proof}
Let us regard a garage sequence with additional $ γ $ at the end. Regard first the case where the final angle is one of $ α_1, …, α_{k-1} $. Then that final angle $ α_{s+1} $ changes to $ γ α_{s+1} $. We claim the effect on $ d ν_P $, applied to the garage sequence, is merely a multiplication by $ γ $. Essentially, this means checking that all old-era $ ν $ contributions stay old-era.

For example, let us check the first three types of terms explicitly. The contribution of term 1 gets only multiplied by $ γ $. For term 2, this is also true, since the final angle is assumed not to be $ α $, and the outer $ μ $ itself gets multiplied by $ γ $. This holds likewise for term 3, except in the case when there are no angles to the left of the inner $ ν $, where the term looks like $ μ(ν(…, α_1), α, …) $. Since the final angle is not $ α_k $, the inner $ ν $ was of new-era type. Prolonging the final angle naturally preserves the new-era type of the $ ν $ evaluation.

Now let us regard the two special cases where the final angle is either $ α $ or $ α_k $. We start with the case where the final angle is $ α $. The garage is then the sequence $ α_1, …, α_k, γ α $. Inspecting all terms 1–7, most terms just get multiplied by $ γ $, but we also incur the following changes:
\begin{itemize}
\item $ μ(ν^1 (γ α), α_k, …, α_1) $ now contributes $ \#ν^1 (γ) $ extra,
\item $ ν(μ(γα, α_k, …), …, α_1) $ keeps contributing as long as $ γ < ℓ^r $, but the rotation amount of the outer $ ν $ now decreases from $ ℓ^r $ to $ ℓ^r γ^{-1} $. This means it contributes $ \#ν^1 (γ) $ less. When $ γ ≥ ℓ^r $, it stops contributing entirely, meaning it contributes $ \#ν^1 (ℓ^r) $ less than in the case without $ γ $.
\item The new term $ ν^1 (μ(γα, α_k, …, α_1)) $ suddenly starts contributing when $ γ > ℓ^r $, namely by $ \#ν^1 (γ ℓ^{-r}) $. The Hochschild sign is $ -1 $.
\end{itemize}
Adding up these extra contributions, the total remains precisely the same as in the case without $ γ $, for $ γ < ℓ^r $ as well as $ γ ≥ ℓ^r $.

Now regard the case that $ α_k $ is the final angle of the garage sequence. The sequence is then $ α, α_1, …, γ α_k $. We have the following contributions:
\begin{itemize}
\item[1.] $ μ(…, ν^1 (α)) $ \par\noindent This simply gets multiplied by $ γ $ and the result is $ \#ν^1 (α) γ $. The Hochschild sign is $ +1 $.
\item[2'.] $ μ(ν(γ α_k, …), …, α) $ \par\noindent While the generic case term 2 has inner $ ν $ of old-era type, this pendant is new-era and yields $ - \#ν^1 (\anglenext) γ $, where $ n $ is the next sector after the top of the garage. The Hochschild sign is $ +1 $.
\item[3.] $ μ(γ α_k, …, ν(…, α_1), α) $ \par\noindent In no case is it possible to take an inner $ ν $ that includes all angles from $ α_1 $ to $ α_k $, since these angles cover strictly more than $ r $ turns around $ m $. Therefore $ γ α_k $ lies outside of the inner $ ν $, which thereby remains old-era. The outer $ μ $ simply gets multiplied by $ γ $ and gives $ - \#ν^1 (\anglefirst) γ $. The Hochschild sign is $ +1 $.
\item[4.] Terms of type 4 do not appear when $ α_k $ is the final angle.
\item[5.] $ μ(…, ν(…), \underbrace{…}_{≥1}, α) $ \par\noindent We do not count $ μ(ν(γ α_k, …), …, α) $ among these terms, since we already attributed it to 2'. Then, all that changes for these type 5 terms is that they get multiplied by $ γ $. They add up to $ \big(- \#ν^1 (α) + \#ν^1 (\anglefirst)\big) γ $. The Hochschild sign is $ +1 $.
\item[6.] The term of type 6 does not appear when $ α_k $ is the final angle.
\item[7.] $ ν(…, μ(α_t, …, α_1, α)) $ \par\noindent Since the angles $ α_1 $, …, $ α_k $ cover strictly more than the angle $ α $ does, we have $ t < k $. In other words, $ ν $ was old-era and becomes new-era. Its new value is $ - \#ν^1 (\anglenext) γ $. The Hochschild sign is $ -1 $.
\end{itemize}
In total, this adds up to zero again. Let us finally comment on the changes we incur once we add $ β $ at the back, in addition to a possible $ γ $. Both $ μ^{≥3} $ and $ ν^{≥2} $ are “equivariant” under appending $ β $ at the back. It remains to check the cases where $ ν^1 $ is involved, and check for longer terms appearing because $ α $ gets longer. If one of $ α_1 $, …, $ α_k $ is the angle in the back of the sequence, it is readily checked that all terms simply get multiplied by $ β $. If $ α $ is the angle in the back, we incur the following changes:
\begin{itemize}
\item $ ν(…, μ(…, αβ)) $ still contributes as long as $ β < ℓ^r $, however the sequence the inner $ μ $ is applied to becomes longer and longer, and similarly the magic angle of the outer $ ν $ becomes shorter and shorter. We lose $ - \#ν^1 (β) $ as coefficient. If $ β ≥ ℓ^r $, the term does not contribute anymore at all, and we have lost $ - \#ν^1 (ℓ^r) $, compared to the sequence without $ β $. The Hochschild sign is $ -1 $. All signs together, we deduce an extra contribution of $ - \#ν^1 (β) $ or $ - \#ν^1 (ℓ^r) $.
\item $ ν^1 (μ(α_k, …, α_1, αβ)) $ starts to contribute once $ β > ℓ^r $, namely by $ \#ν^1 (β ℓ^{-r}) $. The Hochschild sign is $ -1 $.
\item $ μ(α_k, …, α_1, ν^1 (αβ)) $ contributes an extra $ \#ν^1 (β) $. The Hochschild sign is $ +1 $.
\end{itemize}
Whether $ β ≤ ℓ^r $ or $ β > ℓ^r $, we conclude the additional contribution vanishes, compared to the case without $ β $.

Finally, when both non-empty $ β $ and $ γ $ are appended, we conclude the result is a multiple of $ γβ $, which vanishes. Indeed, pick two consecutive angles around the garage sequence. Then appending an angle $ γ $ behind the first and an angle $ β $ behind the second necessarily makes $ β $ and $ γ $ incomposable.
\end{proof}

\subsection{Cancellation on other sequences}
\label{sec:even-other}
In the \autoref{sec:even-parking}, we checked that $ dν $ vanishes on parking garage sequences. Here $ ν $ is an even Hochschild cochain constructed in \autoref{sec:even-construction} from the input data $ m ∈ M $, $ r ≥ 1 $ and input scalars $ \#ν^1 (α) $. In the present section, we check that $ d ν $ also vanishes on all other sequences of angles. The procedure is as follows: Pick a sequence $ α_1, …, α_k $ of angles and evaluate $ d ν (α_k, …, α_1) $. This gives a collection of terms of the form $ μ(…, ν(…), …) $ and $ ν(…, μ(…), …) $. We show how to partition this collection of terms such that within each partition, the terms cancel each other. In contrast to the odd case, the partitions do not always consist of two, but sometimes also of three terms.

Of course, we cannot handle each individual sequence of angles individually, but rather need to classify sequences according to their shape. Most importantly, we distinguish the shapes according to the types of $ μ $ and $ ν $ that can be applied and the number of angles before and after the inner application. This way, we can partition all possible sequences of angles $ α_1, …, α_k $ and terms appearing in $ dν (α_k, …, α_1) $ in bulk format: Each of the partitions we provide makes reference to a particular shape.

In total, this procedure requires considerable case-checking effort, namely
\begin{itemize}
\item[1.] listing all partitions,
\item[2.] characterizing for each partition the required sequence shape,
\item[3.] proving that each partition sums up to zero,
\item[4.] mapping each term in $ d ν (…) $ to one partition,
\item[5.] proving all terms in all partitions are hit at most once,
\item[6.] proving all terms in all partitions are hit at least once.
\end{itemize}
We do not conduct all steps rigorously. In fact, we concentrate on 1, 2, 3, 4, but without rigor. Below, we list all partitions, ordered roughly after the type of sequence involved. Typically, such a sequence winds once around a certain area and then around another, possibly the one being nested in the other. We indicate the type of these areas as “disk”, “$ < ℓ^r $” or “$ ℓ^r $”. In the figures, the thick dot indicates the location of $ m $ and the grey rings indicate the magic angles of the $ ν(…) $ for every involved cancelling term.

\begin{remark}
The list indicates clearly that the given pair or triple of terms cancels. To see this, recall that $ ν(…) $ is by definition weighted with the input scalar $ \#ν^1 $ of its magic angle. In order to make the claimed pairs or triples of terms cancel each other, we need to show that a signed sum of input scalars of the magic angles is zero. Since the input scalars $ \#ν^1 $ are additive on angles, this amounts to checking that every indecomposable angle around $ m $ appears in an even number of magic angles (ignoring signs). The reader can easily convince himself that this is the case by looking at the grey rings around $ m $ in every figure: Every ray away from $ m $ hits an even number of grey rings.
\end{remark}

\input{even/fig_partitions.tex}

We now investigate all possible terms in $ d ν (…) $. For every term, we provide the other terms with which it cancels. These pairs or triples can be found back in the partition list above. A few possible terms are omitted due to analogy with other terms, and we have correspondingly not listed them in the partition table either. In \autoref{th:even-other-conclusion}, we draw the conclusion that $ dν = 0 $.

\begin{lemma}
\label{th:hochschild-construction-mu-nu-end-small-mid}
A contribution $ μ^{≥3} (…, ν^{≥2} (…)) $ with $ ν $ end-split with turning angle $ < ℓ^r $ cancels.
\end{lemma}

\begin{proof}
Since we assumed the inner $ ν $ to be end-split with turning angle $ < ℓ^r $, its result is some angle $ α $ winding around $ m $. Then $ α $ forms a disk sequence together with $ α_{m+1}, …, α_k $. In order to find the terms canceling the contribution, we distinguish the following cases: (a) The angle $ α $ is an ordinary interior angle for the outer $ μ $, and the outer $ μ $ is all-in. (b) The angle $ α $ is an ordinary interior angle for the outer $ μ $, the outer $ μ $ is final-out, and the result part of $ α_k $ is shorter than $ α_1 $. (c) The angle $ α $ is an ordinary interior angle for the outer $ μ $, the outer $ μ $ is final-out, and the result part of $ α_k $ is longer than $ α_1 $. (d) The angle $ α $ is a first-out angle. See \autoref{fig:hochschild-construction-mu-nu-end-small-mid}.

Regard case (a). Then we have the triple cancellation
\begin{equation*}
μ^{≥3} (…, ν^{≥2} (α_m, …, α_1)) + μ^{≥3} (ν^{≥2} (α_k, …, α_{m+1}), α_m, …, α_1) + ν^{≥2} (…, μ^2 (α_{m+1}, α_m), …). 
\end{equation*}

Regard case (b). Then $ α_k $ includes a $ γ $ at the front, but is short. It splits the input sequence of $ ν $ into a “small” disk sequence and a remaining “big” part. We get the cancellation
\begin{equation*}
μ^{≥3} (…, ν^{≥2} (α_m, …, α_1)) + μ^{≥3} (ν^{≥2} (α_k, …, α_t), …, α_1) + ν^{≥2} (…, μ^2 (α_{m+1}, α_m), …).
\end{equation*}

Regard case (c). Then the long $ α_k $ angle makes it possible to create another inner $ ν $, the third one in the following cancellation:
\begin{equation*}
μ^{≥3} (…, ν^{≥2} (α_m, …, α_1)) + μ^2 (ν^{≥2} (α_k, …), α_1) + ν^{≥2} (…, μ^2 (α_{m+1}, α_m), …). 
\end{equation*}

Regard case (d). Then we can apply $ ν $ to both disks. Their sum gets compensated by combining the two disks:
\begin{equation*}
μ^{≥3} (…, ν^{≥2} (α_m, …, α_1)) + μ^{≥3} (ν^{≥2} (α_k, …, α_{m+1}), α_m, …, α_1) + ν^{≥2} (…, μ^2 (α_{m+1}, α_m), …). 
\end{equation*}
\end{proof}

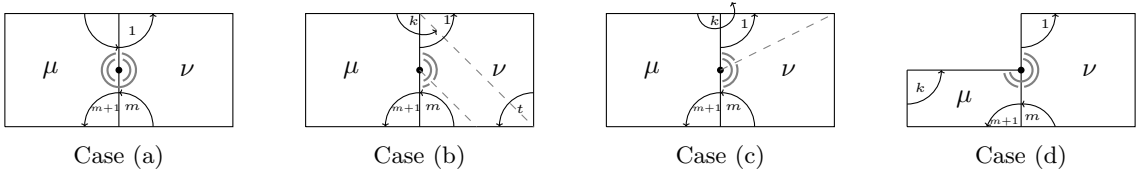
\begin{figure}
\centering
\begin{subfigure}{0.24\linewidth}
\centering
\begin{tikzpicture}[scale=1.5]
\path[use as bounding box] (0, 0) -- (2, 1);
\path[draw] (0, 0) -- ++(right:1) coordinate (bot) coordinate[pos=0.7] (2-end) -- ++(right:1) coordinate[pos=0.3] (4-start) -- ++(up:1) -- ++(left:1) coordinate (top) coordinate[pos=0.7] (3-end) -- ++(left:1) coordinate[pos=0.3] (1-start) -- ++(down:1);
\path[draw] (top) -- (bot) coordinate[pos=0.35] (A) coordinate[pos=0.65] (B) coordinate[midway] (q) coordinate[pos=0.3] (1-end) coordinate[pos=0.7] (2-start);
\path[fill] (q) circle[radius=0.03];
\path[draw, -{To[scale=0.5]}, bend right=45] (1-start) to (1-end);
\path[draw, -{To[scale=0.5]}, bend right=45] (2-start) to node[near start, below] {\scalebox{0.75}{$ \scriptscriptstyle m+1 $}} (2-end);
\path[draw, -{To[scale=0.5]}, bend right=45] (1-end) to node[near start, above] {\tiny $ 1 $} (3-end);
\path[draw, -{To[scale=0.5]}, bend right=45] (4-start) to node[near end, below] {\tiny $ m $} (2-start);
\path[draw, thick, gray] (q) ++(110:0.1) arc(110:250:0.1);
\path[draw, thick, gray] (q) ++(290:0.1) arc(290:430:0.1);
\path[draw, thick, gray] (q) ++(100:0.15) arc(100:440:0.15);
\path (0.4, 0.5) node {$ μ $};
\path (1.6, 0.5) node {$ ν $};
\end{tikzpicture}
\caption*{Case (a)}
\end{subfigure}
\begin{subfigure}{0.24\linewidth}
\centering
\begin{tikzpicture}[scale=1.5]
\path[use as bounding box] (0, 0) -- (2, 1);
\path[draw] (0, 0) -- ++(right:1) coordinate (bot) coordinate[pos=0.7] (2-end) -- ++(right:1) coordinate[pos=0.3] (4-start) coordinate[pos=0.7] (5-end) -- ++(up:1) coordinate[pos=0.3] (5-start) -- ++(left:1) coordinate (top) coordinate[pos=0.7] (3-end) -- ++(left:1) coordinate[pos=0.2] (1-start) -- ++(down:1);
\path[draw] (top) -- (bot) coordinate[pos=0.35] (A) coordinate[pos=0.65] (B) coordinate[midway] (q) coordinate[pos=0.2] (1-end) coordinate[pos=0.7] (2-start) coordinate[pos=0.3] (3-start);
\path[fill] (q) circle[radius=0.03];
\path[draw, dashed, gray] (top) -- (2, 0) coordinate[pos=0.15] (1-end-new);
\path[draw, dashed, gray] (q) -- (1.5, 0);
\path[draw, -{To[scale=0.5]}, bend right=60] (1-start) to node[pos=0.15, right] {\tiny $ k $} (1-end-new);
\path[draw, -{To[scale=0.5]}, bend right=45] (2-start) to node[near start, below] {\scalebox{0.75}{$ \scriptscriptstyle m+1 $}} (2-end);
\path[draw, -{To[scale=0.5]}, bend right=45] (3-start) to node[pos=0.6, above] {\tiny $ 1 $} (3-end);
\path[draw, -{To[scale=0.5]}, bend right=45] (4-start) to node[near end, below] {\tiny $ m $} (2-start);
\path[draw, -{To[scale=0.5]}, bend right=45] (5-start) to node[near start, below] {\tiny $ t $} (5-end);
\path[draw, thick, gray] (q) ++(290:0.1) arc(290:430:0.1);
\path[draw, thick, gray] (q) ++(280:0.15) arc(280:305:0.15);
\path[draw, thick, gray] (q) ++(325:0.15) arc(325:440:0.15);
\path (0.4, 0.5) node {$ μ $};
\path (1.7, 0.5) node {$ ν $};
\end{tikzpicture}
\caption*{Case (b)}
\end{subfigure}
\begin{subfigure}{0.24\linewidth}
\centering
\begin{tikzpicture}[scale=1.5]
\path[use as bounding box] (0, 0) -- (2, 1);
\path[draw] (0, 0) -- ++(right:1) coordinate (bot) coordinate[pos=0.7] (2-end) -- ++(right:1) coordinate[pos=0.3] (4-start) coordinate[pos=0.7] (5-end) -- ++(up:1) coordinate[pos=0.3] (5-start) -- ++(left:1) coordinate (top) coordinate[pos=0.7] (3-end) -- ++(left:1) coordinate[pos=0.2] (1-start) -- ++(down:1);
\path[draw] (top) -- (bot) coordinate[pos=0.35] (A) coordinate[pos=0.65] (B) coordinate[midway] (q) coordinate[pos=0.2] (1-end) coordinate[pos=0.7] (2-start) coordinate[pos=0.3] (3-start);
\path[fill] (q) circle[radius=0.03];
\path[draw, dashed, gray] (q) -- (2, 1);
\path[draw, -{To[scale=0.5]}, bend right=100, looseness=2] (1-start) to node[pos=0.1, right] {\tiny $ k $} ++(0.3, 0.1);
\path[draw, -{To[scale=0.5]}, bend right=45] (2-start) to node[near start, below] {\scalebox{0.75}{$ \scriptscriptstyle m+1 $}} (2-end);
\path[draw, -{To[scale=0.5]}, bend right=45] (3-start) to node[pos=0.5, above] {\tiny $ 1 $} (3-end);
\path[draw, -{To[scale=0.5]}, bend right=45] (4-start) to node[near end, below] {\tiny $ m $} (2-start);
\path[draw, thick, gray] (q) ++(280:0.1) arc(280:440:0.1);
\path[draw, thick, gray] (q) ++(280:0.15) arc(280:380:0.15);
\path[draw, thick, gray] (q) ++(395:0.15) arc(395:445:0.15);
\path (0.4, 0.5) node {$ μ $};
\path (1.6, 0.5) node {$ ν $};
\end{tikzpicture}
\caption*{Case (c)}
\end{subfigure}
\begin{subfigure}{0.24\linewidth}
\centering
\begin{tikzpicture}[scale=1.5]
\path[use as bounding box] (0, 0) -- (2, 1);
\path[draw] (0, 0) -- ++(right:1) coordinate (bot) coordinate[pos=0.7] (2-end) -- ++(right:1) coordinate[pos=0.3] (4-start) -- ++(up:1) -- ++(left:1) coordinate (top) coordinate[pos=0.7] (3-end) ++(left:1) coordinate[pos=0.2] (1-start) ++(down:0.5) -- ++(down:0.5) coordinate[pos=0.6] (5-start);
\path[draw] (top) -- (bot) coordinate[pos=0.35] (A) coordinate[pos=0.65] (B) coordinate[midway] (q) coordinate[pos=0.2] (1-end) coordinate[pos=0.8] (2-start) coordinate[pos=0.3] (3-start);
\path[draw] (q) -- (0, 0.5) coordinate[pos=0.7] (5-end);
\path[fill] (q) circle[radius=0.03];
\path[draw, -{To[scale=0.5]}, bend right=35] (2-start) to node[pos=0.4, below] {\scalebox{0.75}{$ \scriptscriptstyle m+1 $}} (2-end);
\path[draw, -{To[scale=0.5]}, bend right=45] (3-start) to node[pos=0.5, above] {\tiny $ 1 $} (3-end);
\path[draw, -{To[scale=0.5]}, bend right=35] (4-start) to node[near end, below] {\tiny $ m $} (2-start);
\path[draw, -{To[scale=0.5]}, bend right=45] (5-start) to node[near start, above] {\tiny $ k $} (5-end);
\path[draw, thick, gray] (q) ++(-170:0.1) arc(-170:80:0.1);
\path[draw, thick, gray] (q) ++(-170:0.15) arc(-170:-100:0.15);
\path[draw, thick, gray] (q) ++(-80:0.15) arc(-80:80:0.15);
\path (0.5, 0.25) node {$ μ $};
\path (1.6, 0.5) node {$ ν $};
\end{tikzpicture}
\caption*{Case (d)}
\end{subfigure}
\caption{Cancellation for \autoref{th:hochschild-construction-mu-nu-end-small-mid}. Magic angles of each contribution to $ dν $ are drawn as gray rings around $ q $. From these figures one deduces cancellation: Each sector around $ q $ appears exactly an even number of times, here 0 or 2. Inspection shows that overlapping contributions indeed appear with opposite sign.}
\label{fig:hochschild-construction-mu-nu-end-small-mid}
\end{figure}

\begin{lemma}
\label{th:hochschild-construction-mu-nu-end-small-start}
A contribution $ μ^{≥3} (\underbrace{…}_{≥1}, ν^{≥2} (…), \underbrace{…}_{≥1}) $ with $ ν $ end-split with turning angle $ < ℓ^r $ cancels.
\end{lemma}

\begin{proof}
Write $ μ^{≥3} (α_k, …, ν^{≥2} (α_m, …), α_n, …, α_1) $. See \autoref{fig:hochschild-construction-mu-nu-end-small-start-end}. We have a triple cancellation
\begin{equation*}
μ^{≥3} (…, ν^{≥2} (…), …) + ν^{≥2} (…, μ^2 (α_{m+1}, α_m), …) + ν^{≥2} (…, μ^2 (α_{n+1}, α_n), …). 
\end{equation*}
\end{proof}

\begin{figure}
\centering
\begin{subfigure}{0.24\linewidth}
\centering
\begin{tikzpicture}[scale=1.5]
\path[use as bounding box] (0, 0) -- (2, 1);
\path[draw] (0, 0) -- ++(right:1) coordinate (bot) coordinate[pos=0.7] (2-end) -- ++(right:1) coordinate[pos=0.3] (4-start) -- ++(up:1) -- ++(left:1) coordinate (top) coordinate[pos=0.7] (3-end) -- ++(left:1) coordinate[pos=0.3] (1-start) -- ++(down:1) coordinate[pos=0.65] (5-start) coordinate[pos=0.35] (5-end);
\path[draw] (top) -- (bot) coordinate[pos=0.35] (A) coordinate[pos=0.65] (B) coordinate[midway] (q) coordinate[pos=0.3] (1-end) coordinate[pos=0.7] (2-start);
\path[fill] (q) circle[radius=0.03];
\path[draw, gray, dashed] (q) -- (0, 1);
\path[draw, -{To[scale=0.5]}, bend right=90, looseness=2] (5-start) to node[midway, left] {\tiny $ 1 $} (5-end);
\path[draw, -{To[scale=0.5]}, bend right=45] (1-start) to node[near end, above] {\tiny $ n $} (1-end);
\path[draw, -{To[scale=0.5]}, bend right=45] (2-start) to node[near start, below] {\scalebox{0.75}{$ \scriptscriptstyle m+1 $}} (2-end);
\path[draw, -{To[scale=0.5]}, bend right=45] (1-end) to node[near start, above] {\scalebox{0.75}{$ \scriptscriptstyle n+1 $}} (3-end);
\path[draw, -{To[scale=0.5]}, bend right=45] (4-start) to node[near end, below] {\tiny $ m $} (2-start);
\path[draw, thick, gray] (q) ++(-80:0.1) arc(-80:140:0.1);
\path[draw, thick, gray] (q) ++(-85:0.15) arc(-85:80:0.15);
\path[draw, thick, gray] (q) ++(100:0.15) arc(100:145:0.15);
\path (0.4, 0.5) node {$ μ $};
\path (1.6, 0.5) node {$ ν $};
\end{tikzpicture}
\caption*{\autoref{th:hochschild-construction-mu-nu-end-small-start}}
\end{subfigure}
\begin{subfigure}{0.24\linewidth}
\centering
\begin{tikzpicture}[scale=1.5]
\path[use as bounding box] (0, 0) -- (2, 1);
\path[draw] (0, 0) -- ++(right:1) coordinate (bot) coordinate[pos=0.8] (2-end) -- ++(right:1) coordinate[pos=0.3] (4-start) -- ++(up:1) -- ++(left:1) coordinate (top) coordinate[pos=0.7] (3-end) -- ++(left:1) coordinate[pos=0.3] (1-start) -- ++(down:1);
\path[draw] (top) -- (bot) coordinate[pos=0.35] (A) coordinate[pos=0.65] (B) coordinate[midway] (q) coordinate[pos=0.3] (1-end) coordinate[pos=0.7] (2-start);
\path[fill] (q) circle[radius=0.03];
\path[draw, -{To[scale=0.5]}, bend right=45] (1-start) to node[near end, above] {\tiny $ m $} (1-end);
\path[draw, -{To[scale=0.5]}, bend right=100, looseness=2.5] (2-end) ++(0.25, -0.1) to node[pos=0.3, left] {\tiny $ 1 $} (2-end);
\path[draw, -{To[scale=0.5]}, bend right=45] (1-end) to node[near start, above] {\scalebox{0.75}{$ \scriptscriptstyle m+1 $}} (3-end);
\path[draw, -{To[scale=0.5]}, bend right=45] (4-start) to node[pos=0.5, below] {\tiny $ k $} (2-start);
\path[draw, thick, gray] (q) ++(110:0.1) arc(110:250:0.1);
\path[draw, thick, gray] (q) ++(290:0.1) arc(290:430:0.1);
\path[draw, thick, gray] (q) ++(-80:0.15) arc(-80:260:0.15);
\path (0.4, 0.5) node {$ μ $};
\path (1.6, 0.5) node {$ ν $};
\end{tikzpicture}
\caption*{\autoref{th:hochschild-construction-mu-nu-end-small-end}: Case (a)}
\end{subfigure}
\begin{subfigure}{0.24\linewidth}
\centering
\begin{tikzpicture}[scale=1.5]
\path[use as bounding box] (0, 0) -- (2, 1);
\path[draw] (0, 0) ++(right:1) coordinate (bot) -- ++(right:1) coordinate[pos=0.3] (4-start) -- ++(up:1) -- ++(left:1) coordinate (top) coordinate[pos=0.7] (3-end) -- ++(left:1) coordinate[pos=0.3] (1-start) -- ++(down:0.5) coordinate[pos=0.4] (5-end);
\path[draw] (top) -- (bot) coordinate[pos=0.35] (A) coordinate[pos=0.65] (B) coordinate[midway] (q) coordinate[pos=0.3] (1-end) coordinate[pos=0.7] (2-start);
\path[fill] (q) circle[radius=0.03];
\path[draw] (q) -- (0, 0.5) coordinate[pos=0.7] (5-start);
\path[draw, -{To[scale=0.5]}, bend right=45] (5-start) to node[near end, below] {\tiny $ 1 $} (5-end);
\path[draw, -{To[scale=0.5]}, bend right=45] (1-start) to node[near end, above] {\tiny $ m $} (1-end);
\path[draw, -{To[scale=0.5]}, bend right=45] (1-end) to node[near start, above] {\scalebox{0.75}{$ \scriptscriptstyle m+1 $}} (3-end);
\path[draw, -{To[scale=0.5]}, bend right=45] (4-start) to node[near end, below] {\tiny $ k $} (2-start);
\path[draw, thick, gray] (q) ++(-80:0.1) arc(-80:170:0.1);
\path[draw, thick, gray] (q) ++(-80:0.15) arc(-80:80:0.15);
\path[draw, thick, gray] (q) ++(100:0.15) arc(100:170:0.15);
\path (0.4, 0.7) node {$ μ $};
\path (1.6, 0.5) node {$ ν $};
\end{tikzpicture}
\caption*{Case (b)}
\end{subfigure}
\caption{Cancellation for \autoref{th:hochschild-construction-mu-nu-end-small-start} and \autoref{th:hochschild-construction-mu-nu-end-small-end}}
\label{fig:hochschild-construction-mu-nu-end-small-start-end}
\end{figure}

\begin{lemma}
\label{th:hochschild-construction-mu-nu-end-small-end}
A contribution $ μ^{≥3} (ν^{≥2} (…), …) $ with $ ν $ end-split with turning angle $ < ℓ^r $ cancels.
\end{lemma}

\begin{proof}
Since we assumed the inner $ ν $ to be end-split with turning angle $ < ℓ^r $, its result is some angle $ α $ winding around $ m $. Then $ α $ forms a disk sequence together with $ α_1, …, α_m $. By assumption, $ α $ is necessarily an ordinary or the final out angle for this disk. In order to find the terms canceling the contribution, we distinguish the following cases: (a) The angle $ α $ is an ordinary interior angle for the outer disk sequence, and $ α_1 $ reaches outside the orbigon at its tail. (b) The angle $ α $ is a final-out angle. (c) The angle $ α $ is an ordinary interior angle, and $ α_1 $ does not reach outside.

Regard case (a), where $ α $ is an ordinary interior angle of the outer $ μ $. Then we have the triple cancellation
\begin{equation*}
μ^{≥3} (ν^{≥2} (α_k, …, α_{m+1}), α_m, …, α_1) + μ^2 (α_k, ν(α_{k-1}, …, α_1)) + ν(…, μ^2 (α_{m+1}, α_m), …).
\end{equation*}
Regard case (b), where $ α $ is a final-out angle of the outer $ μ $. In particular, the first angle $ α_1 $ of the outer $ μ $ has no $ β $ appended. We have the triple cancellation
\begin{equation*}
μ^{≥3} (ν^{≥2} (α_k, …, α_{m+1}), …) + μ^{≥3} (…, ν^{≥2} (α_m, …, α_1)) + ν^{≥2} (…, μ^2 (α_{m+1}, α_m), …). 
\end{equation*}
Regard case (c). Then we have the cancellation
\begin{equation*}
μ(ν(α_k, …, α_{m+1}), …) + μ(…, ν(α_t, …, α_1)) + ν(…, μ^2 (α_{m+1}, α_m), …).
\end{equation*}
\end{proof}

\begin{lemma}
\label{th:hochschild-construction-mu-nu-end-long}
A contribution $ μ^{≥3} (…, ν^{≥2} (…), …) $ with $ ν $ old- or new-era end-split with turning angle $ ℓ^r $ cancels.
\end{lemma}

\begin{proof}
We distinguish cases: (a) The $ ν $ is old-era and is not the final angle in $ μ $. (b) The $ ν $ is old-era and is the final angle in $ μ $, and $ μ $ is all-in or first-out. (c) The $ ν $ is old-era and is the final angle in $ μ $, and $ μ $ is final-out. (d) The $ ν $ is new-era and has $ β $ appended. (e) The $ ν $ is new-era without $ β $.

Regard case (a). Label the angles as $ μ^{≥3} (α_k, …, ν^{≥2} (α_m, …, α_{n+1}), …) $. Then the next angle $ α_{m+1} $ after $ ν $ winds around $ m $. Prolonging it by $ ℓ^r $ gives precisely a disk sequence, in other words the sequence is a parking garage sequence. 

Regard case (b). Then $ α_1 $ is the first angle of $ μ $ and winds around $ m $. Prolonging it by $ ℓ^r $ gives a disk sequence and we have a parking garage again.

Regard case (c). Label the angles as $ μ^{≥3} (ν^{≥2} (α_k, …, α_m), …, α_1) $. Then we can swap the order in which $ μ $ and $ ν $ are applied:
\begin{equation*}
μ^{≥3} (ν^{≥2} (α_k, …, α_m), …, α_1) + ν^{≥2} (α_k, …, μ^{≥3} (α_m, …, α_1)).
\end{equation*}

Regard case (d). Then the split of $ ν $ necessarily divides the outer $ μ $ disk into two. The angle $ α_t $ where the split touches the opposite boundary of the $ μ $ disk creates a contribution $ μ^{≥3} (…, ν^1 (α_t), …) $. We have a garage sequence.

Case (e) is similar to the combination of (a), (b) and (c): If there is an angle before $ ν $, then we can apply $ ν^1 $ and have a garage sequence. If there is no angle before $ ν $, we have a contribution $ μ^{≥3} (…, ν^{≥2} (…)) $. It this is an all-in or final-out $ μ $, then the final angle $ α_k $ winds around $ m $ and we have a parking garage. If this is a first-out $ μ $, then we can swap the order of $ μ $ and $ ν $ again.
\end{proof}

\begin{figure}
\centering
\begin{subfigure}{0.24\linewidth}
\centering
\begin{tikzpicture}[scale=1.5]
\path[draw] (0, 0) -- ++(right:1) coordinate[pos=0.4] (4-start) -- ++(up:2) -- ++(left:2) -- ++(down:2) -- ++(right:1) coordinate[pos=0.6] (1-end);
\path[draw] (0, 0) -- (0, 1) coordinate (q) coordinate[pos=0.7] (2-start) coordinate[pos=0.4] (4-end);
\path[fill] (q) circle[radius=0.05];
\path[draw, dashed, gray] (q) -- (1, 2) coordinate[pos=0.35] (2-end);
\path[draw, dashed, gray] (0, 0) -- (1, 1) coordinate[pos=0.2] (1-start);
\path (1, 2) -- (1, 1) node[midway, sloped, above] {…};
\path[draw, -{To[scale=0.5]}, bend right=60] (1-start) to node[midway, below] {\scalebox{1}{$ \scriptscriptstyle n+1 $}} (1-end);
\path[draw, -{To[scale=0.5]}, bend right=65, looseness=1.5] (2-start) to node[pos=0.2, above] {\scalebox{1}{$ \scriptscriptstyle m+1 $}} (2-end);
\path[draw, -{To[scale=0.5]}, bend right=45] (4-start) to node[pos=0.45, below] {\tiny $ m $} (4-end);
\path[bend left, decorate, decoration={text along path, text align=center, text={|\tiny|parking garage}}] ($ (q) + (-0.5, -0.8) $) to ($ (q) + (0.3, 0.8) $);
\path (0.7, 1.1) node {$ μ $};
\path (-0.7, 1.6) node {$ ν $};
\end{tikzpicture}
\caption*{Case (a)}
\end{subfigure}
\begin{subfigure}{0.24\linewidth}
\centering
\begin{tikzpicture}[scale=1.5]
\path[draw] (0, 0) -- ++(right:1) coordinate[pos=0.4] (4-start) -- ++(up:2) -- ++(left:2) -- ++(down:2) -- ++(right:1) coordinate[pos=0.6] (1-end);
\path[draw] (0, 0) -- (0, 1) coordinate (q) coordinate[pos=0.7] (2-start) coordinate[pos=0.4] (4-end);
\path[fill] (q) circle[radius=0.05];
\path[draw, dashed, gray] (q) -- (1, 2) coordinate[pos=0.25] (2-end);
\path[draw, dashed, gray] (0, 0) -- (1, 1) coordinate[pos=0.2] (1-start);
\path (1, 2) -- (1, 1) node[midway, sloped, above] {…};
\path[draw, -{To[scale=0.5]}, bend right=60] (1-start) to node[midway, below] {\scalebox{1}{$ \scriptscriptstyle n+1 $}} (1-end);
\path[draw, -{To[scale=0.5]}, bend right=80, looseness=1] ($ (2-start) + (-0.2, -0.1) $) to node[midway, above] {\tiny $ 1 $} (2-end);
\path[draw, -{To[scale=0.5]}, bend right=45] (4-start) to node[pos=0.45, below] {\tiny $ k $} (4-end);
\path[bend left, decorate, decoration={text along path, text align=center, text={|\tiny|parking garage}}] ($ (q) + (-0.5, -0.8) $) to ($ (q) + (0.3, 0.8) $);
\path (0.6, 1.1) node {$ μ $};
\path (-0.7, 1.6) node {$ ν $};
\end{tikzpicture}
\caption*{Case (b)}
\end{subfigure}
\begin{subfigure}{0.24\linewidth}
\centering
\begin{tikzpicture}[scale=1.5]
\path[draw] (0, 0) -- ++(right:1) coordinate[pos=0.4] (4-start) -- ++(up:2) -- ++(left:2) -- ++(down:2) -- ++(right:1) coordinate[pos=0.6] (1-end);
\path[draw] (0, 0) -- (0, 1) coordinate (q) coordinate[pos=0.4] (4-end);
\path[fill] (q) circle[radius=0.05];
\path[draw, dashed, gray] (0, 0) -- (1, 2) coordinate[pos=0.25] (2-end) coordinate[pos=0.85] (2-start);
\path[draw, dashed, gray] (0, 0) -- (1, 0.5) coordinate[pos=0.2] (1-start) coordinate[pos=0.8] (3-end);
\path (1, 2) -- (1, 0.5) node[midway, sloped, above] {…};
\path[draw, -{To[scale=0.75]}, bend right=60, looseness=1.5] (1-start) to node[pos=0.55, below] {$ \scriptscriptstyle n+1 $} (1-end);
\path[draw, -{To[scale=0.75]}, bend right=80, looseness=1] (2-start) to node[midway, above] {\small $ 1 $} ($ (2-start) + (0.4, 0.2) $);
\path[draw, -{To[scale=0.75]}, bend right=45] (4-start) to node[pos=0.45, below] {\small $ k $} (4-end);
\path[draw, -{To[scale=0.75]}, bend right=60] ($ (3-end) + (0.3, 0.3) $) to node[pos=0.45, below] {\small $ n $} (3-end);
\path[draw, thick, gray] (q) ++(-80:0.1) arc(-80:260:0.1);
\path[draw, thick, gray] (q) ++(-80:0.15) arc(-80:260:0.15);
\path (0.6, 0.7) node {$ μ $};
\path (-0.5, 1.2) node {$ ν $};
\end{tikzpicture}
\caption*{Case (c)}
\end{subfigure}
\begin{subfigure}{0.24\linewidth}
\centering
\begin{tikzpicture}[scale=1.5]
\path[draw] (0, 0) -- ++(right:1) coordinate[pos=0.4] (4-start) -- ++(up:2) -- ++(left:2) -- ++(down:2) -- ++(right:1) coordinate[pos=0.6] (1-end);
\path[draw] (0, 0) -- (0, 1) coordinate (q) coordinate[pos=0.4] (4-end);
\path[fill] (q) circle[radius=0.05];
\path[draw, dashed, gray] (0, 0) -- (1, 1) coordinate[pos=0.2] (1-start);
\path[draw, dashed, gray] (0, 0) -- (-1, 1) coordinate[pos=0.3] (4-end);
\path[draw, dashed, gray] ($ (q) + (-0.5, 0.1) $) -- (q) coordinate[pos=0.5] (5-start) -- ($ (q) + (0.5, 0.1) $) coordinate[pos=0.5] (5-end);
\path[bend right, decorate, decoration={text along path, text align=center, text={|\tiny|garage}}] ($ (q) + (-0.5, 0.3) $) to ($ (q) + (0.5, 0.3) $);
\path[draw, -{To[scale=0.5]}, bend right=60] (1-start) to (1-end);
\path[draw, -{To[scale=0.5]}, bend right=60] (4-start) to (4-end);
\path[draw, -{To[scale=0.5]}, bend right=80, looseness=1.5] (5-start) to (5-end);
\path (0.4, 0.6) node {$ μ $};
\path (0, 1.6) node {$ ν $};
\end{tikzpicture}
\caption*{Case (d)}
\end{subfigure}
\caption{Cancellation for \autoref{th:hochschild-construction-mu-nu-end-long}}
\end{figure}

\begin{lemma}
\label{th:hochschild-construction-mu-nu-middle-outsource}
A contribution $ μ^{≥3} (ν^{≥2} (…), …) $ with middle-split first-out $ ν $ and first-out $ μ $ cancels.
\end{lemma}

\begin{proof}
Label the angles as $ μ^{≥3} (ν^{≥2} (α_k, …, α_{m+1}), α_m, …) $. We distinguish cases: (a) The result part of $ α_1 $ is shorter than the corresponding interior angle of $ ν $, it ends before the $ m $ puncture and cuts the $ ν $ piece into two. (b) The result part of $ α_1 $ is shorter than the corresponding interior angle of the $ ν $, and it ends at the $ m $ puncture in $ ν $. (c) The result part of $ α_1 $ is shorter than the corresponding interior angle of $ ν $, it ends after the $ m $ puncture and cuts the $ ν $ piece into two at some angle $ α_t $, and the split is neither at the end of the sequence ($ s+1 = k $) nor at the cut ($ t = s+1 $). (d) The result part of $ α_1 $ is longer than the corresponding interior angle of $ ν $. (e) As in (c), but the split being at the end or at the cut.

Regard case (a). Write $ α_t $ for the angle where the result part of $ α_1 $ hits the $ ν $ sequence. We have a cancellation
\begin{equation*}
μ^{≥3} (ν^{≥2} (α_k, …, α_{m+1}), α_m, …) + ν^{≥2} (…, μ^{≥3} (α_t, …, α_1)).
\end{equation*}
Regard case (b). Write $ α_s, α_{s+1} $ for the angles where the split happens. The split gives rise to an end-split $ ν(α_s, …, α_1) $ and its result fills up the remaining angle from $ α_{s+1} $ to $ α_k $, giving a contribution $ μ^{≥3} (α_k, …, α_{s+1}, ν^{≥2} (α_s, …, α_1)) $. Together with $ μ^2 (α_{s+1}, α_s) $ contraction, this provides a cancellation
\begin{equation*}
μ^{≥3} (ν^{≥2} (α_k, …, α_{m+1}), α_m, …) + ν^{≥2} (…, μ^2 (α_{s+1}, α_s), …) + μ^{≥3} (α_k, …, α_{s+1}, ν^{≥2} (α_s, …, α_1)).
\end{equation*}
Regard case (c). Write $ α_t $ for the angle where the result part of $ α_1 $ hits the $ ν $ sequence. We have a cancellation
\begin{equation*}
μ^{≥3} (ν^{≥2} (α_k, …, α_{m+1}), α_m, …) + μ^{≥3} (…, ν^{≥2} (α_t, …, α_1)) + ν^{≥2} (…, μ^2 (α_{s+1}, α_s), …).
\end{equation*}
Note in case the result part of $ α_1 $ has no arc going to $ m $, the last term vanishes and the first two already cancel out. If the result part of $ α_1 $ however has an arc going to $ m $, then the third term precisely compensates for the difference in magic angle between the first two terms.

Regard case (d). We have a cancellation
\begin{equation*}
μ^{≥3} (ν^{≥2} (α_k, …, α_{m+1}), α_m, …) + ν^{≥2} (…, μ^2 (α_{s+1}, α_s), …) + μ^2 (α_k, ν^{≥2} (α_{k-1}, …, α_1)).
\end{equation*}

Regard case (e). We have a cancellation
\begin{equation*}
μ(ν(α_k, …, α_{m+1}), …) + μ^2 (α_k, ν(α_{k-1}, …, α_1)) + ν(…, μ^2 (α_{s+1}, α_s), …).
\end{equation*}
\end{proof}

\begin{figure}
\centering
\begin{subfigure}{0.24\linewidth}
\centering
\begin{tikzpicture}
\path[use as bounding box] (-2, 0) -- (1, 2);
\path[draw] (0, 0) -- ++(right:1) coordinate[pos=0.4] (4-start) -- ++(up:2) -- ++(left:2) coordinate (corner-top) coordinate[pos=0.3] (3-end) coordinate[pos=0.7] (3-start) coordinate[pos=0.5] (t) coordinate[pos=0.8] (2-end) -- ++(left:1) coordinate[pos=0.6] (2-start) -- ++(down:2) -- ++(right:1) coordinate (corner-bot) coordinate[pos=0.7] (1-end) -- ++(up:2) coordinate[pos=0.2] (7-end);
\path[draw] (corner-bot) -- ++(right:1) coordinate[pos=0.6] (5-end) coordinate[pos=0.4] (7-start) -- ++(up:1) coordinate (q) coordinate[pos=0.4] (4-end);
\path[draw, dashed, gray] (corner-bot) -- (t) coordinate[pos=0.15] (1-start);
\path[draw, dashed, gray] (q) -- (0, 2);
\path[fill] (q) circle[radius=0.05];
\path[draw, -{To[scale=0.75]}, bend right=45] (1-start) to node[pos=0.8, right, shift={(0, 0.02)}] {\small $ 1 $} (1-end);
\path[draw, -{To[scale=0.75]}, bend right=45] (4-start) to node[pos=0.7, below] {\small $ s $} (4-end);
\path[draw, -{To[scale=0.75]}, bend right=45] (4-end) to node[pos=0.3, below] {$ \scriptscriptstyle s+1 $} (5-end);
\path[draw, -{To[scale=0.75]}, bend right=80, looseness=1.5] (2-start) to node[pos=0.4, above] {$ \scriptscriptstyle m+1 $} (2-end);
\path[draw, -{To[scale=0.75]}, bend right=80, looseness=1.5] (3-start) to node[pos=0.55, above] {\small $ t $} (3-end);
\path[draw, -{To[scale=0.75]}, bend right=45] (7-start) to node[pos=0.25, left, shift={(0.04, 0)}] {\small $ k $} (7-end);
\path[draw, very thick, gray] (q) ++(-80:0.13) arc(-80:80:0.13);
\path[draw, very thick, gray] (q) ++(-85:0.2) arc(-85:85:0.2);
\path (-1.5, 1) node {$ μ $};
\path (0.5, 1) node {$ ν $};
\end{tikzpicture}
\caption*{Case (a)}
\end{subfigure}
\begin{subfigure}{0.24\linewidth}
\centering
\begin{tikzpicture}
\path[use as bounding box] (-2, 0) -- (1, 2);
\path[draw] (0, 0) -- ++(right:1) coordinate[pos=0.4] (4-start) -- ++(up:2) -- ++(left:2) coordinate (corner-top) coordinate[pos=0.3] (3-end) coordinate[pos=0.7] (3-start) coordinate[pos=0.5] (t) coordinate[pos=0.8] (2-end) -- ++(left:1) coordinate[pos=0.6] (2-start) -- ++(down:2) -- ++(right:1) coordinate (corner-bot) coordinate[pos=0.7] (1-end) -- ++(up:2) coordinate[pos=0.3] (7-end);
\path[draw] (corner-bot) -- ++(right:1) coordinate[pos=0.6] (5-end) coordinate[pos=0.4] (7-start) -- ++(up:1) coordinate (q) coordinate[pos=0.4] (4-end);
\path[draw, dashed, gray] (corner-bot) -- (q) coordinate[pos=0.2] (1-start);
\path[draw, dashed, gray] (q) -- (0, 2);
\path[fill] (q) circle[radius=0.05];
\path[draw, -{To[scale=0.75]}, bend right=60, looseness=1.5] (1-start) to node[pos=0.8, right] {\small $ 1 $} (1-end);
\path[draw, -{To[scale=0.75]}, bend right=45] (4-start) to node[pos=0.7, below] {\small $ s $} (4-end);
\path[draw, -{To[scale=0.75]}, bend right=45] (4-end) to node[pos=0.3, below] {$ \scriptscriptstyle s+1 $} (5-end);
\path[draw, -{To[scale=0.75]}, bend right=80, looseness=1.5] (2-start) to node[pos=0.4, above] {$ \scriptscriptstyle m+1 $} (2-end);
\path[draw, -{To[scale=0.75]}, bend right=80, looseness=1.5] (3-start) to node[pos=0.55, above] {\small $ t $} (3-end);
\path[draw, -{To[scale=0.75]}, bend right=45] (7-start) to node[pos=0.55, left] {\small $ k $} (7-end);
\path[draw, very thick, gray] (q) ++(-80:0.13) arc(-80:215:0.13);
\path[draw, very thick, gray] (q) ++(-85:0.2) arc(-85:85:0.2);
\path[draw, very thick, gray] (q) ++(95:0.2) arc(95:220:0.2);
\path (-1.5, 1) node {$ μ $};
\path (0.5, 1) node {$ ν $};
\end{tikzpicture}
\caption*{Case (b)}
\end{subfigure}
\begin{subfigure}{0.24\linewidth}
\centering
\begin{tikzpicture}
\path[use as bounding box] (-2, 0) -- (1, 2);
\path[draw] (0, 0) -- ++(right:1) -- ++(up:2) coordinate[pos=0.05] (6-end) coordinate[pos=0.3] (6-start) coordinate[pos=0.4] (5-end) coordinate[pos=0.7] (4-start) -- ++(left:2) coordinate (corner-top) coordinate[pos=0.3] (3-end) coordinate[pos=0.7] (3-start) coordinate[pos=0.5] (t) coordinate[pos=0.8] (2-end) -- ++(left:1) coordinate[pos=0.6] (2-start) -- ++(down:2) -- ++(right:1) coordinate (corner-bot) coordinate[pos=0.7] (1-end) -- ++(up:2) coordinate[pos=0.3] (7-end);
\path[draw] (corner-bot) -- ++(right:1) coordinate[pos=0.4] (7-start) ++(up:1) coordinate (q) -- ++(right:1) coordinate[pos=0.6] (4-end);
\path[draw, dashed, gray] (corner-bot) -- (1, 0.3) coordinate[pos=0.15] (1-start);
\path[draw, dashed, gray] (q) -- ++(up:1);
\path[draw, dashed, gray] (corner-bot) -- (q);
\path[fill] (q) circle[radius=0.05];
\path[draw, -{To[scale=0.75]}, bend right=90, looseness=1.5] (1-start) to node[pos=0.85, right] {\small $ 1 $} (1-end);
\path[draw, -{To[scale=0.75]}, bend right=45] (4-start) to node[pos=0.2, below] {\small $ s $} (4-end);
\path[draw, -{To[scale=0.75]}, bend right=60, looseness=1.5] (4-end) to node[pos=0.7, above] {$ \scriptscriptstyle s+1 $} (5-end);
\path[draw, -{To[scale=0.75]}, bend right=80, looseness=1.5] (2-start) to node[pos=0.4, above] {$ \scriptscriptstyle m+1 $} (2-end);
\path[draw, -{To[scale=0.75]}, bend right=80, looseness=2] (6-start) to node[pos=0.2, below] {\small $ t $} (6-end);
\path[draw, -{To[scale=0.75]}, bend right=45] (7-start) to node[pos=0.55, left] {\small $ k $} (7-end);
\path[draw, very thick, gray] (q) ++(10:0.13) arc(10:215:0.13);
\path[draw, very thick, gray] (q) ++(5:0.2) arc(5:85:0.2);
\path[draw, very thick, gray] (q) ++(95:0.2) arc(95:220:0.2);
\path (-1.5, 1) node {$ μ $};
\path (-0.6, 1) node {$ ν $};
\end{tikzpicture}
\caption*{Case (c)}
\end{subfigure}
\begin{subfigure}{0.24\linewidth}
\centering
\begin{tikzpicture}
\path[use as bounding box] (-2, 0) -- (1, 2);
\path[draw] (0, 0) -- ++(right:1) coordinate[pos=0.4] (4-start) -- ++(up:2) -- ++(left:2) coordinate (corner-top) coordinate[pos=0.3] (3-end) coordinate[pos=0.7] (3-start) coordinate[pos=0.5] (t) coordinate[pos=0.8] (2-end) -- ++(left:1) coordinate[pos=0.6] (2-start) -- ++(down:2) -- ++(right:1) coordinate (corner-bot) coordinate[pos=0.7] (1-end) -- ++(up:2) coordinate[pos=0.3] (7-end);
\path[draw] (corner-bot) -- ++(right:1) coordinate[pos=0.6] (5-end) coordinate[pos=0.4] (7-start) -- ++(up:1) coordinate (q) coordinate[pos=0.4] (4-end);
\path[draw, dashed, gray] (corner-bot) -- (q) coordinate[pos=0.2] (1-start);
\path[draw, dashed, gray] (q) -- (0, 2);
\path[fill] (q) circle[radius=0.05];
\path[draw, -{To[scale=0.75]}, bend right=90, looseness=2.5] ($ (1-end) + (0.5, -0.2) $) to (1-end);
\path (corner-bot) node[above left, shift={(0.15, -0.1)}] {\small $ 1 $};
\path[draw, -{To[scale=0.75]}, bend right=45] (4-start) to node[pos=0.7, below] {\small $ s $} (4-end);
\path[draw, -{To[scale=0.75]}, bend right=45] (4-end) to node[pos=0.3, below] {$ \scriptscriptstyle s+1 $} (5-end);
\path[draw, -{To[scale=0.75]}, bend right=80, looseness=1.5] (2-start) to node[pos=0.4, above] {$ \scriptscriptstyle m+1 $} (2-end);
\path[draw, -{To[scale=0.75]}, bend right=80, looseness=1.5] (3-start) to node[pos=0.55, above] {\small $ t $} (3-end);
\path[draw, -{To[scale=0.75]}, bend right=45] (7-start) to node[pos=0.55, left] {\small $ k $} (7-end);
\path[draw, very thick, gray] (q) ++(-80:0.13) arc(-80:215:0.13);
\path[draw, very thick, gray] (q) ++(-85:0.2) arc(-85:85:0.2);
\path[draw, very thick, gray] (q) ++(95:0.2) arc(95:220:0.2);
\path (-1.5, 1) node {$ μ $};
\path (0.5, 1) node {$ ν $};
\end{tikzpicture}
\caption*{Case (d)}
\end{subfigure}
\caption{Cancellation for \autoref{th:hochschild-construction-mu-nu-middle-outsource}}
\label{fig:hochschild-construction-mu-nu-middle-outsource}
\end{figure}

\begin{lemma}
\label{th:hochschild-construction-mu-nu-mid-split}
A contribution $ μ^{≥3} (…, ν^{≥2} (…), …) $ with middle-split $ ν $ cancels.
\end{lemma}

\begin{proof}
We carry out the inspection only in case $ ν $ is first-out. Distinguish cases: (a) The $ ν^{≥2} $ result is not the final angle for $ μ $. (b) The $ ν^{≥2} $ result is the final angle for $ μ $, and $ μ $ is all-in or final-out. The remaining case that the $ ν^{≥2} $ result is the final angle for $ μ $ and $ μ $ is first-out is the content of \autoref{th:hochschild-construction-mu-nu-middle-outsource}.

Regard case (a). We have a cancellation
\begin{equation*}
μ^{≥3} (…, ν^{≥2} (α_m, …, α_{n+1}), …) + ν^{≥2} (…, μ^2 (α_{m+1}, α_m), …). 
\end{equation*}
The magic angles of both $ ν $ terms are readily seen to be equal: The input sequence for the second $ ν $ is only prolonged by a disk sequence, hence has no additional sectors around $ m $.

Regard case (b). We can then swap the order in which we apply $ μ $ and $ ν $:
\begin{equation*}
μ^{≥3} (ν^{≥2} (…, α_{m+1}), α_m, …, α_1) + ν^{≥2} (…, μ(α_{m+1}, …, α_1)). 
\end{equation*}
\end{proof}

\begin{figure}
\centering
\begin{subfigure}{0.24\linewidth}
\centering
\begin{tikzpicture}
\path[use as bounding box] (-2, 0) -- (1, 2);
\path[draw] (0, 0) -- ++(right:1) coordinate[pos=0.4] (4-start) -- ++(up:2) -- ++(left:2) coordinate (corner-top) coordinate[pos=0.3] (3-end) coordinate[pos=0.7] (3-start) coordinate[pos=0.5] (t) coordinate[pos=0.8] (2-end) -- ++(left:1) coordinate[pos=0.6] (2-start) -- ++(down:2) -- ++(right:1) coordinate (corner-bot) coordinate[pos=0.4] (1-end) -- ++(up:2) coordinate[pos=0.2] (7-end);
\path[draw] (corner-bot) -- ++(right:1) coordinate[pos=0.6] (5-end) coordinate[pos=0.4] (7-start) -- ++(up:1) coordinate (q) coordinate[pos=0.4] (4-end);
\path[draw, dashed, gray] (q) -- (-0.5, 2);
\path[fill] (q) circle[radius=0.05];
\path[draw, -{To[scale=0.75]}, bend right=45] (7-end) to node[pos=0.3, below] {$ \scriptscriptstyle m+1 $} (1-end);
\path[draw, -{To[scale=0.75]}, bend right=45] (7-start) to node[pos=0.7, below] {\small $ m $} (7-end);
\path[draw, -{To[scale=0.75]}, bend right=45] (4-start) to node[pos=0.7, below] {\small $ s $} (4-end);
\path[draw, -{To[scale=0.75]}, bend right=45] (4-end) to node[pos=0.3, below] {$ \scriptscriptstyle s+1 $} (5-end);
\path[draw, -{To[scale=0.75]}, bend right=80, looseness=1.5] (2-start) to node[pos=0.4, above] {$ \scriptscriptstyle n+1 $} (2-end);
\path[draw, very thick, gray] (q) ++(-80:0.13) arc(-80:110:0.13);
\path[draw, very thick, gray] (q) ++(-85:0.2) arc(-85:115:0.2);
\path (-1.5, 1) node {$ μ $};
\path (0.5, 1) node {$ ν $};
\end{tikzpicture}
\caption*{Case (a)}
\end{subfigure}
\begin{subfigure}{0.24\linewidth}
\centering
\begin{tikzpicture}
\path[use as bounding box] (-2, 0) -- (1, 2);
\path[draw] (0, 0) -- ++(right:1) coordinate[pos=0.4] (4-start) -- ++(up:2) -- ++(left:2) coordinate (corner-top) coordinate[pos=0.3] (3-end) coordinate[pos=0.7] (3-start) coordinate[pos=0.5] (t) coordinate[pos=0.8] (2-end) -- ++(left:1) coordinate[pos=0.6] (2-start) -- ++(down:2) -- ++(right:1) coordinate (corner-bot) coordinate[pos=0.6] (1-end) -- ++(up:2) coordinate[pos=0.2] (7-end);
\path[draw] (corner-bot) -- ++(right:1) coordinate[pos=0.6] (5-end) coordinate[pos=0.4] (7-start) -- ++(up:1) coordinate (q) coordinate[pos=0.4] (4-end);
\path[draw, dashed, gray] (q) -- (-0.5, 2);
\path[fill] (q) circle[radius=0.05];
\path[draw, -{To[scale=0.75]}, bend right=45] (7-end) to node[pos=0.3, below] {\small $ 1 $} (1-end);
\path[draw, -{To[scale=0.75]}, bend right=45] (7-start) to node[pos=0.7, below] {\small $ k $} (7-end);
\path[draw, -{To[scale=0.75]}, bend right=45] (4-start) to node[pos=0.7, below] {\small $ s $} (4-end);
\path[draw, -{To[scale=0.75]}, bend right=45] (4-end) to node[pos=0.3, below] {$ \scriptscriptstyle s+1 $} (5-end);
\path[draw, -{To[scale=0.75]}, bend right=80, looseness=1.5] (2-start) to node[pos=0.4, above] {$ \scriptscriptstyle m+1 $} (2-end);
\path[draw, very thick, gray] (q) ++(-80:0.13) arc(-80:110:0.13);
\path[draw, very thick, gray] (q) ++(-85:0.2) arc(-85:115:0.2);
\path (-1.5, 1) node {$ μ $};
\path (0.5, 1) node {$ ν $};
\end{tikzpicture}
\caption*{Case (b)}
\end{subfigure}
\caption{Cancellation for \autoref{th:hochschild-construction-mu-nu-mid-split}}
\label{fig:hochschild-construction-mu-nu-mid-split}
\end{figure}

\begin{lemma}
\label{th:hochschild-construction-nu-mu-end-small}
$ d ν $ vanishes on any sequence that has a $ ν^{≥2} (…, μ^{≥3} (…), …) $ contribution, with $ ν $ end-split with turning angle $ < ℓ^r $.
\end{lemma}

\begin{proof}
We distinguish cases: (a) The inner $ μ $ is first-out, and its result is used as final angle of $ ν $. (b) The inner $ μ $ is first-out, and its result is not used as final angle of $ ν $. (c) The inner $ μ $ is final-out and its result is used as first angle for $ ν $, and the turning angle of $ ν $ together with $ α_1 $ is less than or equal to $ ℓ^r $. (d) The inner $ μ $ is final-out and its result is not used as first angle for $ ν $. (e) As (c), but with angle together larger than $ ℓ^r $.

Regard case (a). Then the final angle $ α_k $ of the inner $ μ $ winds around $ m $. Distinguish (a1) $ α_k $ together with the turning angle of $ ν $ is bigger than $ ℓ^r $. (a2) $ α_k $ together with the turning angle is $ ℓ^r $. (a3) $ α_1 $ together with the turning angle is less than $ ℓ^r $.

Regard case (a1). Then the sequence $ α_1, …, α_{k-1} $ winds more than $ ℓ^r $ times around $ m $ and we simply have a final-out parking garage. Regard case (a2). Then we have an end-split contribution $ ν^{≥2} (α_{k-1}, …, α_1) $ with turning angle $ ≤ ℓ^r $, and find the cancellation
\begin{equation*}
ν^{≥2} (μ^{≥3} (α_k, …, α_{m+1}), α_m, …) + μ(ν^1 (α_k), …) + μ^2 (α_k, ν^{≥2} (α_{k-1}, …, α_1)). 
\end{equation*}
Case (a3) has the same cancellation as (a2).

Regard case (b). Then the result of the inner $ μ $ is not the final angle of the sequence. We can simply connect the final angle $ α_m $ of the inner $ μ $ with the next angle $ α_{m+1} $ of the outer $ ν $, producing a cancellation
\begin{equation*}
ν^{≥2} (…, α_{m+1}, μ^{≥3} (α_m, …, α_{n+1}), …) + ν^{≥2} (…, μ^2 (α_{m+1}, α_m), …). 
\end{equation*}

Case (c) is similar to (a). Indeed, $ α_1 $ winds around $ m $. By assumption $ α_1 $ together with the turning angle of $ ν $ is less than $ ℓ^r $. We have an end-split contribution $ ν^{≥2} (α_k, …, α_2) $ with turning angle $ ≤ ℓ^r $, and find the cancellation
\begin{equation*}
ν^{≥2} (…, μ^{≥3} (α_m, …, α_1)) + μ^{≥3} (…, ν^1 (α_1)) + μ^2 (ν^{≥2} (α_k, …, α_2), α_1). 
\end{equation*}
Case (d) is similar to (b). Case (e) is a parking garage sequence.
\end{proof}

\begin{figure}
\centering
\begin{subfigure}{0.24\linewidth}
\centering
\begin{tikzpicture}[scale=1.5]
\path[draw] (0, 0) -- ++(right:1) coordinate[pos=0.55] (1-start) -- ++(up:2) -- ++(left:2) -- ++(down:2) -- ++(right:1) coordinate[pos=0.6] (3-end);
\path[draw] (0, 0) -- (0, 1) coordinate (q) coordinate[pos=0.6] (2-end) coordinate[pos=0.4] (3-start);
\path[fill] (q) circle[radius=0.05];
\path[draw, dashed, gray] (q) -- (-1, 2) coordinate[pos=0.5] (2-start);
\path[draw, dashed, gray] (0, 0) -- (-1, 1) coordinate[pos=0.2] (1-end);
\path (-1, 2) -- (-1, 1) node[midway, sloped, below] {…};
\path[draw, -{To[scale=0.75]}, bend right=65] (1-start) to node[pos=0.5, below] {\scalebox{1}{$ \scriptstyle m+1 $}} (1-end);
\path[draw, -{To[scale=0.75]}, bend right=60] (2-start) to node[pos=0.2, right] {\small $ k $} (2-end);
\path[draw, -{To[scale=0.75]}, bend right=45] (3-start) to node[pos=0.5, below] {\small $ 1 $} (3-end);
\path (0.7, 1.1) node {$ ν $};
\path (-0.8, 1.2) node {$ μ $};
\path[draw, very thick, gray] (q) ++(260:0.1) arc(260:-80:0.1);
\path[draw, very thick, gray] (q) ++(260:0.15) arc(260:130:0.15);
\path[draw, very thick, gray] (q) ++(130:0.2) arc(130:360:0.2) coordinate (spiral-right);
\path[draw, very thick, gray] (spiral-right) arc(0:267:0.23 and 0.25);
\end{tikzpicture}
\caption*{Case (a)}
\end{subfigure}
\begin{subfigure}{0.24\linewidth}
\centering
\begin{tikzpicture}[scale=1.5]
\path[use as bounding box] (-1, -0.5) -- (1, 2);
\path[draw] (0, 0) -- ++(right:1) -- ++(up:2) -- ++(left:2) -- ++(down:2) coordinate[pos=0.85] (3-end) -- ++(right:1) coordinate[pos=0.4] (3-start);
\path[draw] (0, 0) -- (0, 1) coordinate (q) coordinate[pos=0.3] (1-start);
\path[fill] (q) circle[radius=0.05];
\path[draw] (0, 0) -- ++(down:0.5) coordinate[pos=0.5] (1-end) -- ++(left:1) node[midway, above] {…} -- ++(up:0.5) coordinate[pos=0.5] (2-start);
\path[draw, -{To[scale=0.75]}, bend right=90, looseness=2] (1-start) to node[pos=0.15, below] {$ \scriptscriptstyle n+1 $} (1-end);
\path[draw, -{To[scale=0.75]}, bend right=45] (2-start) to node[near start, above] {\tiny $ m $} (3-start);
\path[draw, -{To[scale=0.75]}, bend right=45] (3-start) to node[pos=0.7, below] {$ \scriptscriptstyle m+1 $} (3-end);
\path[draw, very thick, gray] (q) ++(260:0.1) arc(260:-80:0.1);
\path[draw, very thick, gray] (q) ++(265:0.15) arc(265:-85:0.15);
\path (-0.5, -0.2) node {$ μ $};
\path (0, 1.3) node {$ ν $};
\end{tikzpicture}
\caption*{Case (b)}
\end{subfigure}
\begin{subfigure}{0.24\linewidth}
\centering
\begin{tikzpicture}[scale=1.5]
\path[draw] (0, 0) -- ++(right:1) coordinate[pos=0.4] (4-start) -- ++(up:2) -- ++(left:2) -- ++(down:2) -- ++(right:1) coordinate[pos=0.6] (1-end);
\path[draw] (0, 0) -- (0, 1) coordinate (q) coordinate[pos=0.7] (2-start) coordinate[pos=0.4] (4-end);
\path[fill] (q) circle[radius=0.05];
\path[draw, dashed, gray] (q) -- (1, 2) coordinate[pos=0.35] (2-end);
\path[draw, dashed, gray] (0, 0) -- (1, 1) coordinate[pos=0.2] (1-start);
\path (1, 2) -- (1, 1) node[midway, sloped, above] {…};
\path[draw, -{To[scale=0.75]}, bend right=60] (1-start) to node[midway, below] {\small $ m $} (1-end);
\path[draw, -{To[scale=0.75]}, bend right=65] (2-start) to node[pos=0.5, above] {\small $ 1 $} (2-end);
\path[draw, -{To[scale=0.75]}, bend right=45] (4-start) to node[pos=0.45, below] {\small $ k $} (4-end);
\path[draw, very thick, gray] (q) ++(260:0.1) arc(260:-80:0.1);
\path[draw, very thick, gray] (q) ++(-85:0.15) arc(-85:40:0.15);
\path[draw, very thick, gray] (q) ++(-85:0.2) arc(-85:90:0.2) coordinate (spiral-up);
\path[draw, very thick, gray] (spiral-up) arc(90:270:0.22) coordinate (spiral-down);
\path[draw, very thick, gray] (spiral-down) arc(-90:40:0.25);
\path (0.7, 1.1) node {$ μ $};
\path (-0.7, 1.6) node {$ ν $};
\end{tikzpicture}
\caption*{Case (c)}
\end{subfigure}
\begin{subfigure}{0.24\linewidth}
\centering
\begin{tikzpicture}[scale=1.5]
\path[use as bounding box] (-1, -0.5) -- (1, 2);
\path[draw] (0, 0) -- ++(right:1) -- ++(up:2) -- ++(left:2) -- ++(down:2) coordinate[pos=0.85] (3-end) -- ++(right:1) coordinate[pos=0.6] (2-start);
\path[draw] (0, 0) -- (0, 1) coordinate (q) coordinate[pos=0.3] (1-start);
\path[fill] (q) circle[radius=0.05];
\path[draw] (0, 0) -- ++(down:0.5) coordinate[pos=0.5] (2-end) -- ++(left:1) node[midway, above] {…} -- ++(up:0.5) coordinate[pos=0.5] (3-start);
\path[draw, -{To[scale=0.75]}, bend right=45] (1-start) to node[near start, below] {\small $ n $} (2-start);
\path[draw, -{To[scale=0.75]}, bend right=45] (2-start) to node[near end, above] {$ \scriptscriptstyle n+1 $} (2-end);
\path[draw, -{To[scale=0.75]}, bend right=90, looseness=1.5] (3-start) to node[pos=0.85, below] {\small $ m $} (3-end);
\path[draw, very thick, gray] (q) ++(260:0.1) arc(260:-80:0.1);
\path[draw, very thick, gray] (q) ++(265:0.15) arc(265:-85:0.15);
\path (-0.5, -0.2) node {$ μ $};
\path (0, 1.3) node {$ ν $};
\end{tikzpicture}
\caption*{Case (d)}
\end{subfigure}
\caption{Cancellation for \autoref{th:hochschild-construction-nu-mu-end-small}}
\label{fig:hochschild-construction-nu-mu-end-small}
\end{figure}

\begin{lemma}
\label{th:hochschild-construction-nu-mu-end-old-long}
$ d ν $ vanishes on any sequence that has a $ ν^{≥2} (…, μ^{≥3} (…), …) $ contribution, with $ ν $ old-era or new-era end-split with turning angle $ ℓ^r $.
\end{lemma}

\begin{proof}
See \autoref{fig:hochschild-construction-nu-mu-end-old-long}. Let us check the old-era case first and then comment on the new-era case.

Distinguish cases: (a) The inner $ μ $ is first-out and its result is the final angle of $ ν $. (b) The inner $ μ $ is first-out and its result is an ordinary or the first angle of $ ν $. (c) The inner $ μ $ is final-out and its result is the first angle of $ ν $, and $ ν $ is all-in. (d) The inner $ μ $ is final-out and its result is the first angle of $ ν $, and $ ν $ is first-out. (e) The inner $ μ $ is final-out and its result is not the first angle of the outer $ ν $.

Regard case (a). Write the contribution as $ ν(μ(α_k, …, α_{m+1}), …, ) $. Then the final angle $ α_k $ of the inner $ μ $ winds around $ m $. By assumption, the sequence $ α_1, …, α_{m+1} $ already winds at least $ ℓ^r $ around $ m $, in particular does $ α_1, …, α_{k-1} $. This constitutes a garage sequence.

Regard case (b). We can simply connect the angle $ α_{m+1} $ after $ μ $ to the final angle $ α_m $ of $ μ $:
\begin{equation*}
ν^{≥2} (…, μ^{≥3} (α_m, …, α_{n+1}), …) + ν^{≥2} (…, μ^2 (α_{m+1}, α_m), …).
\end{equation*}
Case (c) is similar to case (a) and yields a parking garage sequence. In case (d), there is no relevant turning around $ m $ and we can swap the order in which we apply $ μ $ and $ ν $:
\begin{equation*}
ν^{≥2} (…, μ^{≥3} (α_m, …, α_1)) + μ^{≥3} (ν^{≥2} (α_k, …, α_m), …). 
\end{equation*}
Both $ ν $ evaluations have equal final angle. Since the left one is old-era, the right-one is old-era as well and both have equal coefficient. Case (e) is similar to (b).

We have proven the old-era case. Finally, let us comment on the case $ ν $ is new-era instead. We claim all cancellations carry over one-to-one. Indeed, in the proof until now we have only used cancellations via parking garage sequences and cancellations in pairs. Those two cancellations from parking garage sequences carry over, since at a parking garage we are free to append $ γ $ at the front. The cancellations in pairs consist of two cancellations with a $ μ^2 $ insertion and one cancellation by swapping the order of applying $ μ $ and $ ν $. In the new-era case, these three cancellations still exist: In all three cancellations, the coefficients of both contributions change simultaneously to $ \#ν^1 $ of the new-era sector and hence still cancel out.
\end{proof}

\begin{figure}
\centering
\begin{subfigure}{0.24\linewidth}
\centering
\begin{tikzpicture}[scale=1.5]
\path[draw] (0, 0) -- ++(right:1) coordinate[pos=0.55] (1-start) -- ++(up:2) -- ++(left:2) -- ++(down:2) -- ++(right:1) coordinate[pos=0.6] (3-end);
\path[draw] (0, 0) -- (0, 1) coordinate (q) coordinate[pos=0.6] (2-end) coordinate[pos=0.4] (3-start);
\path[fill] (q) circle[radius=0.05];
\path[draw, dashed, gray] (q) -- (-1, 2) coordinate[pos=0.3] (2-start);
\path[draw, dashed, gray] (0, 0) -- (-1, 1) coordinate[pos=0.2] (1-end);
\path (-1, 2) -- (-1, 1) node[midway, sloped, below] {…};
\path[draw, -{To[scale=0.75]}, bend right=65] (1-start) to node[pos=0.5, below] {\scalebox{1}{$ \scriptstyle m+1 $}} (1-end);
\path[draw, -{To[scale=0.75]}, bend right=60] (2-start) to node[pos=0.4, right] {\small $ k $} (2-end);
\path[draw, -{To[scale=0.75]}, bend right=45] (3-start) to node[pos=0.5, below] {\small $ 1 $} (3-end);
\path (0.7, 1.1) node {$ ν $};
\path (-0.8, 1.2) node {$ μ $};
\path[bend left, decorate, decoration={text along path, text align=center, text={|\tiny|parking garage}}] (-0.2, 1.5) to (0.5, 0.5);
\end{tikzpicture}
\caption*{Case (a)}
\end{subfigure}
\begin{subfigure}{0.24\linewidth}
\centering
\begin{tikzpicture}[scale=1.5]
\path[use as bounding box] (-1, -0.5) -- (1, 2);
\path[draw] (0, 0) -- ++(right:1) -- ++(up:2) -- ++(left:2) -- ++(down:2) coordinate[pos=0.85] (3-end) -- ++(right:1) coordinate[pos=0.4] (3-start);
\path[draw] (0, 0) -- (0, 1) coordinate (q) coordinate[pos=0.3] (1-start);
\path[fill] (q) circle[radius=0.05];
\path[draw] (0, 0) -- ++(down:0.5) coordinate[pos=0.5] (1-end) -- ++(left:1) node[midway, above] {…} -- ++(up:0.5) coordinate[pos=0.5] (2-start);
\path[draw, -{To[scale=0.75]}, bend right=90, looseness=2] (1-start) to node[pos=0.15, below] {$ \scriptscriptstyle n+1 $} (1-end);
\path[draw, -{To[scale=0.75]}, bend right=45] (2-start) to node[near start, above] {\tiny $ m $} (3-start);
\path[draw, -{To[scale=0.75]}, bend right=45] (3-start) to node[pos=0.7, below] {$ \scriptscriptstyle m+1 $} (3-end);
\path[draw, very thick, gray] (q) ++(260:0.1) arc(260:-80:0.1);
\path[draw, very thick, gray] (q) ++(265:0.15) arc(265:-85:0.15);
\path (-0.5, -0.2) node {$ μ $};
\path (0, 1.3) node {$ ν $};
\end{tikzpicture}
\caption*{Case (b)}
\end{subfigure}
\begin{subfigure}{0.24\linewidth}
\centering
\begin{tikzpicture}[scale=1.5]
\path[draw] (0, 0) -- ++(right:1) coordinate[pos=0.4] (4-start) -- ++(up:2) -- ++(left:2) -- ++(down:2) -- ++(right:1) coordinate[pos=0.6] (1-end);
\path[draw] (0, 0) -- (0, 1) coordinate (q) coordinate[pos=0.7] (2-start) coordinate[pos=0.4] (4-end);
\path[fill] (q) circle[radius=0.05];
\path[draw, dashed, gray] (q) -- (1, 2) coordinate[pos=0.3] (2-end);
\path[draw, dashed, gray] (0, 0) -- (1, 1) coordinate[pos=0.2] (1-start);
\path (1, 2) -- (1, 1) node[midway, sloped, above] {…};
\path[draw, -{To[scale=0.75]}, bend right=60] (1-start) to node[midway, below] {\small $ m $} (1-end);
\path[draw, -{To[scale=0.75]}, bend right=65] (2-start) to node[pos=0.5, above] {\small $ 1 $} (2-end);
\path[draw, -{To[scale=0.75]}, bend right=45] (4-start) to node[pos=0.45, below] {\small $ k $} (4-end);
\path[bend left, decorate, decoration={text along path, text align=center, text={|\tiny|parking garage}}] ($ (q) + (-0.5, -0.8) $) to ($ (q) + (0.3, 0.8) $);
\path (0.7, 1.1) node {$ μ $};
\path (-0.7, 1.6) node {$ ν $};
\end{tikzpicture}
\caption*{Case (c)}
\end{subfigure}
\begin{subfigure}{0.24\linewidth}
\centering
\begin{tikzpicture}[scale=1.5]
\path[draw] (0, 0) -- ++(right:1) coordinate[pos=0.4] (4-start) -- ++(up:2) -- ++(left:2) -- ++(down:2) -- ++(right:1) coordinate[pos=0.6] (1-end);
\path[draw] (0, 0) -- (0, 1) coordinate (q) coordinate[pos=0.4] (4-end);
\path[fill] (q) circle[radius=0.05];
\path[draw, dashed, gray] (0, 0) -- (1, 2) coordinate[pos=0.25] (2-end) coordinate[pos=0.85] (2-start);
\path[draw, dashed, gray] (0, 0) -- (1, 0.5) coordinate[pos=0.2] (1-start) coordinate[pos=0.8] (3-end);
\path (1, 2) -- (1, 0.5) node[midway, sloped, above] {…};
\path[draw, -{To[scale=0.75]}, bend right=60, looseness=1.5] (1-start) to node[pos=0.55, below] {$ \scriptscriptstyle n+1 $} (1-end);
\path[draw, -{To[scale=0.75]}, bend right=80, looseness=1] (2-start) to node[midway, above] {\small $ 1 $} ($ (2-start) + (0.4, 0.2) $);
\path[draw, -{To[scale=0.75]}, bend right=45] (4-start) to node[pos=0.45, below] {\small $ k $} (4-end);
\path[draw, -{To[scale=0.75]}, bend right=60] ($ (3-end) + (0.3, 0.3) $) to node[pos=0.45, below] {\small $ n $} (3-end);
\path[draw, thick, gray] (q) ++(-80:0.1) arc(-80:260:0.1);
\path[draw, thick, gray] (q) ++(-80:0.15) arc(-80:260:0.15);
\path (0.6, 0.7) node {$ μ $};
\path (-0.5, 1.2) node {$ ν $};
\end{tikzpicture}
\caption*{Case (d)}
\end{subfigure}
\begin{subfigure}{0.24\linewidth}
\centering
\begin{tikzpicture}[scale=1.5]
\path[use as bounding box] (-1, -0.5) -- (1, 2);
\path[draw] (0, 0) -- ++(right:1) -- ++(up:2) -- ++(left:2) -- ++(down:2) coordinate[pos=0.85] (3-end) -- ++(right:1) coordinate[pos=0.6] (2-start);
\path[draw] (0, 0) -- (0, 1) coordinate (q) coordinate[pos=0.3] (1-start);
\path[fill] (q) circle[radius=0.05];
\path[draw] (0, 0) -- ++(down:0.5) coordinate[pos=0.5] (2-end) -- ++(left:1) node[midway, above] {…} -- ++(up:0.5) coordinate[pos=0.5] (3-start);
\path[draw, -{To[scale=0.75]}, bend right=45] (1-start) to node[near start, below] {\small $ n $} (2-start);
\path[draw, -{To[scale=0.75]}, bend right=45] (2-start) to node[near end, above] {$ \scriptscriptstyle n+1 $} (2-end);
\path[draw, -{To[scale=0.75]}, bend right=90, looseness=1.5] (3-start) to node[pos=0.85, below] {\small $ m $} (3-end);
\path[draw, very thick, gray] (q) ++(260:0.1) arc(260:-80:0.1);
\path[draw, very thick, gray] (q) ++(265:0.15) arc(265:-85:0.15);
\path (-0.5, -0.2) node {$ μ $};
\path (0, 1.3) node {$ ν $};
\end{tikzpicture}
\caption*{Case (e)}
\end{subfigure}
\caption{Cancellation for \autoref{th:hochschild-construction-nu-mu-end-old-long}}
\label{fig:hochschild-construction-nu-mu-end-old-long}
\end{figure}

\begin{lemma}
\label{th:hochschild-construction-nu-mu-mid}
Any contribution $ ν^{≥2} (…, μ^{≥3} (…), …) $ with $ ν $ mid-split cancels.
\end{lemma}

\begin{proof}
We only check this in case $ μ $ is final-out. Distinguish cases: (a) The $ μ $ result is used as first angle for $ ν $, and $ ν $ is first-out. (b) The $ μ $ is used as first angle for $ ν $, and the result is an ordinary angle for $ ν $. (c) The $ μ $ result is used as final-out angle, as first part of the split, or as ordinary but not first angle for $ ν $. (d) The $ μ $ is used as second part of the split.


Regard case (a). Then $ ν $ has no final-out angle and we simply swap the order of $ μ $ and $ ν $. 

Regard case (b). Then $ ν $ may have a final-out angle, preventing us from producing a cancellation from swapping the order. This complicated case can be dealt with in a similar way as in the case distinction of \autoref{th:hochschild-construction-mu-nu-mid-split}. 


Regard case (c). Label angles as $ ν(…, μ(α_m, …, α_{n+1}), …) $. We argue there is an angle $ α_n $ before $ α_{n+1} $ and we can produce a cancellation with the contraction $ ν(…, μ^2 (α_{n+1}, α_n), …) $. Indeed, if the $ μ $ result is used as first part of the split, then it is not the first angle in the $ ν $ sequence, since a mid-split contribution with first angle being the first part of the split vanishes already. Therefore we can assume there is an ordinary, second part of the split or first-out angle before $ α_s $. This produces a cancellation from a simple augmentation by $ μ^2 $. 

Regard case (d). This means the first angle of the inner $ μ $ goes around $ m $ and we have a parking garage. 
\end{proof}

\begin{figure}
\centering
\begin{subfigure}{0.24\linewidth}
\centering
\begin{tikzpicture}
\path[use as bounding box] (-3, -0.2) -- (1, 2.3);
\path[draw] (0, 0) -- ++(right:1) coordinate[pos=0.3] (1-start) -- ++(up:2) -- ++(left:2) coordinate (corner-top) coordinate[pos=0.8] (3-end) -- ++(down:2) -- ++(right:1) coordinate[pos=0.5] (2-end) -- ++(up:1) coordinate (q) coordinate[pos=0.3] (1-end);
\path[draw] (corner-top) -- ++(150:1) coordinate[pos=0.5] (3-start) -- ++(240:2) -- ++(330:1) coordinate[pos=0.5] (4-end) -- cycle coordinate[pos=0.2] (4-start);
\path[fill] (q) circle[radius=0.05];
\path[draw, dashed, gray] (q) -- (corner-top);
\path[draw, -{To[scale=0.75]}, bend right=45] (1-start) to node[pos=0.8, below] {\small $ s $} (1-end);
\path[draw, -{To[scale=0.75]}, bend right=45] (1-end) to node[pos=0.3, below] {$ \scriptscriptstyle s+1 $} (2-end);
\path[draw, -{To[scale=0.75]}, bend right=100, looseness=2] (3-start) to node[pos=0.4, above] {\small $ m $} (3-end);
\path[draw, -{To[scale=0.75]}, bend right=45] (4-start) to node[pos=0.5, below] {\small $ 1 $} (4-end);
\path[draw, very thick, gray] (q) ++(-80:0.13) arc(-80:135:0.13);
\path[draw, very thick, gray] (q) ++(-85:0.2) arc(-85:135:0.2);
\path (-2, 1) node {$ μ $};
\path (0, 1.5) node {$ ν $};
\end{tikzpicture}
\caption*{Case (a)}
\end{subfigure}
\begin{subfigure}{0.24\linewidth}
\centering
\begin{tikzpicture}
\path[use as bounding box] (-2, -0.2) -- (1, 2.3);
\path[draw] (0, 0) -- ++(right:1) coordinate[pos=0.3] (1-start) -- ++(up:2) -- ++(left:2) coordinate (corner-top) coordinate[pos=0.8] (3-end) -- ++(down:2) -- ++(right:1) coordinate[pos=0.3] (5-start) coordinate[pos=0.5] (2-end) -- ++(up:1) coordinate (q) coordinate[pos=0.3] (1-end);
\path[draw] (corner-top) -- ++(left:1) coordinate[pos=0.5] (3-start) -- ++(down:2) -- ++(right:1) coordinate[pos=0.5] (4-end) -- ++(up:2) coordinate[pos=0.2] (4-start);
\path[fill] (q) circle[radius=0.05];
\path[draw, -{To[scale=0.75]}, bend right=45] (1-start) to node[pos=0.8, below] {\small $ s $} (1-end);
\path[draw, -{To[scale=0.75]}, bend right=45] (1-end) to node[pos=0.3, below] {$ \scriptscriptstyle s+1 $} (2-end);
\path[draw, -{To[scale=0.75]}, bend right=80, looseness=1.5] (3-start) to node[pos=0.4, above] {\small $ m $} (3-end);
\path[draw, -{To[scale=0.75]}, bend right=45] (4-start) to node[pos=0.8, right] {\small $ 1 $} (4-end);
\path[draw, -{To[scale=0.75]}, bend right=100, looseness=2.5] (5-start) to node[pos=0.9, right] {\small $ k $} ($ (5-end) + (-0.7, -0.3) $);
\path (-1.5, 1) node {$ μ $};
\path (0, 1.5) node {$ ν $};
\path[decorate, decoration={text along path, text align=center, text={|\tiny|complicated}}, bend left] (-1, 0.9) to (1, 0.9);
\path[decorate, decoration={text along path, text align=center, text={|\tiny|as in Lemma {\ref*{th:hochschild-construction-mu-nu-mid-split}}}}, bend right] (-1, 1) to (1, 1);
\end{tikzpicture}
\caption*{Case (b)}
\end{subfigure}
\begin{subfigure}{0.24\linewidth}
\centering
\begin{tikzpicture}
\path[use as bounding box] (-2, -0.2) -- (1, 2.3);
\path[draw] (0, 0) -- ++(right:1) coordinate[pos=0.3] (1-start) -- ++(up:2) coordinate[pos=0.9] (5-start) coordinate[pos=0.55] (5-end) -- ++(left:2) coordinate (corner-top) -- ++(down:2) coordinate[pos=0.3] (4-end) coordinate[pos=0.8] (1-end) -- ++(right:1) coordinate[pos=0.3] (1-start) ++(up:1) coordinate (q);
\path[draw] (corner-top) -- ++(down:1) -- ++(left:1) coordinate[pos=0.3] (3-start) -- ++(down:1) -- ++(right:1) coordinate[pos=0.5] (2-end) -- ++(up:2);
\path[draw] (q) -- (-1, 1) coordinate[pos=0.4] (3-end);
\path[fill] (q) circle[radius=0.05];
\path[draw, dashed, gray] (q) -- (1, 1.5);
\path[draw, -{To[scale=0.75]}, bend right=60] (3-end) to node[pos=0.65, below] {$ \scriptscriptstyle m+1 $} (4-end);
\path[draw, -{To[scale=0.75]}, bend right=45] (1-start) to node[pos=0.8, below] {\small $ n $} (1-end);
\path[draw, -{To[scale=0.75]}, bend right=45] (1-end) to node[pos=0.3, below] {$ \scriptscriptstyle n+1 $} (2-end);
\path[draw, -{To[scale=0.75]}, bend right=80, looseness=1.5] (3-start) to node[pos=0.6, above] {\small $ m $} (3-end);
\path[draw, -{To[scale=0.75]}, bend right=80, looseness=1.5] (5-start) to node[pos=0.15, below] {\small $ 1 $} (5-end);
\path[draw, very thick, gray] (q) ++(-170:0.13) arc(-170:35:0.13);
\path[draw, very thick, gray] (q) ++(-175:0.2) arc(-175:35:0.2);
\path (-1.5, 0.5) node {$ μ $};
\path (0, 1.5) node {$ ν $};
\end{tikzpicture}
\caption*{Case (c)}
\end{subfigure}
\begin{subfigure}{0.24\linewidth}
\centering
\begin{tikzpicture}
\path[use as bounding box] (-1, -0.2) -- (1, 2.3);
\path[draw] (0, 0) -- ++(right:1) coordinate[pos=0.4] (4-start) -- ++(up:2) -- ++(left:2) -- ++(down:2) -- ++(right:1) coordinate[pos=0.6] (1-end);
\path[draw] (0, 0) -- (0, 1) coordinate (q) coordinate[pos=0.7] (2-start) coordinate[pos=0.4] (4-end);
\path[fill] (q) circle[radius=0.05];
\path[draw, dashed, gray] (q) -- (1, 2) coordinate[pos=0.4] (2-end);
\path[draw, dashed, gray] (0, 0) -- (1, 1) coordinate[pos=0.2] (1-start);
\path (1, 2) -- (1, 1) node[midway, sloped, above] {…};
\path[draw, -{To[scale=0.75]}, bend right=60] (1-start) to node[pos=0.4, below] {\small $ m $} (1-end);
\path[draw, -{To[scale=0.75]}, bend right=65, looseness=1.5] (2-start) to node[pos=0.25, above] {$ \scriptscriptstyle n+1 $} (2-end);
\path[draw, -{To[scale=0.75]}, bend right=45] (4-start) to node[pos=0.5, below] {\small $ n $} (4-end);
\path[bend left, decorate, decoration={text along path, text align=center, text={|\tiny|parking garage}}] ($ (q) + (-0.5, -0.8) $) to ($ (q) + (0.3, 0.8) $);
\path (0.7, 1.1) node {$ μ $};
\path (-0.7, 1.6) node {$ ν $};
\end{tikzpicture}
\caption*{Case (d)}
\end{subfigure}
\caption{Cancellation for \autoref{th:hochschild-construction-nu-mu-mid}}
\label{fig:hochschild-construction-nu-mu-mid}
\end{figure}

\begin{lemma}
\label{th:hochschild-construction-nu1-mu}
Any sequence contributing $ ν^1 (μ^{≥3} (…)) $ vanishes under $ d ν $.
\end{lemma}

\begin{proof}
We shall make a case distinction whether the inner $ μ $ is first-out or final-out. Both cases work similarly, so let us just assume the inner $ μ $ is first-out. Then the sequence is of the form $ α_1 β, …, α_k $ with $ α_1, …, α_k $ a disk sequence. This means $ α_1 $ is precisely as long as the total angle that $ α_2, …, α_k $ winds back. Moreover, $ ν^1 (β) $ is nonzero, hence $ β $ and also $ α_1 $ wind around $ m $.

To find a cancellation, our best guess is that $ α_1 β, …, α_k $ constitutes a parking garage with inner spiral angle (part of) $ α_1 β $. Whether this is the case or not depends on the size of the angle that $ α_2, …, α_k $ covers. Distinguish cases as follows: (a) The angle $ α_1 $ is bigger than $ ℓ^r $. (b) The angle $ α_1 $ is smaller than or equal to $ ℓ^r $.

In case (a) we have a parking garage sequence. Regard case (b). Then no inner $ ν^{≥2} $ application is possible with at most $ k - 2 $ terms, since $ α_2, …, α_k $ is too short. However since $ α_1 ≤ ℓ^r $, we have an end-split $ ν^{≥2} (α_k, …, α_2) $ and obtain a cancellation in a triple
\begin{equation*}
ν^1 (μ(α_k, …, α_1 β)) + μ^{≥3} (α_k, …, ν^1 (α_1 β)) + μ^2 (ν^{≥2} (α_k, …, α_2), α_1 β)).
\end{equation*}
See \autoref{fig:hochschild-construction-nu1-mu}.
\end{proof}

\begin{figure}
\centering
\begin{subfigure}{0.25\linewidth}
\centering
\begin{tikzpicture}[scale=1.5]
\path[draw] (0, 0) -- ++(120:1) coordinate[pos=0.3] (2-end) coordinate[pos=0.7] (3-start) -- ++(60:1) coordinate[pos=0.3] (3-end) coordinate[pos=0.7] (4-start) -- ++(right:1.2) coordinate[pos=0.3] (4-end) coordinate[pos=0.7] (5-start) -- ++(300:1) coordinate[pos=0.3] (5-end) coordinate[pos=0.7] (6-start) -- ++(240:1) coordinate[pos=0.3] (6-end) coordinate[pos=0.7] (1-start) coordinate (q) -- ++(left:1.2) coordinate[pos=0.4] (1-end) coordinate[pos=0.7] (2-start);
\foreach \i in {2, 3, 4, 5, 6} \path[draw, -{To[scale=0.75]}, bend right=60] (\i-start) to (\i-end);
\foreach \i in {2, 3, 4, 5, 6} \path ($ (\i-start)!0.5!(\i-end) $) node {\small $ \i $};
\path[draw, -{To[scale=0.75]}, bend right=80, looseness=1.5] ($ (q) + (0.5, 0) $) to node[pos=0.6, below] {\small $ 1 $} node[pos=0.3, below] {\small $ β $} (1-end);
\path[fill] (q) circle[radius=0.05];
\path[draw, very thick, gray] (q) ++(0:0.1) arc(0:170:0.1);
\path[draw, very thick, gray] (q) ++(70:0.15) arc(70:175:0.15);
\path[draw, very thick, gray] (q) ++(0:0.15) arc(0:50:0.15);
\end{tikzpicture}
\caption*{Case (b)}
\end{subfigure}
\begin{subfigure}{0.25\linewidth}
\centering
\begin{tikzpicture}[scale=1.5]
\path[draw] (0, 0) -- ++(right:1) coordinate[pos=0.3] (k-start) -- ++(up:2) -- ++(left:2) -- ++(down:2) -- ++(right:1) coordinate[pos=0.7] (1-end) -- ++(up:1) coordinate (q) coordinate[pos=0.3] (1-start);
\path[draw, -{To[scale=0.75]}] (q) ++(0, 0.5) node[below] {$ \scriptstyle 1=ℓ^r $} arc(90:270:0.55);
\path[draw] (q) ++(0, 0.5) arc(90:-90:0.4 and 0.5) coordinate (end);
\path[draw] (end) arc(270:180:0.4) node[pos=0.6, right] {\small $ β $};
\path[draw, very thick, gray] (q) ++(260:0.1) arc(260:-80:0.1);
\path[draw, very thick, gray] (q) ++(265:0.15) arc(265:0:0.15) coordinate (spiral-right);
\path[draw, very thick, gray] (spiral-right) arc(0:-180:0.175 and 0.19);
\path[draw, very thick, gray] (q) ++(265:0.25) arc(265:180:0.25);
\path[fill] (q) circle[radius=0.05];
\path[draw, -{To[scale=0.75]}, bend right=45] (1-start) to node[near start, below] {\small $ 2 $} (1-end);
\path[draw, -{To[scale=0.75]}, bend right=45] (k-start) to node[near end, below] {\small $ k $} (1-start);
\end{tikzpicture}
\caption*{Case (c)}
\end{subfigure}
\caption{Cancellation for \autoref{th:hochschild-construction-nu1-mu}}
\label{fig:hochschild-construction-nu1-mu}
\end{figure}

\begin{lemma}
\label{th:hochschild-construction-mu-nu1-outsource}
Any contribution $ μ^{≥3} (α_k, …, ν^1 (α_1)) $ with first-out $ μ $ cancels.
\end{lemma}

\begin{proof}
Label the angles as $ μ(α_k, …, ν^1 (α_1)) $. Since this contribution is a first-out $ μ $, we can write $ ν^1 (α_k) = α β $. While $ α_1, …, ν^1 (α_k) $ is a first-out disk, the sequence $ α, α_2, …, α_k $ is an actual (all-in) disk sequence. The angle $ α $ is the angle that describes how much $ α_2, …, α_k $ turns. Let us distinguish cases after the length of $ α $: (a) We have $ α > ℓ^r $. (b1) We have $ α ≤ ℓ^r $ and $ α_1 < α $. (c) We have $ α ≤ ℓ^r $ and $ α_1 ≥ α $.

In case (a), we have first-out garage sequence.

Regard case (b1). Then the tail arc of $ α_1 $ cuts the $ α, α_2, …, α_k $ disk into two. Denote by $ α_t $ the angle where the arc touches the opposite side of the disk. We reach the cancellation
\begin{equation*}
μ(α_k, …, ν^1 (α_1)) + μ^2 (ν(α_k, …, α_2), α_1) + ν(…, μ(α_t, …, α_1)).
\end{equation*}
(b1a) has the same cancellation as (b1). 


Regard case (c). Then the angles $ α_1, …, α_k $ already conclude a disk and $ α_2, …, α_k $ are short enough to produce an end-split $ ν $ and we reach the cancellation
\begin{equation*}
μ(α_k, …, ν^1 (α_1)) + ν^1 (μ(α_k, …, α_1)) + μ^2 (ν(α_k, …, α_2), α_1).
\end{equation*}
\end{proof}

\begin{figure}
\centering
\begin{subfigure}{0.25\linewidth}
\centering
\begin{tikzpicture}[scale=1.5]
\path[use as bounding box] (-0.5, -0.2) -- (1.7, 1.7);
\path[draw] (0, 0) -- ++(120:1) coordinate[pos=0.3] (2-end) coordinate[pos=0.7] (3-start) -- ++(60:1) coordinate[pos=0.3] (3-end) coordinate[pos=0.7] (4-start) -- ++(right:1.2) coordinate[pos=0.3] (4-end) coordinate[pos=0.7] (5-start) -- ++(300:1) coordinate[pos=0.3] (5-end) coordinate[pos=0.7] (6-start) -- ++(240:1) coordinate[pos=0.3] (6-end) coordinate[pos=0.7] (1-start) coordinate (q) -- ++(left:1.2) coordinate[pos=0.4] (1-end) coordinate[pos=0.7] (2-start);
\foreach \i in {2, 3, 4, 5, 6} \path[draw, -{To[scale=0.75]}, bend right=60] (\i-start) to (\i-end);
\foreach \i in {2, 3, 4, 5, 6} \path ($ (\i-start)!0.5!(\i-end) $) node {\small $ \i $};
\path[draw, -{To[scale=0.75]}] ($ (q) + (0.5, 0) $) arc(0:180:0.5 and 0.4);
\path[draw] ($ (q) + (0.5, 0) $) arc(0:-180:0.45 and 0.35) arc(180:0:0.4 and 0.3) arc(0:-180:0.35 and 0.25) arc(180:0:0.3 and 0.2);
\path (q) ++(0.1, -0.1) node {\small $ 1 $};
\path[fill] (q) circle[radius=0.05];
\path[bend left, decorate, decoration={text along path, text align=center, text={|\tiny|parking garage}}] (0.4, 0.2) to (1.4, 1);
\end{tikzpicture}
\caption*{Case (a)}
\end{subfigure}
\begin{subfigure}{0.25\linewidth}
\centering
\begin{tikzpicture}[scale=1.5]
\path[use as bounding box] (-0.5, -0.2) -- (1.7, 1.7);
\path[draw] (0, 0) -- ++(120:1) coordinate[pos=0.3] (2-end) coordinate[pos=0.7] (3-start) -- ++(60:1) coordinate[pos=0.3] (3-end) coordinate[pos=0.7] (4-start) -- ++(right:1.2) coordinate[pos=0.3] (4-end) coordinate[pos=0.7] (5-start) -- ++(300:1) coordinate[pos=0.3] (5-end) coordinate[pos=0.7] (6-start) -- ++(240:1) coordinate[pos=0.3] (6-end) coordinate[pos=0.7] (1-start) coordinate (q) -- ++(left:1.2) coordinate[pos=0.4] (1-end) coordinate[pos=0.7] (2-start);
\foreach \i in {2, 3, 4, 5, 6} \path[draw, -{To[scale=0.75]}, bend right=60] (\i-start) to (\i-end);
\foreach \i in {2, 3, 4, 5, 6} \path ($ (\i-start)!0.5!(\i-end) $) node {\small $ \i $};
\path[draw, -{To[scale=0.75]}, bend right=80, looseness=1.5] ($ (q) + (0.5, 0) $) to node[pos=0.6, below] {\small $ 1 $} node[pos=0.3, below] {\small $ β $} (1-end);
\path[fill] (q) circle[radius=0.05];
\path[draw, very thick, gray] (q) ++(0:0.1) arc(0:170:0.1);
\path[draw, very thick, gray] (q) ++(70:0.15) arc(70:175:0.15);
\path[draw, very thick, gray] (q) ++(0:0.15) arc(0:50:0.15);
\end{tikzpicture}
\caption*{Case (b)}
\end{subfigure}
\begin{subfigure}{0.25\linewidth}
\centering
\begin{tikzpicture}[scale=1.5]
\path[use as bounding box] (-0.5, -0.2) -- (1.7, 1.7);
\path[draw] (0, 0) -- ++(120:1) coordinate[pos=0.3] (2-end) coordinate[pos=0.7] (3-start) -- ++(60:1) coordinate[pos=0.3] (3-end) coordinate[pos=0.7] (4-start) -- ++(right:1.2) coordinate (4) coordinate[pos=0.3] (4-end) coordinate[pos=0.7] (5-start) -- ++(300:1) coordinate[pos=0.3] (5-end) coordinate[pos=0.7] (6-start) -- ++(240:1) coordinate[pos=0.3] (6-end) coordinate[pos=0.7] (1-start) coordinate (q) -- ++(left:1.2) coordinate[pos=0.4] (1-end) coordinate[pos=0.7] (2-start);
\foreach \i in {2, 3, 4, 5, 6} \path[draw, -{To[scale=0.75]}, bend right=60] (\i-start) to (\i-end);
\foreach \i in {2, 3, 4, 5, 6} \path ($ (\i-start)!0.5!(\i-end) $) node {\small $ \i $};
\path[draw, dashed, gray] (q) -- (4) coordinate[pos=0.3] (1-start);
\path[draw, -{To[scale=0.75]}, bend right=50] (1-start) to node[pos=0.2, below] {\small $ 1 $} (1-end);
\path[fill] (q) circle[radius=0.05];
\path[draw, very thick, gray] (q) ++(170:0.1) arc(170:70:0.1);
\path[draw, very thick, gray] (q) ++(175:0.15) arc(175:100:0.15);
\path[draw, very thick, gray] (q) ++(80:0.15) arc(80:65:0.15);
\end{tikzpicture}
\caption*{Case (c)}
\end{subfigure}
\caption*{Cancellation for \autoref{th:hochschild-construction-mu-nu1-outsource}}
\label{fig:hochschild-construction-mu-nu1-outsource}
\end{figure}

\begin{lemma}
\label{th:hochschild-construction-mu-nu1}
All contributions $ μ^{≥3} (…, ν^1 (α_m), …) $ cancel.
\end{lemma}

\begin{proof}
In \autoref{th:hochschild-construction-mu-nu1-outsource}, we already dealt with the case of $ ν^1 $ being the first-out angle for $ μ $. The case it is the final-out angle is similar. We are left with the checking the case it is an ordinary angle, that is, neither first-out nor final-out. Label the angles as
\begin{equation*}
μ(α_k, …, ν^1 (α_m), …, α_1 β) \quad \text{or} \quad μ(γ α_k, …, ν^1 (α_m), …, α_1) \quad \text{or} \quad μ(α_k, …, ν^1 (α_m), …, α_1),
\end{equation*}
depending on whether $ μ $ is first-out, final-out or all-in. This means $ α_1, …, ν^1 (α_m), …, α_k $ is a disk sequence. Since $ ν^1 (α_m) $ has length bigger than $ ℓ^r $, the remaining angles turn more than $ r $ turns around $ m $ and we conclude $ α_1, …, α_m, …, α_k $ is a garage sequence. Finally the original sequence, which may have additional $ β $ and $ γ $, is an (all-in, first-out or final-out) garage sequence.
\end{proof}

\begin{lemma}
\label{th:hochschild-construction-mu2-nu-mid-outsource-1}
Any contribution $ μ^2 (ν^{≥2} (α_k, …, α_2), α_1) $ with $ ν $ middle-split final-out or all-in cancels. We assume $ α_1 $ is at the same side as $ α_k $, which is always the case if $ ν $ is final-out.
\end{lemma}

\begin{proof}
Distinguish cases: (a) The angle $ α_1 $ is shorter than the corresponding interior angle of $ ν $ and starts after $ m $. (b) The angle $ α_1 $ is shorter and starts at $ m $. (c) The angle $ α_1 $ is shorter and starts before $ m $. (d) The angle $ α_1 $ is longer. (e) The split is at $ (α_{k-1}, α_k) $ and $ α_1 $ is at least as long as the interior angle after the split.

Regard case (a). Then we have a cancellation
\begin{equation*}
μ^2 (ν^{≥2} (α_k, …, α_2), α_1) + ν(…, μ(α_t, …, α_1)).
\end{equation*}
It is readily checked that the two magic angles agree.

Regard case (b). Then we have a cancellation
\begin{equation*}
μ^2 (ν^{≥2} (α_k, …, α_2), α_1) + μ(…, ν(α_s, …, α_1)) + ν(…, μ^2 (α_{s+1}, α_s), …).
\end{equation*}

Regard case (c). Then we have a cancellation
\begin{equation*}
μ^2 (ν^{≥2} (α_k, …, α_2), α_1) + μ(…, ν(α_t, …, α_1)) + ν(…, μ^2 (α_{s+1}, α_s), …).
\end{equation*}
Note in case there is no arc reaching to $ m $ within $ α_1 $, then the last term vanishes, but the first two are then already equal.

Regard case (d). Then we have a cancellation
\begin{equation*}
μ^2 (ν (α_k, …, α_2), α_1) + ν(…, μ^2 (α_{s+1}, α_s), …) + μ^2 (α_k, ν(α_{k-1}, …, α_1)).
\end{equation*}

Regard case (e). Then we have a similar cancellation to (d), namely
\begin{equation*}
μ^2 (ν (α_k, …, α_2), α_1) + ν(μ^2 (α_k, α_{k-1}), …) + μ^2 (α_k, ν(α_{k-1}, …, α_1)).
\end{equation*}
\end{proof}

\begin{lemma}
\label{th:hochschild-construction-mu2-nu-outsource-2}
Any contribution $ μ^2 (ν^{≥2} (α_k, …, α_2), α_1) $ with $ ν $ first-out middle split, or end-split with turning angle $ < ℓ^r $, or end-split with turning angle $ ℓ^r $ cancels.
\end{lemma}

\begin{proof}
Distinguish cases: (a) $ ν $ is first-out middle-split. (b) $ ν $ is end-split $ ≤ ℓ^r $, and $ α_1 $ winds around $ m $, and $ α_1 $ is shorter than the turning angle of $ ν $. (c) $ ν $ is end-split, and $ α_1 $ winds around $ m $, and $ α_1 $ is at least as long as the turning angle of $ ν $. (d) $ ν $ is end-split $ ℓ^r $ and all-in, and $ α_1 $ does not wind around $ m $.

Regard case (a). Then $ α_1 $ can be appended to $ α_2 $, still forming a middle-split $ ν $. Therefore we have a simple cancellation
\begin{equation*}
μ^2 (ν^{≥2}, …) + ν^{≥2} (…, μ^2 (α_2, α_1)).
\end{equation*}
In case (b), we have a cancellation
\begin{equation*}
μ^2 (ν^{≥2} (α_k, …, α_2), α_1) + μ(α_k, …, α_2, ν^1 (α_1)) + ν(α_k, …, μ(α_t, …, α_1)).
\end{equation*}
In case (c), we have the cancellation
\begin{equation*}
μ^2 (ν^{≥2} (α_k, …, α_2), α_1) + μ(α_k, …, ν^1 (α_1)) + ν^1 (μ(α_k, …, α_1)).
\end{equation*}
Regard case (d). Then $ α_1 $ is composable with $ α_2 $ and we have the cancellation
\begin{equation*}
μ^2 (ν^{≥2} (α_k, …, α_2), α_1) + ν^{≥2} (α_k, …, μ^2 (α_2, α_1)).
\end{equation*}
Since this only changes the first angle of the $ ν $, the two terms are either both old-era or both new-era, and hence have the same magic angle and produce an equal result.
\end{proof}

\begin{lemma}
\label{th:hochschild-construction-mu2-x-nu2-outsource-1}
A contribution $ μ^2 (α_k, ν^{≥2} (α_{k-1}, …)) $ with $ ν $ first-out or all-in middle-split cancels.
\end{lemma}

\begin{proof}
Distinguish cases: (a) The angle $ α_k $ is shorter than the corresponding interior angle of $ ν $, and stops before $ m $. (b) The angle $ α_k $ is shorter and stops at $ m $. (c) The angle $ α_k $ is shorter and stops after $ m $. (d) The angle $ α_k $ is longer.

Regard case (a). Then the target arc of $ α_k $ hits the opposite side of $ ν $ at some angle $ α_t $. We have the cancellation
\begin{equation*}
μ^2 (α_k, ν^{≥2} (α_{k-1}, …)) + ν(μ(α_k, …, α_t), …, α_1).
\end{equation*}
Regard case (b). Then we have the cancellation
\begin{equation*}
μ^2 (α_k, ν^{≥2} (α_{k-1}, …)) + μ(ν(α_k, …, α_{s+1}), …, α_1) + ν(…, μ^2 (α_{s+1}, α_s), …).
\end{equation*}
Regard case (c). Then the target arc of $ α_k $ hits the opposite side of $ ν $ at some angle $ α_t $. Then we have the cancellation
\begin{equation*}
μ^2 (α_k, ν^{≥2} (α_{k-1}, …)) + ν(…, μ^2 (α_{s+1}, α_s), …) + μ(ν(α_k, …, α_t), …, α_1).
\end{equation*}
Regard case (d). Then we have the cancellation
\begin{equation*}
μ^2 (α_k, ν^{≥2} (α_{k-1}, …)) + μ^2 (ν(α_k, …, α_2), α_1) + ν(…, μ^2 (α_{s+1}, α_s), …).
\end{equation*}
\end{proof}

\begin{lemma}
\label{th:hochschild-construction-mu2-x-nu2-outsource-2}
A contribution $ μ^2 (α_k, ν^{≥2} (α_{k-1}, …)) $, with $ ν $ end-split with turning angle $ ℓ^r $ cancels.
\end{lemma}

\begin{proof}
Distinguish cases: (a) $ ν $ is all-in and $ α_k $ starts at the opposite side of the split angle. (b) $ ν $ is final-out. (c) $ ν $ is first-out and not final-out, and $ α_k $ is shorter than angle in $ ν $ after the split. (d) $ ν $ is first-out and not final-out, and $ α_k $ is longer.

Regard case (a). Distinguish cases: (a1) $ α_k < ℓ^r $. (a2) $ α_k ≥ ℓ^r $. In case (a1), the target of $ α_k $ touches the opposite side of the $ ν $ orbigon at some angle $ α_t $ and we have the cancellation
\begin{equation*}
μ^2 (α_k, ν(α_{k-1}, …, α_1)) + μ(ν^1 (α_k), …, α_1) + ν(μ(α_k, …, α_t), …, α_1).
\end{equation*}
In case (a2), we have the cancellation
\begin{equation*}
μ^2 (α_k, ν(α_{k-1}, …, α_1)) + μ(ν^1 (α_k), …, α_1) + ν^1 (μ(α_k, …, α_1)).
\end{equation*}

Regard case (b). Then $ α_k $ is composable with $ α_{k-1} $ and we have a simple cancellation
\begin{equation*}
μ^2 (α_k, ν^{≥2} (α_{k-1}, …)) + ν(μ^2 (α_k, α_{k-1}), …, α_1).
\end{equation*}
Here both $ ν $ are new-era. 

Regard case (c). The target arc of $ α_k $ hits the opposite side of $ ν $ at some angle $ α_t $ and we have the cancellation
\begin{equation*}
μ^2 (α_k, ν^{≥2} (α_{k-1}, …)) + μ(ν(α_k, …, α_t), …, α_1) + ν(μ^2 (α_k, α_{k-1}), …, α_1).
\end{equation*}
Regard case (d). We have the cancellation
\begin{equation*}
μ^2 (α_k, ν^{≥2} (α_{k-1}, …)) + ν(μ^2 (α_k, α_{k-1}), …, α_1) + μ^2 (ν(α_k, …, α_2), α_1).
\end{equation*}
\end{proof}

\begin{lemma}
\label{th:hochschild-construction-mu2-x-nu2-outsource-3}
A contribution $ μ^2 (α_k, ν^{≥2} (α_{k-1}, …)) $, with $ ν $ final-out middle-split, or end-split with turning angle $ < ℓ^r $ cancels.
\end{lemma}

\begin{proof}
Distinguish cases: (a) $ ν $ is final-out middle-split. (b) $ ν $ is end-split with turning angle $ < ℓ^r $.

Regard case (a). Then we have a simple cancellation
\begin{equation*}
μ^2 (α_k, ν^{≥2} (α_{k-1}, …)) + ν(μ^2 (α_k, α_{k-1}), …, α_1).
\end{equation*}
Regard case (b). Distinguish cases: (b1) $ α_k $ is smaller than the turning angle of $ ν $. (b2) $ α_k $ is at least the turning angle of $ ν $. In case (b1), the target of $ α_k $ touches the opposite side of the $ ν $ orbigon at some angle $ α_t $ and we have the cancellation
\begin{equation*}
μ^2 (α_k, ν(α_{k-1}, …, α_1)) + μ(ν^1 (α_k), …, α_1) + ν(μ(α_k, …, α_t), …, α_1).
\end{equation*}
In case (b2), we have the cancellation
\begin{equation*}
μ^2 (α_k, ν(α_{k-1}, …, α_1)) + μ(ν^1 (α_k), …, α_1) + ν^1 (μ(α_k, …, α_1)).
\end{equation*}
\end{proof}

\begin{lemma}
\label{th:hochschild-construction-nu-mu2}
Any contribution $ ν^{≥2} (…, μ^2 (…), …) $ cancels.
\end{lemma}

\begin{proof}
Write this sequence as $ ν^{≥2} (…, μ^2 (α_{m+1}, α_m), …) $. Distinguish cases: (a) The product $ α_{m+1} α_m $ is an ordinary angle of $ ν $, and $ α_m $ splits the $ ν $ orbigon into two. (b) The product $ α_{m+1} α_m $ is an ordinary angle of $ ν $, and $ α_m $ and $ α_{m+1} $ create another split of the orbigon. (c) $ α_{m+1} α_m $ is a first-out or final-out angle of end-split $ ν $. (d) $ α_{m+1} α_m $ is a final-out angle of end-split $ ν $.

In case (a), we can produce a $ ν^{≥2} (…, μ^{≥3} (…), …) $. In case (b), we can generally produce a $ μ^{≥3} (…, ν^{≥2} (α_m, …, α_{s+1}), …) $. In both cases, we fall already into the regime of the previous lemmas.

Regard case (c). Then again depending on the configuration of $ α_m $ and $ α_{m+1} $ we can split off a disk sequence from $ ν $ and land in one of the cases already dealt with.
\end{proof}

\begin{proposition}
\label{th:even-other-conclusion}
The cochain $ ν ∈ \HC(\Gtl \cA) $ is a Hochschild cocycle: $ ν ∈ \Ker(d) $.
\end{proposition}

\begin{proof}
We have analyzed all terms $ μ(…, ν(…), …) $ and $ ν(…, μ(…), …) $. For a given sequence, the previous lemmas show that its set of contributions can be partitioned so that the contributions in each partition cancel out together. The only terms we have not checked are $ ν^1 (…, μ^2 (…), …) $ and $ μ^2 (…, ν^1 (…), …) $. Their cancellation follows precisely from $ ν^1 $ being a derivation.
\end{proof}

\subsection{Summary}
\label{sec:even-summary}
In this section, we summarize our findings. Our starting point is the knowledge that the cochain defined in \autoref{def:even-construction-def} is indeed a Hochschild cocycle. We construct the ordinary even Hochschild cochains $ ν_{m, r}^{\even} $ and show that together with the sporadic classes they form a basis for $ \HH^{\even} (\Gtl \cA) $. Finally, we calculate the Gerstenhaber bracket and the cup product in analogy to \paperone\ and conclude that the classification theorem of \paperone\ still holds.

The idea for the even Hochschild cocycle $ ν_{m, r}^{\even} $ is to choose input scalars $ \#ν^1 (α) $ for all indecomposable angles around $ m $ such that their sum is $ 1 $. The Hochschild cocycle $ ν_{m, r}^{\even} $ is then defined as the cocycle we constructed from this data in \autoref{sec:even-construction}:

\begin{definition}
Let $ \cA $ be a full arc system with [NMD]. Let $ m ∈ M $ and $ r ≥ 1 $. Choose any collection of input scalars $ \# ν^1 $ such that their sum is $ 1 $. The \emph{ordinary even Hochschild cocycle} $ ν_{m, r}^{\even} $ is the cocycle $ ν $ constructed from $ \# ν^1 $ by \autoref{def:even-construction-def}.
\end{definition}

This way, the cocyle $ ν_{m, r}^{\even} $ is not canonical. However, different choices $ \# ν^1, \# ν'{}^1 $ yield gauge equivalent cocyles $ ν, ν' $ in the sense that $ ν - ν' ∈ \Ker(d) $:

\begin{lemma}
\label{th:even-summary-choice}
Different choices of $ \# ν^1 $ yield gauge equivalent cocyles $ ν_{m, r}^{\even} $.
\end{lemma}

\begin{proof}
The clue is to regard the cochain $ ε = ε^0 $ given by $ r $ full turns around $ m $, starting from an arbitrary arc incidence at $ m $. The 1-adic component $ (dε)^1 $ then reads $ (dε)^1 (α) = α ε^0 - ε^0 α $. To apply this to $ ν $ and $ ν' $, note that the difference $ \# ν^1 - \# ν'{}^1 $ sums up to zero around $ m $, therefore the component $ ν^1 $ can be written as a sum of cochains of the type $ (dε)^1 $. We conclude that $ ν - ν' $ can be gauged so that its 0-adic and 1-adic components vanish. By \autoref{th:existing-crit-gauging}, $ ν - ν' $ can be gauged to zero entirely. This finishes the proof.
\end{proof}

Together with the odd and the sporadic even classes, the ordinary even Hochschild classes form a basis for $ \HH (\Gtl \cA) $:

\begin{theorem}
Let $ (S, M) $ be a punctured surface and $ \cA $ a full arc system with [NMD]. Then the odd classes $ ν_{\id} $, $ ν_{m, r}^{\odd} $, the sporadic even classes $ \{ν_P\}_{P ∈ \sporadic} $ and ordinary even classes $ ν_{m, r}^{\even} $ for a basis for $ \HH(\Gtl \cA) $.
\end{theorem}

\begin{proof}
Jointly, the classes satisfy the requirements of the generation criterion \autoref{th:existing-generation-odd} and \ref{th:existing-generation-even}. The statement is then immediate.
\end{proof}

\begin{theorem}
Let $ (S, M) $ be a punctured surface and $ \cA $ a full arc system with [NMD]. Let $ m_1 ≠ m_2 $ be two distinct punctures in $ M $, let $ i, j ≥ 0 $ be two indices, and $ ν_P, ν_Q $ two sporadic classes. Then the Gerstenhaber bracket in cohomology reads as follows:
\begin{align*}
[ν_{m_1, i}^{\odd}, ν_{m_2, j}^{\odd}] &= 0, \\
[ν_{m_1, i}^{\even}, ν_{m_2, j}^{\odd}] &= δ_{m_1, m_2} j · ν_{m_1, i+j}^{\odd}, \\
[ν_{m_1, i}^{\even}, ν_{m_2, j}^{\even}] &= δ_{m_1, m_2} (j-i) · ν_{m_1, i+j}^{\even}, \\
[ν_{m_1, i}^{\even}, ν_P] &= -i \#ν_P (ℓ_q) ·  ν_{m_1, i}^{\even}, \\
[ν_P, ν_{m_1, i}^{\odd} ] &= i \#ν_P (ℓ_q) · ν_{m_1, i}^{\odd}, \\
[ν_P, ν_Q] &= 0.
\end{align*}
\end{theorem}

\begin{proof}
Recall that a bracket $ [ν, η] $ on cohomology level is defined as the projection to cohomology $ π[ν, η] $ of the bracket on chain level. This makes for the following strategy: We compute just enough of the bracket at chain level in order to deduce its projection to cohomology. In fact, by \autoref{th:existing-crit-gauging} odd Hochschild classes $ ν $ are already determined by their 0-adic component $ ν^0 $ and even Hochschild classes are already determined by their 1-adic component $ ν^1 $. It therefore suffices to compute the 0-adic component $ [ν, η]^0 $ on chain level in case the bracket value is odd, and the 1-adic component $ [ν, η]^1 $ in case the bracket value is even. To compute these 0-adic and 1-adic components, we typically only need to know the 0-adic and 1-adic component of $ ν $ and $ η $:
\begin{align*}
[ν, η]^0 &= ν^1 (η^0) - (-1)^{‖ν‖ ‖η‖} η^1 (ν^0), \\
[ν, η]^1 (α) &= ν^1 (η^1 (α)) + (-1)^{‖η‖ ‖α‖} ν^2 (η^0, α) + ν^2 (α, η^0) \\
&~ - (-1)^{‖ν‖ ‖η‖} \big(η^1 (ν^1 (α)) + (-1)^{‖ν‖ ‖α‖} η^2 (ν^0, α) + η^2 (α, ν^0)\big).
\end{align*}
We are now ready to check the claimed identities. For the first identity, we have
\begin{equation*}
[ν_{m_1, i}^{\odd}, ν_{m_2, j}^{\odd}]^1 (α) = 0.
\end{equation*}
Indeed, the classes $ ν_{m_1, i}^{\odd} $ and $ ν_{m_2, j}^{\odd} $ have no 1- and 2-adic components. For the second identity, regard
\begin{equation*}
[ν_{m_1, i}^{\even}, ν_{m_2, j}^{\odd}]^0 = (ν_{m_1, i}^{\even})^1 ((ν_{m_2, j}^{\odd})^0) = δ_{m_1, m_2} j ℓ_{m_1}^{i+j}.
\end{equation*}
Indeed, the 0-adic component of the cocycle $ ν_{m_2, j}^{\odd} $ consists of $ j $ full turns around $ m_2 $, starting at each arc incident at $ m_2 $. The derivation $ (ν_{m_1, i}^{\even})^1 $ multiplies each angle around $ m_1 $ by a certain scalar, in fact such that it multiplies one full turn by precisely $ 1 $. Given it is a derivation, it multiplies $ j $ full turns around $ q $ precisely by $ j $. It however sends turns around $ m_2 ≠ m_1 $ to zero. For the third identity, we have
\begin{align*}
[ν_{m_1, i}^{\even}, ν_{m_2, j}^{\even}]^1 (α) &= (ν_{m_1, i}^{\even})^1 ((ν_{m_2, j}^{\even})^1 (α)) - (ν_{m_2, j}^{\even})^1 ((ν_{m_1, i}^{\even})^1 (α)) \\
&= δ_{m_1, m_2} \big(\# ν_{m_2, j}^1 (α) (\# ν_{m_1, i}^1 (α) + j) ℓ_{m_1}^{i+j} α - \# ν_{m_1, i}^1 (α) (\# ν_{m_2, j}^1 (α) + i) ℓ_{m_1}^{i+j}\big) \\
&= δ_{m_1, m_2} (j-i) ℓ_{m_1}^{i+j} α.
\end{align*}
Here we have used that $ \# ν^1_{m_1, i} (α) = \# ν^1_{m_1, j} (α) $. In other words, we have assumed that the ordinary even cocycles for different $ i ≠ j $ but equal puncture have been constructed with the same input scalars, which is legitimate by \autoref{th:even-summary-choice}. Finally, this bracket $ [ν_{m_1, i}^{\even}, ν_{m_2, j}^{\even}] $ has the same 1-adic component as the cohomology class $ δ_{m_1, m_2} (j-i) ν_{m_1, i+j}^{\even} $ and hence projects to it. For the fourth identity, regard
\begin{align*}
[ν_{m_1, i}^{\even}, ν_P]^1 (α) &= ((ν_{m_1, i}^{\even})^1 (ν_P^1 (α)) - ν_P^1 ((ν_{m_1, i}^{\even})^1 (α)) \\
&= \# ν_{m_1, i}^1 (α) \#ν_P (α) ℓ_{m_1}^i α - \# ν_{m_1, i}^1 (α) (\#ν_P (α) + i \#ν_P (ℓ_{m_1})) ℓ_{m_1}^i α.
\end{align*}
We conclude that this bracket has precisely the same 1-adic component as $ - i \#ν_P (ℓ_{m_1}) ν_{m_1, i}^{\even} $ and hence projects to it. For the fifth identity, regard
\begin{equation*}
[ν_P, ν_{m_1, i}^{\odd}]^0 = ν_P^1 (ℓ_{m_1}^i) = i \#ν_P (ℓ_{m_1}) ℓ_{m_1}^i.
\end{equation*}
We conclude that the bracket has the same 0-adic component as $ i \#ν_P (ℓ_{m_1}) ν_{m_1, i}^{\odd} $ and hence projects to it. For the sixth identity, regard
\begin{equation*}
[ν_P, ν_Q]^1 (α) = ν_P^1 (ν_Q^1 (α)) - ν_Q^1 (ν_P^1 (α)) = \#ν_P (α) \#ν_Q (α) α - \#ν_Q (α) \#ν_P (α) α = 0.
\end{equation*}
This means the 1-adic component of $ [ν_P, ν_Q] $ vanishes, this bracket is therefore gauge equivalent to zero and its projection to cohomology vanishes.
\end{proof}

\begin{theorem}
Let $ (S, M) $ be a punctured surface and $ \cA $ a full arc system with [NMD]. Let $ m_1 ≠ m_2 $ be two distinct punctures and let $ i, j ≥ 1 $ be two indices, and $ ν_P, ν_Q $ two sporadic classes. Then the cup product in cohomology reads as follows:
\begin{alignat*}{2}
ν_{m_1, i}^{\odd} ~&∪~ ν_{m_2, j}^{\odd} &&= \quad δ_{m_1, m_2} ν_{m_1, i+j}^{\odd}, \\
ν_{m_1, i}^{\odd} ~&∪~ ν_{m_2, j}^{\even} &\quad&= \quad δ_{m_1, m_2} ν_{m_1, i+j}^{\even}, \\
ν_{m_1, i}^{\odd} ~&∪~ ν_P &&= \quad \#ν_P (ℓ_{m_1}) ν_{m_1, i}^{\even}, \\
ν_{m_1, i}^{\even} ~&∪~ ν_{m_2, j}^{\even} &&= \quad 0, \\
ν_{m_1, i}^{\even} ~&∪~ ν_P &&= \quad 0, \\
ν_P ~&∪~ ν_Q &&= \quad 0.
\end{alignat*}
The odd class $ ν_{\id} $ acts as identity element: $ ν_{\id} ∪ κ = κ ∪ ν_{\id} = κ $.
\end{theorem}

\begin{proof}
We compute the cup products of the given Hochschild cocycles first on chain level. Then we compute their projection to cohomology. In fact, for the odd products it suffices to compute the 0-adic component $ (ν ∪ η)^0 $ and for the even products it suffices to compute the 1-adic component $ (ν ∪ η)^1 $. We are now ready to start the calculations. For the first identity, regard
\begin{equation*}
(ν_{m_1, i}^{\odd} ∪ ν_{m_2, j}^{\odd})^0 = δ_{m_1, m_2} μ^2 (ℓ_{m_1}^i, ℓ_{m_1}^j) = ℓ_{m_1}^{i+j}.
\end{equation*}
This is precisely the 0-adic component of the Hochschild cohomology class $ ν_{m_1, i+j}^{\odd} $ and hence projects to it. For the second identity, regard
\begin{align*}
(ν_{m_1, i}^{\odd} ∪ ν_{m_2, j}^{\even})^1 (α) &= (-1)^{‖α‖ + 1} μ^2 (ℓ_{m_1}^i, ν_{m_2, j}^1 (α)) \\
&= δ_{m_1, m_2} (-1)^{‖α‖ + 1 + |α|} ℓ_{m_1}^i \# ν_{m_2, j}^1 (α) ℓ_{m_1}^j α = δ_{m_1, m_2} \#ν(α) ℓ_{m_1}^{i+j} α.
\end{align*}
If $ m_1 = m_2 $, then this is precisely equal to $ (ν_{m_1, i+j}^{\even})^1 (α) $, hence the cup product on chain level projects to $ ν_{m_1, i+j}^{\even} $. For the third identity, regard
\begin{equation*}
(ν_{m_1, i}^{\odd} ∪ ν_P)^1 (α) = (-1)^{‖α‖ + 1} μ^2 (ℓ_{m_1}^i, ν_P^1 (α)) = (-1)^{‖α‖ + 1} μ^2 (ℓ_{m_1}^i, \#ν_P(α) α) = \#ν_P (α) ℓ_{m_1}^i α.
\end{equation*}
This product $ ν_{m_1, i}^{\odd} ∪ ν_P $ sends every angle winding around a puncture different from $ m_1 $ to zero. Since we can add commutators with arbitrary turns around $ m_1 $ as gauges, the projection of $ ν_{m_1, i}^{\odd} ∪ ν_P $ to cohomology can be read off from the sum of the coefficients $ \#ν_P (α) $ over all indecomposable angles $ α $ around $ m_1 $. In total, the product projects to
\begin{equation*}
\# ν_P (ℓ_{m_1}) ν_{m_1, i}^{\even}.
\end{equation*}
To discuss the fourth, fifth and sixth identity, let us write $ ν $ and $ η $ for the two factors. Both $ ν $ and $ η $ are even. This means they have vanishing 0-adic component $ ν^0 = η^0 = 0 $. In particular, their cup product's 0-adic component
\begin{equation*}
(ν ∪ η)^0 = μ^2 (ν^0, η^0)
\end{equation*}
vanishes. This shows that $ ν ∪ η = 0 $ in these three cases. To see that the odd class $ ν_{\id} $ acts as identity element, note that
\begin{align*}
(κ ∪ ν_{\id}) (α_k, …, α_1) &= (-1)^{‖ν_{\id}‖ + 1} μ^2 (κ(α_k, …, α_1), \id) = κ(α_k, …, α_1), \\
(ν_{\id} ∪ κ) (α_k, …, α_1) &= (-1)^{‖ν_{\id}‖ (‖α_1‖ + … + ‖α_k‖) + ‖κ‖ + 1} μ^2 (\id, κ(α_k, …, α_1)) = κ(α_k, …, α_1).
\end{align*}
This finishes the proof.
\end{proof}

\begin{remark}
In \paperone, we also proved a formality theorem for $ \HH(\Gtl \cA) $ and a classification theorem for formal deformations of $ \Gtl \cA $. Our proof of the formality theorem builds on a topological grading for $ \Gtl \cA $. Without the [NL2] condition, the power of the topological grading collapses. The classification theorem however stays intact since it does not make explicit reference to the even Hochschild cocycles.
\end{remark}

\printbibliography

@article{Bocklandt,
author={Bocklandt, Raf},
title={Noncommutative mirror symmetry for punctured surfaces},
note={With an appendix by Mohammed Abouzaid},
journal={Trans. Amer. Math. Soc.},
volume={368},
year={2016},
number={1},
pages={429--469}}

@inproceedings {Seidel-relative,
    AUTHOR = {Seidel, Paul},
     TITLE = {Fukaya categories and deformations},
 BOOKTITLE = {Proceedings of the {I}nternational {C}ongress of
              {M}athematicians, {V}ol. {II} ({B}eijing, 2002)},
     PAGES = {351--360},
 PUBLISHER = {Higher Ed. Press, Beijing},
      YEAR = {2002},
   MRCLASS = {53D40 (14J32 18E30 53D35 57R17 57R56)},
  MRNUMBER = {1957046},
MRREVIEWER = {Richard P. Thomas},
}

@misc{Paper-I,
Author = {Raf Bocklandt and Jasper van de Kreeke},
Title = {Deformations of Gentle $A_\infty$-Algebras},
Year = {2023},
Eprint = {arXiv:2304.10223},
}

@misc{Paper-IIA,
Author = {Jasper van de Kreeke},
Title = {$ A_∞ $-Deformations and their Derived Categories},
Year = {2023},
Eprint = {arXiv:2308.08026},
}

\end{document}